\documentclass[12pt]{article}
\usepackage[T1]{fontenc}
\usepackage{amsmath,amsthm,amssymb}
\usepackage{indentfirst}
\usepackage{mathrsfs}
\usepackage{graphicx}
\usepackage{tikz-cd}
\usepackage{tabularx} 
\usepackage[margin=1in]{geometry} 
\usepackage{dynkin-diagrams}
\usepackage{blindtext}
\usepackage{mathtools}
\usepackage{graphicx} 
\usepackage{hyperref}
\hypersetup{unicode=true,bookmarksnumbered=false,bookmarksopen=false, breaklinks=false,pdfborder={0 0 1},colorlinks=true,linkcolor=blue}
\usepackage{csquotes}
\usepackage{booktabs} 
\usepackage{array} 
\usepackage{paralist} 
\usepackage{subfig} 
\usepackage{amsfonts, bbm, dsfont}
\usepackage{stmaryrd}
\usepackage{float}
\usepackage{enumitem}
\setenumerate[1]{itemsep=0pt,partopsep=0pt,parsep=\parskip,topsep=5pt}
\setitemize[1]{itemsep=0pt,partopsep=0pt,parsep=\parskip,topsep=5pt}
\usepackage{url}
\usepackage{tkz-graph}
\usepackage{eucal}  
\usepackage{multirow}
\usepackage{titlesec}
\usepackage{cleveref}
\allowdisplaybreaks

\crefformat{section}{\S#2#1#3} 
\crefformat{subsection}{\S#2#1#3}
\crefformat{subsubsection}{\S#2#1#3}
\Crefname{paragraph}{\S}{Paragraphs}

\numberwithin{equation}{section}

\usepackage[backend=biber,style=alphabetic]{biblatex}
\NewBibliographyString{diplomathesis}
\DefineBibliographyStrings{english}{diplomathesis = {diploma thesis},}

\pgfkeys{/Dynkin diagram,
edge length=0.8cm,
fold radius=0.8cm,
root radius=0.1cm,
indefinite edge/.style={
draw=black,
fill=white,
thin,
densely dashed}}
\newtheoremstyle{special}
    {\topsep}
    {\topsep}
    {\itshape}
    {}
    {\bfseries}
    {}
    {0.5em}
    {{\thmname{#1$^{\bm*}\!$}\thmnumber{ #2.}\thmnote{\ \textmd{(#3)}}}}

\newtheorem{thm}{Theorem}[subsection]
\newtheorem{cor}[thm]{Corollary}
\newtheorem{lemma}[thm]{Lemma}
\newtheorem{prop}[thm]{Proposition}

\newtheorem{introthm}{Theorem}

\newtheorem{introconj}[introthm]{Conjecture}
\theoremstyle{remark}
\newtheorem{ex}[thm]{Example}

\newtheorem{rmk}[thm]{Remark}

\theoremstyle{definition}
\newtheorem{defi}[thm]{Definition}
\newtheorem{notation}[thm]{Notation}
\theoremstyle{special}

\newcommand{\norm}{\mathrm{N}}
\newcommand{\sym}{\operatorname{Sym}}
\newcommand{\spin}{\mathbf{Spin}}

\newcommand{\G}{\mathscr{G}}
\newcommand{\Z}{\mathbb{Z}}
\newcommand{\C}{\mathbb{C}}

\newcommand{\Q}{\mathbb{Q}}

\newcommand{\oct}{\mathbb{O}}
\newcommand{\coct}{\oct_{\Z}}
\newcommand{\R}{\mathbb{R}}
\newcommand{\GL}{\mathbf{GL}}
\newcommand{\SL}{\mathbf{SL}}
\newcommand{\A}{\mathbb{A}}
\newcommand{\tr}{\operatorname{Tr}}

\newcommand{\sorth}{\mathbf{SO}}
\newcommand{\symp}{\mathbf{Sp}}

\newcommand{\modulo}{\operatorname{mod}}
\newcommand{\pgl}{\mathbf{PGL}}
\newcommand{\cusp}{\mathrm{cusp}}
\newcommand{\disc}{\mathrm{disc}}

\newcommand{\lietype}[2]{\operatorname{#1}_{#2}}
\newcommand{\smallmat}[4]{\left(\begin{smallmatrix}#1&#2\\#3&#4\end{smallmatrix}\right)}
\newcommand{\mat}[4]{\left(\begin{matrix}#1&#2\\#3&#4\end{matrix}\right)}
\newcommand{\lrangle}[2]{\langle #1,#2\rangle}
\newcommand{\bracket}[1]{\left(#1\right)}

\newcommand{\grpF}{\mathbf{F}_{4}}
\newcommand{\grpE}{\mathbf{E}_{7}}
\newcommand{\triv}{\mathbf{1}}
\newcommand{\jord}{\mathrm{J}}
\newcommand{\midline}{\,\middle\vert\,}
\newcommand{\set}[2]{\left\{#1\midline #2\right\}}
\newcommand{\vrep}[1]{\mathrm{V}_{#1}}

\newcommand{\para}{\mathbf{P}}
\newcommand{\LeviM}{\mathbf{M}}
\newcommand{\grpG}{\mathbf{G}}
\newcommand{\sqint}[2]{\mathrm{L}^{2}_{#2}(#1)}

\begin{document}
\title{The global exceptional theta correspondence for the dual pair \texorpdfstring{$\grpF\times\pgl_{2}$}{PDFstring}}
\author{Yi Shan}
\date{\today}
\maketitle
\begin{abstract}
    Let $\grpF$ be the unique (up to isomorphism) connected semisimple algebraic group over $\Q$ of type $\lietype{F}{4}$,
    with compact real points and split over $\Q_{p}$ for all primes $p$.
    A conjectural computation \cite[Proposition 6.3.6]{shan2024levelautomorphicrepresentationsanisotropic} 
    predicts the existence of a family of level one automorphic representations of $\grpF$,
    which are expected to be functorial lifts of cuspidal representations of $\pgl_{2}$ associated with Hecke eigenforms.
    In this paper,
    we study the exceptional theta correspondence for $\grpF\times\pgl_{2}$,
    and establish the existence of such a family of automorphic representations for $\grpF$.
    Moreover, we show that for the automorphic representation on $\pgl_{2}$ associated with any even weight $k\geq 12$, level $N$ cuspidal eigenform,
    its global theta lift to $\grpF$ is non-zero.
    For the other direction, 
    motivated by \cite{pollack2023exceptional},
    we develop a notion of ``exceptional theta series'' on $\pgl_{2}$,
    arising from certain automorphic representations of $\grpF$.
    These theta series are holomorphic modular forms on $\SL_{2}(\Z)$, with explicit Fourier expansions,
    and span the entire space of level one cusp forms.
\end{abstract}
\setcounter{tocdepth}{1}
\tableofcontents

\section{Introduction}\label{section introduction}
Since the last century, automorphic representations
of general linear groups and classical groups have been widely studied.
For those of \emph{exceptional groups},
\emph{i.e.}\,algebraic groups with Lie type $\lietype{G}{2},\lietype{F}{4},\lietype{E}{6},\lietype{E}{7}$ or $\lietype{E}{8}$,
most of the known results are about the smallest exceptional group $\mathbf{G}_{2}$,
either split or anisotropic.
In this paper, we will study a family of automorphic representations
for $\grpF$,
the unique (up to isomorphism) 
connected semisimple algebraic group over $\Q$ of type $\lietype{F}{4}$,
with compact real points and split over $\Q_{p}$ for every prime $p$.
\subsection{Motivation from \texorpdfstring{\cite{shan2024levelautomorphicrepresentationsanisotropic}}{PDFstring}}
\label{section motivation from Arthur conj}
In \cite{shan2024levelautomorphicrepresentationsanisotropic},
we compute the number of \emph{level one} automorphic representations for $\grpF$,
\emph{i.e.}\,unramified at every finite place,
with any given arbitrary archimedean component.
Furthermore, the \emph{discrete global Arthur parameters} of these automorphic representations
are classified \emph{conjecturally},
admitting the existence of the (level one) Langlands group and Arthur's multiplicity formula \cite{ArthurConjUnipAutoRep}.
In particular,
we conjecture the existence of a specific family of automorphic representations for $\grpF$,
which are related to classical modular forms for $\SL_{2}(\Z)$.
Before recalling this statement,
we introduce some notations:
\begin{itemize}
    \item Let $\varpi_{4}$ be the highest weight of the $26$-dimensional irreducible representation of $\grpF(\R)$.
    \item There is a morphism $\symp_{6}(\C)\times\SL_{2}(\C)\rightarrow \widehat{\grpF}(\C)=\grpF(\C)$
    whose kernel is a cyclic group of order $2$,
    the image of this morphism is a maximal proper regular closed subgroup of $\grpF(\C)$ (see \cite[\S 4.3.2]{shan2024levelautomorphicrepresentationsanisotropic}).
    Denote by $\iota$ the morphism:
        \[\SL_{2}(\C)\times\SL_{2}(\C)\xhookrightarrow{\left(\text{principal embedding},\,\mathrm{id}\right)} \symp_{6}(\C)\times\SL_{2}(\C)\rightarrow \grpF(\C).\]
    \item Denote by $e_{p}$ the conjugacy class of $\smallmat{p^{1/2}}{}{}{p^{-1/2}}$ in $\SL_{2}(\C)$.
\end{itemize}
\begin{introconj}\label{conj existence of representations}
    \cite[Proposition 6.3.6]{shan2024levelautomorphicrepresentationsanisotropic}
    Let $\pi$ be the level one algebraic automorphic representation of $\pgl_{2}$
    associated to a cuspidal Hecke eigenform of weight $2n+12$ for $\SL_{2}(\Z)$,
    and $c_{p}$ the Satake parameter of $\pi_{p}$,
    viewed as a semisimple conjugacy class in $\widehat{\pgl_{2}}(\C)=\SL_{2}(\C)$. 
    There exists a level one automorphic representation $\Pi$ of $\grpF$ such that:
    \begin{itemize}
        \item $\Pi_{\infty}\simeq \vrep{n\varpi_{4}}$, the irreducible representation of $\grpF(\R)$ with highest weight $n\varpi_{4}$;
        \item for every prime $p$, the Satake parameter of $\Pi_{p}$ is the conjugacy class of $\iota(e_{p},c_{p})$.
    \end{itemize}
\end{introconj}
Motivated by the Langlands functoriality principle,
the automorphic representation $\Pi$ in \Cref{conj existence of representations}
is expected to be a functorial lift of $\pi$
with respect to the embedding 
\begin{align}\label{eqn embedding functoriality}
    i:\widehat{\pgl_{2}}=\SL_{2}\xhookrightarrow{(1,\mathrm{id})}\symp_{6}\times\SL_{2}\hookrightarrow\widehat{\grpF}.
\end{align}
One useful tool for constructing functorial lifts is the \emph{theta correspondence},
which studies the restriction of a \emph{minimal representation} to reductive dual pairs.
There exists a reductive dual pair $\pgl_{2}\times \grpF$ inside 
certain algebraic group $\grpE$ of Lie type $\lietype{E}{7}$ (see \Cref{section exceptional groups} for more details).
For the theta correspondence associated with this dual pair over a characteristic $0$ local field,
one already has the following results (see also \Cref{section local theta correspondences}):
\begin{itemize}
    \item Over $\R$, Gross and Savin describe the restriction of the minimal representation of $\grpE(\R)$ to $\pgl_{2}(\R)\times\grpF(\R)$ \cite[Proposition 3.2]{gross_savin_1998}, 
    which shows that the theta lift $\Theta(\pi_{\infty})$ of $\pi_{\infty}$ is isomorphic to $\vrep{n\varpi_{4}}$;
    \item Over a $p$-adic field, this theta correspondence is studied by Karasiewicz and Savin 
    in \cites{Savin1994}{karasiewicz2023dualpairmathrmautctimesf4}.
    In particular, they demonstrate that the theta lift $\Theta(\pi_{p})$ of 
    the unramified tempered principal series representation $\pi_{p}$
    is irreducible and has the desired Satake parameter $\iota(e_{p},c_{p})$.    
\end{itemize}
Based on these local results, it is natural to expect that 
the functorial lift $\Pi$ is exactly the global theta lift $\Theta(\pi)$ of $\pi$ to $\grpF$.
The main result in this paper confirms this expectation:
\begin{introthm}\label{introthm main theorem}
    (\Cref{cor nonvanishing global theta lift from pgl2 to f4})
    Let $\pi\in\Pi_{\cusp}(\pgl_{2})$ be the automorphic representation associated to a Hecke eigenform in $\mathrm{S}_{2n+12}(\Gamma_{0}(N))$, $n\geq 0$,
    where $\Gamma_{0}(N)$ is the congruence subgroup $\set{\smallmat{a}{b}{c}{d}\in\SL_{2}(\Z)}{c\equiv 0 \modulo N}$.
    The global theta lift $\Theta(\pi)$ is a non-zero irreducible automorphic representation of $\grpF$,
    and satisfies the local-global compatibility of theta correspondence $\Theta(\pi)\simeq \otimes_{v}^{\prime}\Theta(\pi_{v})$.
    In particular, \Cref{conj existence of representations} holds.
\end{introthm}
\subsection{Strategy towards \texorpdfstring{\Cref{introthm main theorem}}{PDFstring}}
\label{section strategy towards the non-vanishing}
Now we illustrate our strategy for the proof of \Cref{introthm main theorem}.

Let $\pi\simeq\otimes_{v}^{\prime}\pi_{v}$ be the automorphic representation of $\pgl_{2}$ associated to a Hecke eigenform $f\in\mathrm{S}_{2n+12}(\Gamma_{0}(N))$.
We want to show that the global theta lift $\Theta_{\phi}(\varphi)$ of some vector $\varphi\in\pi$, 
with respect to some vector $\phi$ in the \emph{minimal representation} of $\grpE(\A)$,
is non-zero.
For this goal,
we compute the \emph{$\spin_{9}$-period integral} of $\Theta_{\phi}(\varphi)$,
where $\spin_{9}$ is a maximal proper regular closed subgroup of $\grpF$.
The $\spin_{9}$-period of an automorphic form $f$ on $[\grpF]=\grpF(\Q)\backslash\grpF(\A)$ is defined as follows, 
where $dg$ is taken to be the Tamagawa measure: 
\[\mathcal{P}_{\spin_{9}}(f):=\int_{\spin_{9}(\Q)\backslash\spin_{9}(\A)}f(g)dg.\]
\begin{rmk}\label{rmk intro connection with SV conj}
    One motivation for considering this $\spin_{9}$-period is the global conjecture of Sakellaridis-Venkatesh \cite{SVPeriods}.
    The homogeneous $\grpF$-space $\mathbf{X}=\spin_{9}\backslash \grpF$ is a spherical variety
    whose \emph{dual group} is $\mathbf{G}_{\mathbf{X}}^{\vee}=\SL_{2}$,
    equipped with the embedding $i:\mathbf{G}_{\mathbf{X}}^{\vee}\rightarrow\widehat{\grpF}$ as described in \eqref{eqn embedding functoriality}.
    Roughly speaking, 
    the conjecture of Sakellaridis-Venkatesh predicts that
    the cuspidal automorphic representations of $\grpF$ with non-zero $\spin_{9}$-periods
    arise from functorial lifts with respect to the embedding $i:\mathbf{\widehat{\pgl_{2}}}\rightarrow\widehat{\grpF}$.
    Therefore, we expect the global theta lift $\Theta_{\phi}(\varphi)$ to have a non-zero $\spin_{9}$-period (for some suitable choice of $\varphi$ and $\phi$).
\end{rmk}
Using a see-saw duality argument 
and an \emph{exceptional Siegel-Weil formula} that we prove in \Cref{section exceptional siegel weil formula},
we rewrite the $\spin_{9}$-period of $\Theta_{\phi}(\varphi)$
as the Rankin-Selberg integral of $\varphi$ with an Eisenstein series on $\GL_{2}\times\GL_{2}$.
This integral factorizes into an Euler product of local zeta integrals,
and the calculation for these zeta integrals is well-known (for example, see \cite[\S 3.8]{Bump_Automorphic_forms_representations}).
At unramified places $p\nmid N$,
$\varphi_{p}$ and $\phi_{p}$ chosen to be spherical vectors,
the zeta integral  is the product of two local L-factors.
At other places,
by choosing suitable local data, the zeta integral can be equal to $1$.
In consequence, we have the following result,
which verifies the prediction of Sakellaridis-Venkatesh \cites[\S 17]{SVPeriods}[Table 1]{Sakellaridis_Rank1Transfer} for the global period associated with spherical variety $\spin_{9}\backslash\grpF$:
\begin{introthm}\label{introthm SV L-factor for Spin9 F4}
    (\Cref{thm prod local zeta integral L function})
    For any integer $n\geq 0$,
    let $\pi$ be the cuspidal representation of $\pgl_{2}$ associated with a Hecke eigenform in $\mathrm{S}_{2n+12}(\Gamma_{0}(N))$.
    There exists a smooth vector $\phi\in \Pi_{\min}$ and an automorphic form $\varphi\in\pi$
    such that
    \[\mathcal{P}_{\spin_{9}}(\Theta_{\phi}(\varphi))=\mathrm{L}^{S}(\pi,\frac{5}{2})\cdot\mathrm{L}^{S}(\pi,\frac{11}{2}),\]
    where $S$ is the set of primes $p|N$ and $\mathrm{L}^{S}(\pi,s)$ is the partial standard L-function of $\pi$ (as an Euler product over all the finite places $p\notin S$).
\end{introthm}
The $\mathrm{L}$-values in \Cref{introthm SV L-factor for Spin9 F4} are non-zero,
thus the theta lift $\Theta_{\phi}(\varphi)$ is non-zero,
which implies the non-vanishing part of \Cref{introthm main theorem}.
The local-global compatibility in \Cref{introthm main theorem}
follows from the results of Karasiewicz-Savin \cite{karasiewicz2023dualpairmathrmautctimesf4} on the $p$-adic exceptional theta correspondence.

\subsection{Exceptional theta series}\label{section intro exceptional theta series}
For the other direction of the exceptional theta correspondence,
\emph{i.e.\,}from $\grpF$ to $\pgl_{2}$,
we develop a notion of (level one) ``\emph{exceptional theta series}'',
motivated by Pollack's construction of Siegel modular forms for $\symp_{6}(\Z)$.
This is a variant of the classical weighted theta series developed by Jacobi and Hecke,
and gives an explicit theta lift from certain automorphic forms of $\grpF$ to $\pgl_{2}$.
The exceptional theta series is also used to give an alternative proof for the non-vanishing of the archimedean Rankin-Selberg zeta integral appearing in the proof of \Cref{introthm SV L-factor for Spin9 F4} when $n\neq 1$ (see \Cref{section nonvanishing archimedean zeta integral alternative proof}).

\subsubsection{Classical theta series}\label{section classical theta series}
We first recall the classical construction of theta series.
Let $L$ be an even unimodular lattice in the Euclidean space $\R^{n}$,
\emph{i.e.}\,a self-dual lattice for any element $v$ of which the scalar product $v.v$ is even.
Such a lattice exists if and only if $8|n$.
A well-known result states that the series
\[\vartheta_{L}(z)=\sum_{v\in L}q^{\frac{v.v}{2}},\text{ where }q=e^{2\pi i z},\,z\in\mathcal{H}=\set{x+iy}{y>0},\]
is a modular form of level $\SL_{2}(\Z)$ and weight $n/2$.
One can weight this theta series by a homogeneous harmonic polynomial $P$ of degree $d$ over $\R^{n}$ \cite{Hecke_ThetaSeries}:
\begin{align}\label{eqn classical theta series}
    \vartheta_{L,P}(z)=\sum_{v\in L}P(v)q^{\frac{v.v}{2}},
\end{align}
and the resulting weighted theta series is a modular form for $\SL_{2}(\Z)$ of weight $\frac{n}{2}+d$.
It is a cusp form when $d>0$,
and Waldspurger shows in \cite{WaldspurgerBasisProblem} that
for a fixed pair of integers $(n,d)$ such that $8|n$ and $d$ is even, 
the space $\mathrm{S}_{\frac{n}{2}+d}(\SL_{2}(\Z))$ 
of weight $\frac{n}{2}+d$ cusp forms
is spanned by: 
\[\set{\vartheta_{L,P}}{L\subseteq \R^{n} \text{ is an even unimodular lattice, and }P\in\mathscr{H}_{d}(\R^{n})},\]
where $\mathscr{H}_{d}(\R^{n})$ is the space of homogeneous harmonic polynomials of degree $d$ over $\R^{n}$.
\subsubsection{Corresponding structures in the exceptional case}\label{section players in the exceptional case}
We want to produce a family of modular forms analogous to \eqref{eqn classical theta series},
starting from automorphic representations for $\grpF$ with archimedean component $\vrep{n\varpi_{4}}$.
The table below highlights the corresponding structures in the classical and exceptional settings:
\begin{table}[H]
    \centering
    \begin{tabular}{|c|c|c|c|c|}
        \hline
                        & classical case & exceptional case \\ \hline 
        underlying space & Euclidean space $\R^{n}$ & Euclidean Albert $\R$-algebra $\jord_{\R}$ \\ \hline 
        group of automorphisms & $\mathbf{O}_{n}(\R)$ & $\grpF(\R)$ \\ \hline 
        integral structure & even unimodular lattice &  Albert lattice \\ \hline 
        homogeneous polynomials & harmonic polynomials & a polynomial model of $\vrep{n\varpi_{4}}$ \\ \hline 
    \end{tabular}
    \caption{Comparison between classical and exceptional cases}
    \label{table comparison classical and exceptional theta series}
\end{table}
We briefly explain the objects appearing in \Cref{table comparison classical and exceptional theta series},
and the details will be provided in \Cref{section glimpse of Albert algebra} and \Cref{section exceptional group F4}:
\begin{itemize}
    \item The $27$-dimensional \emph{Euclidean Albert $\R$-algbera (or Euclidean exceptional Jordan $\R$-algebra)} $\jord_{\R}=\mathrm{Her}_{3}(\oct_{\R})$
    is the space of ``Hermitian'' 3-by-3 matrices over 
    the \emph{real octonion division algebra} $\oct_{\R}$,
    equipped with the distinguished element $\mathrm{I}=\mathrm{diag}(1,1,1)$,
    the adjoint map $\#:\jord_{\R}\rightarrow\jord_{\R}$,
    and the determinant $\det:\jord_{\R}\rightarrow\R$.
    Precisely,
    together with these structures,
    $\jord_{\R}$ is a \emph{cubic Jordan $\R$-algebra}
    and furthermore it is an \emph{Albert $\R$-algebra}.
    We call it Euclidean because its underlying vector space admits a symmetric inner product $(A,B)=\frac{1}{2}\tr(AB+BA)$ that is positive definite.
    \item  The group of Albert $\R$-algebra automorphisms of $\jord_{\R}$ is the real points $\grpF(\R)$ of $\grpF$,
    \emph{i.e.}\,$\grpF(\R)=\set{g\in\mathrm{GL}(\jord_{\R})}{g\mathrm{I}=\mathrm{I},\,\det(gA)=\det(A),\text{ for any }A\in\jord_{\R}}$.
    \item 
    By an \emph{Albert lattice},
    we mean a lattice $J\subseteq \jord_{\R}$
    satisfying that $\mathrm{I}\in J$, $J$ is stable under $\#$, $\det(J)\subseteq \Z$,
    and $(J,\mathrm{I},\#,\det)$ is an \emph{Albert $\Z$-algebra}.
    
    \item In \Cref{section harmonic polynomials representations of F4},
    we describe a polynomial model $\vrep{n}(\jord_{\C})$ of $\vrep{n\varpi_{4}}$:
    the space spanned by degree $n$ homogeneous polynomials over $\jord_{\R}$ of the form:
    \[X\mapsto \left(X,A\right)^{n},\text{ where }0\neq A\in \jord_{\R}\otimes_{\R}\C,\,A^{2}=0,\,\tr(A)=0.\]    
\end{itemize}
\begin{rmk}
    \label{rmk Euclidean space or its symmetric power}
    In \Cref{table comparison classical and exceptional theta series}, 
    the real analogue of $\jord_{\R}$ in the classical case is $\sym^{2}\R^{n}$.
    However, for convenience, we continue to use $\R^{n}$ instead of $\sym^{2}\R^{n}$.
    Note also that one can view $P\in \mathscr{H}_{2d}(\R^{n})$ (since the weighted theta series for odd degree polynomials vanish) as a degree $d$ polynomial over $\sym^{2}\R^{n}$, 
    then the weight of $\vartheta_{L,P}$ becomes $n/2$ plus \emph{twice} the degree $d$ of $P$,
    which matches the factor $2\deg(P)$ in the weight-degree relation in the exceptional case (see \Cref{introthm theta series is modular form}).
\end{rmk}
\subsubsection{Weighting the theta series constructed by Elkies-Gross}\label{section construct weighted theta series from that of Elkies Gross}
The starting point of the exceptional theta series associated with $\jord_{\R}$
is the work of Elkies and Gross \cite{Leech}.

Let $\mathcal{J}$ be the set of Albert lattices,
and equip it with the natural $\grpF(\R)$-action.
This set is the disjoint union of two $\grpF(\R)$-orbits
\cite[Proposition 5.3]{G96}.
We take a set of representatives $\{J_{1},J_{2}\}$ 
for these two orbits,
where $J_{1}=\jord_{\Z}$ (see \Cref{ex trivial Albert Z algebra}) is taken as the base point of $\mathcal{J}$.
For $J=J_{1}$ or $J_{2}$,
Elkies and Gross construct the following theta series:
\[\vartheta_{J}(z)=1+240\sum_{\substack{J\ni T\geq 0,\\\mathrm{rank}(T)=1}}\sigma_{3}(\mathrm{c}_{J}(T))q^{\tr(T)},\,q=e^{2\pi i z},\,z\in\mathcal{H},\]
where $\mathrm{c}_{J}(T)$ is the largest integer $c$ such that $T/c\in J$,
and $\sigma_{3}(n)=\sum_{d|n}d^{3}$.
This theta series is a modular form of weight $12$ for $\SL_{2}(\Z)$.
Moreover, 
\[\vartheta_{J_{1}}=E_{12}+\frac{432000}{691}\Delta,\,\vartheta_{J_{2}}=E_{12}-\frac{65520}{691}\Delta,\]
where $E_{12}$ is the normalized Eisenstein series of weight $12$,
and $\Delta$ is the discriminant modular form.
\begin{rmk}\label{rmk origin of coefficients in exceptional theta series}
    The coefficient $240\sigma_{3}(\mathrm{c}(T))$ appearing in the Fourier expansion of $\vartheta_{J}$
    comes from \emph{Kim's modular form} $\mathrm{F}_{Kim}$,
    an Eisenstein series on the exceptional tube domain $\mathcal{H}_{\jord}$ (see \Cref{section exceptional modular forms}),
    which is constructed in \cite{Kim1993}
    and serves as our source for producing theta series.
\end{rmk}
We extend the construction of Elkies-Gross to \emph{weighted exceptional theta series} as follows:
\begin{introthm}\label{introthm theta series is modular form}
    (\Cref{thm Fourier of theta lift},\Cref{Cor Fourier theta lift no average})
    For any Albert lattice $J\in\mathcal{J}$ 
    and a polynomial $P\in\vrep{n}(\jord_{\C})$,
    the theta series
    \begin{align}\label{eqn intro exceptional theta series}
        \vartheta_{J,P}(z):=\sum_{\substack{J\ni T\geq 0,\\ \mathrm{rank}(T)=1}}\sigma_{3}(\mathrm{c}_{J}(T))P(T)q^{\tr(T)}
    \end{align}
    is a modular form of weight $2n+12$ for $\SL_{2}(\Z)$.
    When $n=\deg(P)>0$, $\vartheta_{J,P}$ is a cusp form.
\end{introthm}
Our proof of \Cref{introthm theta series is modular form} follows Pollack's method for the proof of \cite[Theorem 1.1.1]{pollack2023exceptional}.
For the automorphic form (or precisely, \emph{algebraic modular form}) of $\grpF$ associated with $J$ and $P$,
we construct its global theta lift to $\pgl_{2}$,
taking certain (iterated) differential of Kim's modular form $\mathrm{F}_{Kim}$ as the kernel function.
Then we show that this global theta lift arises from a holomorphic modular form,
whose Fourier expansion is exactly \eqref{eqn intro exceptional theta series}.
\begin{rmk}\label{rmk realization of theta lift in terms of theta series}
    Here we explain briefly how we describe the global theta lift from $\grpF$ to $\pgl_{2}$ in terms of exceptional theta series,
    and more details can be found in \Cref{section automorphic forms of F4}.
    The space $\mathcal{A}_{\vrep{n\varpi_{4}}}(\grpF)$ of \emph{level one ``vector-valued'' automorphic form of $\grpF$ with weight $\vrep{n\varpi_{4}}$}
    can be identified with the space of functions 
    $f:\mathcal{J}\rightarrow \vrep{n}(\jord_{\C})$ satisfying $f(gJ)=g.f(J)$ for any $g\in\grpF(\R)$ and $J\in\mathcal{J}$.
    The global theta lift of $f$ to $\pgl_{2}$ is the modular form 
    \[\frac{1}{|\Gamma_{1}|}\vartheta_{J_{1},f(J_{1})}+\frac{1}{|\Gamma_{2}|}\vartheta_{J_{2},f(J_{2})}\in\mathrm{M}_{2n+12}(\SL_{2}(\Z)),\]
    where $\Gamma_{i}$ is the automorphism group of the Albert $\Z$-algebra $J_{i}$, $i=1,2$.
\end{rmk}
For any $n\geq 2$,
if we take $P_{n}$ to be the unique non-zero $\spin_{9}(\R)$-invariant polynomial in $\vrep{n}(\jord_{\C})$,
then \Cref{introthm theta series is modular form} gives us a weight $2n+12$ cusp form $\vartheta_{\jord_{\Z},P_{n}}$,
which can be verified to be non-zero by analyzing the Fourier coefficient of $q$ (\Cref{thm for each weight a non-zero lift}).
Combining this with \Cref{introthm SV L-factor for Spin9 F4},
we show that the Petersson inner product of $\vartheta_{\jord_{\Z},P_{n}}$ with any Hecke eigenform in $\mathrm{S}_{2n+12}(\SL_{2}(\Z))$ is non-zero,
and we have the following analogue to Waldspurger's result for classical theta series: 
\begin{introthm}\label{introthm theta series generate whole space of cusp form}
    (\Cref{cor theta series span the whole space})
    For any $n>0$,
    the space $\mathrm{S}_{2n+12}(\SL_{2}(\Z))$
    is spanned by the set of weighted exceptional theta series
    $\set{\vartheta_{J,P}}{J=J_{1}\text{ or }J_{2},P\in\vrep{n}(\jord_{\C})}$.
\end{introthm}

We end the introduction with a short summary of the contents of this paper.
We recall the necessary preliminaries on exceptional groups in \Cref{section exceptional groups},
and the results on local theta correspondences in \Cref{section local theta correspondences}.
We establish the global theta correspondence in \Cref{section automorphic representations and global theta lift from min rep},
then study the Fourier expansions of exceptional theta series and prove \Cref{introthm theta series is modular form} in \Cref{section exceptional theta series on F4}.
The last section \Cref{section nonvanishing global theta lift} 
is for the proof of 
\Cref{introthm main theorem}, \Cref{introthm SV L-factor for Spin9 F4} and \Cref{introthm theta series generate whole space of cusp form}.
\subsection*{Acknowledgements}
The author thanks Gaëtan Chenevier 
for proposing this interesting project,
and also thanks Edmund Karasiewicz, Nhat Hoang Le, Chuhao Huang, Aaron Pollack, Gordan Savin and Jialiang Zou for stimulating discussions.
The author thanks Skip Garibaldi, Holger P. Petersson and Michel L. Racine for providing a simpler characterization of Albert lattices described in \Cref{prop simplify definition of Albert lattices}. 
Special thanks go to Wee Teck Gan 
for pointing out the strategy towards \Cref{conj existence of representations} to the author,
the invitation for a visit to the National University of Singapore,
and a lot of helpful discussions.
The author warmly thanks the referee for insightful comments and suggestions,
in particular for those concerning the removal of the level one condition in the original version of \Cref{introthm main theorem}.

\section{Preliminaries on exceptional groups}\label{section exceptional groups}
 In this section we recall the definitions of two reductive algebraic groups $\grpF$ and $\grpE$ over $\Q$ 
and construct the following two reductive dual pairs
\footnote{Actually we do not prove in this paper that $\spin_{9}\times\sorth_{2,2}$ is indeed a reductive dual pair,
instead we only give a homomorphism $\spin_{9}\times\sorth_{2,2}\rightarrow \grpE$,
whose kernel is a central cyclic group of order $2$.
}
inside $\grpE$:
\[\grpF\times \pgl_{2}\text{ and }\spin_{9}\times \sorth_{2,2}.\]
\subsection{Octonions}\label{section octonions}
We first recall the notion of octonions, 
which are crucial for defining exceptional groups.
\begin{defi}\label{def general octonion algebra}
    An \emph{octonion algebra} over a commutative ring $k$ 
    is a locally free $k$-module $C$ of rank $8$,
    equipped with a non-degenerate quadratic form $N:C\rightarrow k$
    and a (possibly non-associative) $k$-algebra structure admitting a $2$-sided identity element $e$,
    such that $N(xy)=N(x)N(y),\,x,y\in C$.
    The quadratic form $N$ is referred as the \emph{norm} on $C$.
\end{defi}
Now we recall some basic properties of octonion algebras, for which we refer to \cite{OctSpr}.
There is a unique anti-involution of algebra $x\mapsto \overline{x}$ called the \emph{conjugation} on $C$,
with the property that $N(x)=x\overline{x}=\overline{x}x$.
The \emph{trace} is defined as the linear map $\tr:C\rightarrow k,\,x\mapsto x+\overline{x}$.
The symmetric bilinear form associated with $N$ is 
$\langle x,y\rangle:=N(x+y)-N(x)-N(y)=\mathrm{Tr}(x\overline{y})$.

Although the multiplication law of $C$ is not associative,
it is still \emph{trace-associative} in the sense that 
$\tr((xy)z)=\tr(x(yz))$ for all $x,y,z\in C$,
and we can define a trilinear form: $\tr(xyz):=\tr((xy)z)=\tr(x(yz))$.

When considering octonion algebras over $\R$, we have the following classification result:
\begin{prop}\label{prop unique real definite octonion}
    \cite[Theorem 15.1]{AdamsExceptionalGrp}
    Up to $\R$-algebra isomorphism,
    there is a unique octonion algebra $\oct_{\R}$ over $\R$ 
    whose norm $\mathrm{N}$ is positive definite,
    which is named as the \emph{real octonion division algebra}.
\end{prop}
We choose a basis $\{\mathrm{e}_{0},\mathrm{e}_{1},\ldots,\mathrm{e}_{7}\}$ as in \cite[\S 4]{G96},
where $\mathrm{e}_{0}$ is the $2$-sided identity element.
Identify the real numbers $\R$ with the subalgebra $\R \mathrm{e}_{0}$ of $\oct_{\R}$,
and denote the identity element $\mathrm{e}_{0}$ by $1$.
On $\oct_{\R}$,
the conjugation is defined by $\overline{1}=1$ and $\overline{\mathrm{e}_{i}}=-\mathrm{e}_{i}$ for each $i$.
For any element $x=\sum\limits_{i=0}^{7}x_{i}\mathrm{e}_{i}\in\oct_{\R}$,
one has $\mathrm{N}(x)=\sum_{i=0}^{7}x_{i}^{2}$ and $\tr(x)=2x_{0}$.
\begin{defi}\label{def Cayley octonion algebra}
    \emph{Cayley's definite octonion algebra} $\oct_{\Q}$ is 
    the sub-$\Q$-algebra of $\oct_{\R}$,
    generated by $\{\mathrm{e}_{1},\ldots,\mathrm{e}_{7}\}$,
    which is an octonion $\Q$-algebra with the norm $\mathrm{N}|_{\oct_{\Q}}$.   
\end{defi}
The following definition gives an integral structure of Cayley's definite octonion algebra:
\begin{defi}\label{def Coxeter integral octonion}
\emph{Coxeter's integral order} $\coct$ in $\oct_{\Q}$ is the lattice spanned by $\Z\oplus\Z \mathrm{e}_{1}\oplus\cdots\oplus\Z \mathrm{e}_{7}$ and
\begin{align*}
    \mathrm{h}_{1}&=(1+\mathrm{e}_{1}+\mathrm{e}_{2}+\mathrm{e}_{4})/2,\,\mathrm{h}_{2}=(1+\mathrm{e}_{1}+\mathrm{e}_{3}+\mathrm{e}_{7})/2,\\
    \mathrm{h}_{3}&=(1+\mathrm{e}_{1}+\mathrm{e}_{5}+\mathrm{e}_{6})/2,\,\mathrm{h}_{4}=(\mathrm{e}_{1}+\mathrm{e}_{2}+\mathrm{e}_{3}+\mathrm{e}_{5})/2,
\end{align*}
which is an octonion $\Z$-algebra with the norm $\mathrm{N}|_{\oct_{\Z}}$.
\end{defi}

\subsection{Albert algebras}
\label{section glimpse of Albert algebra}
In this section, we will not generally define either an Albert algebra or a (cubic) Jordan algebra,
where precise definitions and details can be found in \cite{Garibaldi_Petersson_Racine_AlbertAlgOverZ}.
Instead, we recall some examples and properties of Albert algebras that are important for us.
\subsubsection{Hermitian 3-by-3 matrices over octonion algebras}
\label{section Her 3 algebra}
Given an octonion algebra $C$ over a commutative ring $k$,
we consider the space $\mathrm{Her}_{3}(C)$ consisting of ``Hermitian matrices'' in $\mathrm{M}_{3}(C)$,
\emph{i.e.}\,matrices of the form 
\begin{align*}
        [a,b,c\,;x,y,z]:=\left(\begin{matrix}
        a&z&\overline{y}\\
        \overline{z}&b&x\\
        y&\overline{x}&c
        \end{matrix}\right),\ a,b,c\in k,\ x,y,z\in C,
\end{align*}
equipped with the following structures,
where the maps are all polynomial laws in the sense of \cite{RobyPolyLaw}:
\begin{itemize}
    \item the identity matrix $\mathrm{I}=\mathrm{diag}(1,1,1)$,
    \item the adjoint map $\#:\mathrm{Her}_{3}(C)\rightarrow\mathrm{Her}_{3}(C)$, which is a quadratic map over $k$: 
    \begin{align}\label{eqn adjoint of albert algebra}
        \left(\begin{matrix}
            a&z&\overline{y}\\
            \overline{z}&b&x\\
            y&\overline{x}&c
            \end{matrix}\right)
            \mapsto \left(\begin{matrix}
            bc-N(x)&\overline{xy}-cz&zx-b\overline{y}\\
            xy-c\overline{z}&ca-N(y)&\overline{yz}-ax\\
            \overline{zx}-by&yz-a\overline{x}&ab-N(z)
            \end{matrix}\right),
    \end{align}
        \item and the determinant, which is a cubic form over $k$: 
        \begin{align}\label{eqn det albert algebra}
            \det([a,b,c\,;x,y,z]):=abc+\mathrm{Tr}(xyz)-a N(x)-b N(y)-c N(z).
        \end{align}
\end{itemize}
One can construct more polynomial laws from these structures:
\begin{itemize}
    \item There exists a symmetric bilinear form on $\mathrm{Her}_{3}(C)$: 
    \begin{align*}
        (A,B):=\left(\nabla_{A}\det\right)(\mathrm{I})\cdot \left(\nabla_{B}\det\right)(\mathrm{I})-\left(\nabla_{A}\nabla_{B}\det\right)(\mathrm{I}),
    \end{align*}
    where $\nabla$ is the directional derivative of polynomial laws (see \cite[\S 3]{Garibaldi_Petersson_Racine_AlbertAlgOverZ}).
    Explicitly, if $A=[a,b,c\,;x,y,z]$ and $B=[a^{\prime},b^{\prime},c^{\prime}\,;x^{\prime},y^{\prime},z^{\prime}]$,
    then \[(A,B)=aa^{\prime}+bb^{\prime}+cc^{\prime}+\lrangle{x}{x^{\prime}}+\lrangle{y}{y^{\prime}}+\lrangle{z}{z^{\prime}}.\] 
    \item The \emph{trace} $\tr:\mathrm{Her}_{3}(M)\rightarrow k$ is defined as $\tr(A)=(A,\mathrm{I})$.
    \item The linearization of $\#$ gives a symmetric cross product $A\times B:=(A+B)^{\#}-A^{\#}-B^{\#}$.
\end{itemize}
With these structures, we can define the rank of a matrix in $\mathrm{Her}_{3}(C)$:
\begin{defi}\label{def rank of matrix in Jordan algebra}
The \emph{rank} of $A\in \mathrm{Her}_{3}(C)$ is defined as follows:
\begin{itemize}
    \item If $A=0$, then $\mathrm{rank}(A)=0$;
    \item If $A\neq 0$ and $A^{\#}=0$, then $\mathrm{rank}(A)=1$;
    \item If $A\neq 0,A^{\#}\neq 0$ and $\det(A)=0$, then $\mathrm{rank}(A)=2$;
    \item Otherwise, $\mathrm{rank}(A)=3$. 
\end{itemize}    
\end{defi}
\subsubsection{Euclidean exceptional Jordan \texorpdfstring{$\R$}{PDFstring}-algebra and its $\Q$-structure}
\label{section exceptional real Jordan algebra}
One of the most important Albert algebras appearing in this article is the following:
\begin{defi}\label{def real definite exceptional Jordan algebra}
    The \emph{Euclidean exceptional Jordan $\R$-algebra} (or \emph{Euclidean Albert $\R$-algebra})
    is defined to be $\jord_{\R}:=\mathrm{Her}_{3}(\oct_{\R})$,
    where $\oct_{\R}$ is the real octonion division algebra.
\end{defi}
The space $\jord_{\R}$ 
is a commutative but not associative $\R$-algebra 
with respect to the $\R$-bilinear multiplication law
$A\circ B:=\frac{1}{2}(AB+BA)$,
where $AB$ and $BA$ denote the matrix multiplication,
and $\mathrm{I}$ is its $2$-sided identity element.
One can easily check that the symmetric bilinear form $(\,,\,)$ satisfies 
$(A,B)=\tr(A\circ B)$ for any $A,B\in\jord_{\R}$,
and it is positive definite.
\begin{defi}\label{def positive definite of Jordan element}
    A matrix $A=[a,b,c\,;x,y,z]\in \jord_{\R}$ is \emph{positive semi-definite} if its seven minor determinants
\[a,b,c,bc-\norm(x),ca-\norm(y),ab-\norm(z),\det(A)\]
are all non-negative, 
and we write $A\geq 0$.
When these minor determinants are all positive, we call $A$ \emph{positive definite} and write $A>0$.
\end{defi}
Similarly to \Cref{def Cayley octonion algebra}, this algebra $\jord_{\R}$ admits a rational structure:
\begin{defi}\label{def rational and integral exceptional Jordan algebra}
    The \emph{Euclidean exceptional Jordan $\Q$-algebra} $\jord_{\Q}$
    is the sub-$\Q$-algebra of $\jord_{\R}$ consisting of $[a,b,c\,;x,y,z],a,b,c\in\Q,x,y,z\in\oct_{\Q}$
    equipped with the multiplication $\circ$.
\end{defi}
\begin{notation}\label{notation elements of Jordan algebra}
    Here we fix some elements in $\jord_{\Q}$ that will be used frequently in this paper:
    \[ \mathrm{E}_{1}:=[1,0,0\,;0,0,0],\,\mathrm{E}_{2}:=[0,1,0\,;0,0,0],\,\mathrm{E}_{3}:=[0,0,1\,;0,0,0].\]
\end{notation}
\subsubsection{Albert algebras over \texorpdfstring{$\Z$}{PDFstring}}
\label{section Albert algebras}
Let $k$ be a commutative ring.
\emph{Albert} $k$-algebras are defined in \cite[Definition 7.1]{Garibaldi_Petersson_Racine_AlbertAlgOverZ}
Roughly speaking, 
an Albert $k$-algebra 
is a projective $k$-module $J$ together with 
a distinguished point $1_{J}$, a quadratic map $\#:J\rightarrow J$ and a cubic form $d:J\rightarrow k$ (as polynomial laws in the sense of \cite{RobyPolyLaw})
satisfying certain equations,
such that for some faithfully flat $k$-algebra $K$,
$J\otimes_{k}K$ is isomorphic to $\mathrm{Her}_{3}(C_{K})$ as \emph{Jordan $K$-algebras},
where $C_{K}$ is an octonion $K$-algebra.
For any ring homomorphism $k\rightarrow k^{\prime}$,
the scalar extension $J\otimes_{k}k^{\prime}$ of an Albert $k$-algebra $J$
is an Albert $k^{\prime}$-algebra.
\begin{defi}\label{def isomorphism of Albert algebras}
   \cite[Lemma 10.3]{Garibaldi_Petersson_Racine_AlbertAlgOverZ}
    An isomorphism of Albert $k$-algebras $\phi:J\rightarrow J^{\prime}$
is a $k$-module isomorphism such that $\phi(1_{J})=1_{J^{\prime}}$ and $d_{J^{\prime}}\circ \phi=d_{J}$
\footnote{Here $\circ$ means the composition, not the multiplication defined in \Cref{section exceptional real Jordan algebra}.}
as polynomial laws.
\end{defi}

\begin{ex}\label{ex all algebras above are Albert}
    The space of 3-by-3 Hermitian matrices
    $\mathrm{Her}_{3}(C)$ defined in \Cref{section Her 3 algebra} is an Albert $k$-algebra.
    In particular,
    $\jord_{\R}$ and $\jord_{\Q}$ defined in \Cref{section exceptional real Jordan algebra}
    are Albert algebras over $\R$ and $\Q$ respectively. 
\end{ex}
Here are several classification results in \cites[\S 5.8]{OctSpr}[\S 11, \S 14]{Garibaldi_Petersson_Racine_AlbertAlgOverZ} about Albert algebras that will be useful for us:
\begin{enumerate}[label=(\arabic*)]
    \item There is a unique isomorphism class of Albert $\R$-algebras that are \emph{Euclidean},
    \emph{i.e.}\,the associated symmetric bilinear form is positive definite,
    and this class is represented by $(\jord_{\R},\mathrm{I},\#,\det)$ defined in \Cref{section exceptional real Jordan algebra}.
    \item Euclidean Albert $\Q$-algebras are also unique up to isomorphism.
    \item Albert $\Z_{p}$-algebras are unique up to isomorphism.
    \item There are exactly two isomorphism classes of Euclidean Albert $\Z$-algebras.
\end{enumerate}
In this article, 
we concentrate on the following family of Euclidean Albert $\Z$-algebras:
\begin{defi}\label{def definite Albert algebra inside real model}
    An \emph{Albert lattice} of $\jord_{\R}$ is a lattice $J\subseteq \jord_{\R}$ satisfying:
    \begin{itemize}
        \item The identity matrix $\mathrm{I}=\mathrm{diag}(1,1,1)\in\jord_{\R}$ is contained in $J$;
        \item It is stable under the adjoint map $\#$ defined in \eqref{eqn adjoint of albert algebra};
        \item The cubic form $\det$ defined in \eqref{eqn det albert algebra} takes integral values on $J$;
        \item Together with $\mathrm{I},\#$ and $\det$, $J$ is an Albert $\Z$-algebra.
    \end{itemize}
    Denote the set of Albert lattices inside $\jord_{\R}$ by $\mathcal{J}$.
\end{defi}
\begin{ex}\label{ex trivial Albert Z algebra}
        Let $\jord_{\Z}:=\mathrm{Her}_{3}(\oct_{\Z})$,
        \emph{i.e.}\,the rank $27$ lattice 
        \[\set{[a,b,c\,;x,y,z]\in\jord_{\Q}}{a,b,c\in\Z,x,y,z\in\oct_{\Z}}\]
        inside $\jord_{\Q}$.       
        It satisfies the conditions in \Cref{def definite Albert algebra inside real model},
        thus it is an Albert lattice.   
\end{ex}
\begin{ex}\label{ex isotopy Albert Z algebra}
    An Albert $\Z$-algebra not isomorphic to $(\jord_{\Z},\mathrm{I},\#,\det)$ defined in \Cref{ex trivial Albert Z algebra}
    is constructed as follows, following \cites[\S 4]{G96}[Definition 14.1]{Garibaldi_Petersson_Racine_AlbertAlgOverZ}.
    Take
        \begin{align*}
            \mathrm{E}=[2,2,2;\beta,\beta,\beta],\,\beta=\frac{1}{2}\left(-1+\mathrm{e}_{1}+\mathrm{e}_{2}+\cdots+\mathrm{e}_{7}\right)\in\coct.
        \end{align*}
        This element $\mathrm{E}\in\jord_{\Z}$ is positive definite and has determinant $1$. 
        We define another quadratic map $\#^{\mathrm{E}}$ on $\jord_{\Z}$
        by $X^{\#^{\mathrm{E}}}:=(\mathrm{E},X^{\#})\mathrm{E}-\mathrm{E}^{\#}\times X^{\#}$.
        Set $\jord_{\Z}^{(\mathrm{E})}:=(\jord_{\Z},\mathrm{E},\#^{\mathrm{E}},\det)$,
        where $\det$ is still the restriction of $\det:\jord_{\R}\rightarrow\R$ to $\jord_{\Z}$.
        This ``\emph{isotopy}'' $\jord_{\Z}^{(\mathrm{E})}$ is an Albert $\Z$-algebra \cite[Corollary 13.11]{Garibaldi_Petersson_Racine_AlbertAlgOverZ},
        and it is not isomorphic to $(\jord_{\Z},\mathrm{I},\#,\det)$ as Albert $\Z$-algebras \cite[Proposition 5.5]{Leech}. 

        The associated symmetric bilinear form $(\,,\,)$ on $\jord_{\Z}^{(\mathrm{E})}$ 
        is positive definite \cite[Proposition 2.10]{Leech},
        thus $\jord_{\Z}^{(\mathrm{E})}$ is Euclidean.
        By the classification result about Euclidean Albert $\R$-algebras listed above,
        we have an isomorphism $\varphi:\jord_{\Z}^{(\mathrm{E})}\otimes_{\Z}\R\rightarrow \jord_{\R}$ of Albert $\R$-algebras.
        Its image $\varphi(\jord_{\Z}^{(\mathrm{E})})$ is an Albert lattice of $\jord_{\R}$ in the sense of \Cref{def definite Albert algebra inside real model}.
\end{ex}
Next, we note a simpler characterization of Albert lattices that we thank Garibaldi, Petersson and Racine for pointing out:
\begin{prop}
    \label{prop simplify definition of Albert lattices}
    \cite[Theorem A2.7]{Garibaldi_Petersson_Racine_Solutions_Alber_algebras}
    Let $J$ be a unimodular lattice in $\jord_{\R}$ that contains $\mathrm{I}$ and is closed under the adjoint map $\#$,
    then $J$ is an Albert lattice.
\end{prop}

\subsection{\texorpdfstring{$\grpF$}{}}\label{section exceptional group F4}
We start to define exceptional algebraic groups.
\begin{defi}\label{def algebraic group F4}
    Define $\grpF$ to be the closed subgroup of the algebraic $\Q$-group $\GL_{\jord_{\Q}}$,
    that (as a functor) sends a commutative $\Q$-algebra $R$ to the group 
    \[\grpF(R):=\set{g\in\mathrm{GL}(\jord_{\Q}\otimes_{\Q}R)}{g(A\circ B)=g(A)\circ g(B),\text{ for any }A,B\in \jord_{\Q}\otimes_{\Q}R}.\]
\end{defi}
By \cite[Theorem 7.2.1]{OctSpr},
$\grpF$ is a semisimple and simply-connected $\Q$-group of Lie type $\mathrm{F}_{4}$. 
The real points $\lietype{F}{4}:=\grpF(\R)$ of $\grpF$ is contained in the isometry group $\mathrm{O}(\jord_{\R},\mathrm{q})$ of the positive definite quadratic form $\mathrm{q}$,
thus it is compact.
For every prime $p$,
$\grpF$ is split over $\Q_{p}$.
By \cite[Proposition 5.9.4]{OctSpr},
the $\Q$-group $\grpF$ coincides with the algebraic group consisting of the Albert algebra automorphisms of $\jord_{\Q}$,
\emph{i.e.}\,sending any commutative $\Q$-algebra $R$ to
\[\set{g\in\mathrm{GL}(\jord_{\Q}\otimes_{\Q}R)}{g(\mathrm{I})=\mathrm{I},\det(gA)=\det(A),\text{ for any }A\in\jord_{\Q}\otimes_{\Q}R}.\]
With this coincidence,
we construct \emph{reductive $\Z$-models} of $\grpF$ in the sense of \cite{G96} 
as group of Albert algebra automorphisms of elements $J\in\mathcal{J}$.

\begin{defi}\label{def Z group of type F4}
    Given an Albert lattice $J\in\mathcal{J}$,
    define $\mathbf{Aut}_{J/\Z}$ to be the $\Z$-group scheme sending 
    a commutative $\Z$-algebra $R$ to the group 
    \[\mathbf{Aut}_{J/\Z}(R):=\set{g\in\mathrm{GL}(J\otimes_{\Z}R)}{g(\mathrm{I})=\mathrm{I},\det(gA)=\det(A),\text{ for any }A\in J\otimes_{\Z}R}.\]
    If we take $J$ to be $\jord_{\Z}$ defined in \Cref{ex trivial Albert Z algebra},
    we denote the group scheme $\mathbf{Aut}_{\jord_{\Z}/\Z}$ by $\mathcal{F}_{4,\mathrm{I}}$.
\end{defi}
The following result shows that $\mathbf{Aut}_{J/\Z}$ is a reductive $\Z$-model of $\grpF$:
\begin{prop}\label{prop model of F4 from a polarization}
    \cite[Lemma 9.1]{Garibaldi_Petersson_Racine_AlbertAlgOverZ}
    For any choice of Albert lattice $J\in\mathcal{J}$,
    the group scheme $\mathbf{Aut}_{J/\Z}$ is smooth and 
    its fiber $\mathbf{Aut}_{J/\Z}\otimes_{\Z}\Z/p\Z$ is semisimple for every prime $p$.
    Moreover, the generic fiber of $\mathbf{Aut}_{J/\Z}$ is $\grpF$.
\end{prop}
In \cite[Proposition 5.3]{G96},
Gross proves that there are exactly two $\grpF(\Q)$-orbits on the equivalence classes of reductive $\Z$-models of $\grpF$ in the \emph{genus} of $\mathcal{F}_{4,\mathrm{I}}$.
From now on we fix a reductive $\Z$-model $\mathcal{F}_{4,\mathrm{I}}$ of $\grpF$,
and we have the following formulation of the double cosets space $\grpF(\Q)\backslash\grpF(\A)/\mathcal{F}_{4,\mathrm{I}}(\widehat{\Z})$.

\begin{prop}\label{prop reformulation of double coset}
    There is a bijection $\mathcal{J}\xrightarrow{\simeq}\grpF(\Q)\backslash \grpF(\A)/\mathcal{F}_{4,\mathrm{I}}(\widehat{\Z})$ sending
    the base point $\jord_{\Z}$ to the double coset of the identity of $\grpF(\A)$.
\end{prop}
\begin{proof}
    For any $J\in\mathcal{J}$, 
    the Albert $\Q$-algebras $J\otimes_{\Z}\Q$ and $\jord_{\Z}\otimes_{\Z}\Q$ are isomorphic,
    so there exists an element $g_{\infty}\in \grpF(\R)$ inducing $J\otimes_{\Z}\Q\xrightarrow{\simeq}\jord_{\Z}\otimes_{\Z}\Q$.
    Set $J^{\prime}=g_{\infty}(J)$,
    which is an Albert $\Z$-algebra inside $\jord_{\Z}\otimes_{\Z}\Q=\jord_{\Q}$.
    Since $J^{\prime}\otimes_{\Z}\Z_{p}$ and $\jord_{\Z}\otimes_{\Z}\Z_{p}$ are isomorphic 
    as Albert $\Z_{p}$-algebras,
    we can choose an element $g_{p}\in \grpF(\Q_{p})$ that induces $\jord_{\Z}\otimes_{\Z}\Z_{p}\xrightarrow{\simeq}J^{\prime}\otimes_{\Z}\Z_{p}$.
    For almost all prime numbers $p$, 
    we have $J^{\prime}\otimes_{\Z}\Z_{p}=\jord_{\Z}\otimes_{\Z}\Z_{p}$,
    so the element $g_{p}$ lies in $\mathcal{F}_{4,\mathrm{I}}(\Z_{p})$ for almost all $p$.

    In this way, we associate with $J\in\mathcal{J}$ an element $(g_{\infty},g_{2},g_{3},\ldots)\in \grpF(\A)$,
    and it can be easily verified that 
    its image in $\grpF(\Q)\backslash\grpF(\A)/\mathcal{F}_{4,\mathrm{I}}(\widehat{\Z})$
    does not depend on the choice of $g_{\infty}$ and $g_{p}$.
    So we have a well-defined map $\mathcal{J}\rightarrow \grpF(\Q)\backslash\grpF(\A)/\mathcal{F}_{4,\mathrm{I}}(\widehat{\Z})$,
    and its inverse is:
    \[(g_{v})\mapsto g_{\infty}^{-1}\left(\bigcap_{p}\left(g_{p}\left(\jord_{\Z}\otimes_{\Z}\Z_{p}\right)\cap \jord_{\Q}\right)\right)\in\mathcal{J}.\qedhere\]   
\end{proof}
\begin{notation}\label{notation representatives of double cosets}
    We choose a set of representatives $\{1,\gamma_{\mathrm{E}}\}$ of $\grpF(\Q)\backslash\grpF(\A_{f})/\mathcal{F}_{4,\mathrm{I}}(\widehat{\Z})$,
and denote by $\jord_{\mathrm{E}}\subseteq \jord_{\Q}$ the Albert lattice corresponding to $\gamma_{\mathrm{E}}$.
Equipped with the natural $\grpF(\R)$-action,
$\mathcal{J}$ is the disjoint union of the $\grpF(\R)$-orbits of $\jord_{\Z}$ and $\jord_{\mathrm{E}}$.     
\end{notation}

\subsubsection{An algebraic group of type \texorpdfstring{$\lietype{E}{6}$}{PDFstring}}

\label{section exceptional group E6}
If we remove the condition of fixing the identity element $\mathrm{I}$ in the definition of $\mathcal{F}_{4,\mathrm{I}}$,
we get the following group of type $\lietype{E}{6}$:
\begin{defi}\label{def E6 group}
    Define $\mathbf{M}_{\jord}$ to be the $\Z$-group scheme sending any commutative ring $R$ to
    \[\set{(\lambda(g),g)\in R^{\times}\times \mathrm{GL}(\jord_{\Z}\otimes_{\Z}R)}{\det(gA)=\lambda(g)\det(A),\text{ for any }A\in\jord_{\Z}\otimes_{\Z}R},\]
    and $\mathbf{M}_{\jord}^{1}$ to be $\ker\lambda$.
\end{defi}
By \cite[Proposition 6.5]{NonRed},
$\mathbf{M}_{\jord}^{1}$ is a simply-connected semisimple group scheme of type $\mathrm{E}_{6}$,
and its generic fiber has $\Q$-rank $2$.
\begin{rmk}\label{rmk bilinear form not E6 invariant}
    Notice that the bilinear form $(\,,\,)$ on $\jord_{\Z}\otimes_{\Z}R$ is not $\mathbf{M}_{\jord}(R)$-invariant.
    For any $m\in\mathbf{M}_{\jord}(R)$, 
    we denote by $m^{*}$ the unique element in $\mathbf{M}_{\jord}(R)$ such that 
    $\left(m(X),m^{*}(Y)\right)=\left(m^{*}(X),m(Y)\right)=(X,Y)$ for any $X,Y\in \jord_{\Z}\otimes R$.    
\end{rmk}
We have already seen two Albert $\Z$-algebras $\jord_{\Z}^{(\mathrm{E})}$ and $\jord_{\mathrm{E}}$ that are both not isomorphic to $\jord_{\Z}$ and their extensions to $\Q$ are isomorphic to $\jord_{\Q}$.
By the classification result listed in \Cref{section Albert algebras} they are isomorphic,
although they have different distinguished points.
This fact gives us an element that will be used in the proof of \Cref{thm Fourier of theta lift}:
\begin{lemma}\label{lemma delta element between Albert algebras}
    There exists an element $\delta\in\mathbf{M}_{\jord}^{1}(\Q)$ 
    that induces an isomorphism of Albert $\Z$-algebras $\jord_{\Z}^{(\mathrm{E})}\xrightarrow{\simeq} \jord_{\mathrm{E}}$.
    Moreover, if we denote the image of $\delta$ under the diagonal embedding $\mathbf{M}_{\jord}^{1}(\Q)\hookrightarrow\mathbf{M}_{\jord}^{1}(\A)=\mathbf{M}_{\jord}^{1}(\R)\times\mathbf{M}_{\jord}^{1}(\A_{f})$
    by $(\delta_{\infty},\delta_{f})$,
    then $\delta_{\infty}(\jord_{\Z})=\jord_{\mathrm{E}}$, $\delta_{\infty}(\mathrm{E})=\mathrm{I}$ and $\delta_{f}^{-1}\gamma_{\mathrm{E}}\in\mathbf{M}_{\jord}^{1}(\widehat{\Z})$.
\end{lemma}
\begin{proof}
    Since the Albert $\Z$-algebras $\jord_{\Z}^{(\mathrm{E})},\mathrm{J}_{\mathrm{E}}$ contained in $\jord_{\Q}$
    are isomorphic,
    there is a $\Q$-linear isomorphism $\delta$ of $\jord_{\Q}$
    such that 
    $\delta(\jord_{\Z}^{(\mathrm{E})})=\jord_{\mathrm{E}}$, $\delta(\mathrm{E})=\mathrm{I}$
    and $\det(\delta A)=\delta(A)$ for any $A\in\jord_{\Q}$.
    In other words, $\delta$ is our desired element in $\mathbf{M}_{\jord}^{1}(\Q)$.
    The properties of $\delta_{\infty}$ follows immediately.
    Forgetting the Albert algebra structures,
    $\delta_{f}^{-1}\gamma_{\mathrm{E}}:\jord_{\Z}\otimes_{\Z}\widehat{\Z}\rightarrow \jord_{\Z}^{(\mathrm{E})}\otimes_{\Z}\widehat{\Z}$
    is a linear automorphism of $\jord_{\Z}\otimes_{\Z}\widehat{\Z}$
    preserving the determinant,
    thus $\delta_{f}^{-1}\gamma_{\mathrm{E}}\in\mathbf{M}_{\jord}^{1}(\widehat{\Z})$.
\end{proof}

\subsection{\texorpdfstring{$\grpE$}{}}\label{section exceptional group E7}
Now we recall the definition of $\grpE$,
a larger algebraic group over $\Q$ containing $\grpF$,
and our main references are \cites[\S 2.2]{PolFourierExpansion}[\S 3]{kim_yamauchi_2016}
\footnote{Notice that there are some slight mistakes in \cite[\S 3]{kim_yamauchi_2016}
and the correction is in \cite[\S 2]{KimYamauchi_HigherLevelE7}.}.

Consider the $56$-dimensional vector space $\mathrm{W}_{\mathrm{J}}=\jord_{\Q}\oplus \Q\oplus \jord_{\Q}\oplus \Q$
\footnote{In \cite{PolFourierExpansion}, Pollack considers the space $\Q\oplus\jord_{\Q}\oplus\jord_{\Q}^{\vee}\oplus\Q$.
An element $(X,\xi,X^{\prime},\xi^{\prime})\in\mathrm{W}_{\jord}$ corresponds to 
$(a,b,c,d)=(\xi^{\prime},X,\left(-,X^{\prime}\right),\xi)$
under the notations of Pollack.},
equipped with the following structures:
\begin{itemize}
    \item A symplectic form: for $w_{i}=(X_{i},\xi_{i},X_{i}^{\prime},\xi_{i}^{\prime})\in \mathrm{W}_{\jord},\,i=1,2$,
    \begin{align*}
    \langle w_{1},w_{2}\rangle_{\jord}:=\xi_{1}\xi_{2}^{\prime}-\xi_{2}\xi_{1}^{\prime}+(X_{1},X_{2}^{\prime})-(X_{2},X_{1}^{\prime});
    \end{align*}
    \item A quartic form: for $w=(X,\xi,X^{\prime},\xi^{\prime})\in \mathrm{W}_{\jord}$, 
    \begin{align*}
        \mathrm{Q}(w)=\left(\xi\xi^{\prime}-(X,X^{\prime})\right)^{2}+4\xi\det(X)+4\xi^{\prime}\det(X^{\prime})-4(X^{\#},X^{\prime\#}).        
    \end{align*}
\end{itemize}
\begin{defi}\label{def E7 group}
    Define $\mathbf{H}_{\jord}$ to be the algebraic subgroup of $\GL_{\mathrm{W}_{\jord}}$ 
    consisting of elements that preserve the forms $\lrangle{\,}{\,}_{\jord}$ and $\mathrm{Q}$
    up to some similitude $\nu:\mathbf{H}_{\jord}\rightarrow\mathbf{G}_{\mathrm{m}}$,
    \emph{i.e.}
    \[\mathbf{H}_{\jord}=\set{(\nu(g),g)\in \mathbf{G}_{\mathrm{m}}\times\GL_{\mathrm{W}_{\jord}}}{\lrangle{gv}{gw}_{\jord}=\nu(g)\lrangle{v}{w}_{\jord},\mathrm{Q}(gv)=\nu(g)^{2}\mathrm{Q}(v),\forall v,w\in\mathrm{W}_{\jord}}.\]
    Define $\mathbf{H}_{\jord}^{1}$ to be the kernel of $\nu$, 
    which is simply-connected and has $\Q$-rank $3$ and Lie type $\lietype{E}{7}$ \cite{FREUDENTHAL1954218},
    and $\grpE$ to be the adjoint group of $\mathbf{H}_{\jord}$.
\end{defi}
\begin{rmk}\label{rmk center of E7}
    The center of $\mathbf{H}_{\jord}$ consists of scalars,
    and it contains a specific element $\iota^{2}=-\mathrm{Id}_{\mathrm{W}_{\jord}}$,
    where $\iota\in \mathbf{H}_{\jord}$ is defined as 
    \begin{align}\label{eqn action of order 4 element in E7}
        \iota(X,\xi,X^{\prime},\xi^{\prime})=(-X^{\prime},-\xi^{\prime},X,\xi).
    \end{align}
\end{rmk}
%
In \cite{G96}, 
we know that $\grpE$ has a unique (up to equivalence) reductive $\Z$-models,
and we will also denote this $\Z$-group scheme by $\grpE$ when there is no confusion.
Note that $\grpE(\Z)$ is the image under the projection $\mathbf{H}_{\mathrm{J}}(\Q)\rightarrow \grpE(\Q)$
of the stabilizer in $\mathbf{H}_{\mathrm{J}}(\Q)$
of the lattice $\jord_{\Z}\oplus\Z\oplus\jord_{\Z}\oplus\Z\subseteq \mathrm{W}_{\jord}$.

\subsubsection{Siegel parabolic subgroup of \texorpdfstring{$\grpE$}{PDFstring}}
\label{section Siegel parabolic of E7}
\begin{defi}\label{def Siegel parabolic subgroup of E7}
    The \emph{Siegel parabolic subgroup} $\mathbf{P}_{\jord,\mathrm{sc}}$ of $\mathbf{H}_{\jord}^{1}$ 
    is defined as the stabilizer of the line $\Q(0,1,0,0)\subseteq\mathrm{W}_{\jord}$.
    A Levi subgroup of $\mathbf{P}_{\jord,\mathrm{sc}}$ can be defined as 
    the subgroup that also stabilizes $\Q(0,0,0,1)$.
    Denote by $\mathbf{P}_{\jord}$ the image of $\mathbf{P}_{\jord,\mathrm{sc}}$ in $\grpE$. 
\end{defi}
This Levi subgroup is isomorphic to $\mathbf{M}_{\jord}$,
and the action of $(\lambda(m),m)\in\mathbf{M}_{\jord}$ on $\mathrm{W}_{\jord}$ is
\begin{align*}
    m(X,\xi,X^{\prime},\xi^{\prime})=(m^{*}X,\lambda(m)\xi,mX^{\prime},\lambda(m)^{-1}\xi^{\prime}).
\end{align*}
The unipotent radical $\mathbf{N}_{\jord}$ of $\mathbf{P}_{\jord,\mathrm{sc}}$ is abelian and satisfies $\mathbf{N}_{\jord}(\Q)\simeq \jord_{\Q}$,
and any element of $\mathbf{N}_{\jord}(\Q)$ has the following form:

{\footnotesize\begin{align*}
    \mathrm{n}(A)(X,\xi,X^{\prime},\xi^{\prime})=\left(X+\xi^{\prime}A,\xi+(A,X^{\prime})+(A^{\#},X)+\xi^{\prime}\det(A),X^{\prime}+A\times X+\xi^{\prime}A^{\#},\xi^{\prime}\right),\,A\in\jord_{\Q}.
\end{align*}}
We have the Levi decomposition $\mathbf{P}_{\jord,\mathrm{sc}}=\mathbf{M}_{\jord}\mathbf{N}_{\jord}$,
and the action of $\mathbf{M}_{\jord}$ on $\mathbf{N}_{\jord}$ is given by the following lemma:
\begin{lemma}\label{lemma Levi action on unipotent radical}
    For any $m\in \mathbf{M}_{\jord}(\Q)\subseteq\mathbf{P}_{\jord,\mathrm{sc}}$ and $A\in \jord_{\Q}$, 
    we have the following identity:
    \begin{align*}
        m \mathrm{n}(A) m^{-1}=\mathrm{n}\bracket{\lambda(m)m^{*}A}.
    \end{align*}    
\end{lemma}
\begin{proof}
    This follows from a direct calculation using the property:
    for any $m\in\mathbf{M}_{\jord}(\Q)$ and $X,Y\in\jord_{\Q}$,
    we have $m(X\times Y)=\lambda(m)(m^{*}X)\times (m^{*}Y)$.
\end{proof}
The Levi subgroup of $\mathbf{P}_{\jord}\subseteq\grpE$ induced by $\mathbf{M}_{\jord}$ is the quotient of $\mathbf{M}_{\jord}$ by $\mu_{2}$,
where $\mu_{2}$ is generated by the element $X\mapsto -X$ in $\mathbf{M}_{\jord}$.
We identify this Levi subgroup with $\mathbf{M}_{\jord}$
via the short exact sequence:
\begin{align}\label{eqn identification of Levi subgroups}
    1\rightarrow \mu_{2}\rightarrow \mathbf{M}_{\jord}\xrightarrow{m\mapsto \lambda(m)m^{*}}\mathbf{M}_{\jord}\rightarrow 1.
\end{align} 
Hence we still have the Levi decomposition $\mathbf{P}_{\jord}\simeq\mathbf{M}_{\jord}\mathbf{N}_{\jord}$,
but with a different action:
\[m\mathrm{n}(A)m^{-1}=\mathrm{n}(mA),\text{ for any }m\in\mathbf{M}_{\jord}(\Q),A\in\jord_{\Q}.\]
\begin{rmk}\label{rmk opposite siegel parabolic}
    For any $A\in \jord_{\Q}$, we define $\mathrm{n}^{\vee}(A)=\iota\mathrm{n}(-A)\iota^{-1}$.
    Set $\overline{\mathbf{N}_{\jord}}=\iota \mathbf{N}_{\jord}\iota^{-1}$,
    then $\overline{\mathbf{P}_{\jord,\mathrm{sc}}}=\mathbf{M}_{\jord}\overline{\mathbf{N}_{\jord}}$ is the parabolic subgroup opposite to $\mathbf{P}_{\jord,\mathrm{sc}}$.
    The action of $\mathbf{M}_{\jord}$ on $\overline{\mathbf{N}_{\jord}}$ is: 
    \[m \mathrm{n}^{\vee}(A)m^{-1}=\mathrm{n}^{\vee}\left(\lambda(m)^{-1}mA\right),\text{ for any }m\in \LeviM_{\jord}(\Q),A\in \jord_{\Q}.\]
\end{rmk}

\subsubsection{The Lie algebra \texorpdfstring{$\mathfrak{e}_{7}$}{PDFstring}}\label{section Lie algebra E7}
Denote the Lie algebra of $\mathbf{H}_{\jord}^{1}(\C)$ by $\mathfrak{e}_{7}$,
which admits a decomposition 
\begin{align}\label{eqn Langlands decomposition of e7}
    \mathfrak{e}_{7}=\mathrm{n}_{\mathrm{L}}^{\vee}(\jord_{\C})\oplus \mathfrak{m}_{\jord}\oplus \mathrm{n}_{\mathrm{L}}(\jord_{\C}),
\end{align}
where 
\begin{itemize}
    \item $\mathfrak{m}_{\jord}=\mathrm{Lie}(\LeviM_{\jord}(\C))$;
    \item for any $A\in\jord_{\C}$,
    define $\mathrm{n}_{\mathrm{L}}(A)$ to be the element in $\mathrm{Lie}(\mathbf{N}_{\jord}(\C))$ such that $\exp(\mathrm{n}_{\mathrm{L}}(A))=\mathrm{n}(A)$,
    and denote $\mathrm{Lie}(\mathbf{N}_{\jord}(\C))$ by $\mathrm{n}_{\mathrm{L}}(\jord_{\C})$;
    \item for any $A\in\jord_{\C}$,
    define $\mathrm{n}_{\mathrm{L}}(A)$ to be the element in $\mathrm{Lie}(\overline{\mathbf{N}_{\jord}}(\C))$ such that $\exp(\mathrm{n}_{\mathrm{L}}(A))=\mathrm{n}^{\vee}(A)$,
    and denote $\mathrm{Lie}(\overline{\mathbf{N}_{\jord}}(\C))$ by $\mathrm{n}^{\vee}_{\mathrm{L}}(\jord_{\C})$.
\end{itemize}
Besides this decomposition, we also have the Cartan decomposition of $\mathfrak{e}_{7}$.
Let $\mathrm{K}_{\lietype{E}{7}}$ be the subgroup of $\mathbf{H}_{\jord}^{1}(\R)$
that fixes the line in $\mathrm{W}_{\jord}\otimes\C$ spanned by $(i\mathrm{I},-i,-\mathrm{I},1)$,
which is a maximal compact subgroup of $\mathbf{H}_{\jord}^{1}(\R)$.
Take $\mathfrak{k}_{\lietype{E}{7}}$ to be the complexified Lie algebra of $\mathrm{K}_{\lietype{E}{7}}$,
then we have the following Cartan decomposition of $\mathfrak{e}_{7}$:
\begin{align}\label{eqn Cartan decomposition}
    \mathfrak{e}_{7}=\mathfrak{p}_{\jord}^{-}\oplus \mathfrak{k}_{\lietype{E}{7}}\oplus \mathfrak{p}_{\jord}^{+},
\end{align}
where $\mathfrak{p}_{\jord}^{+}\oplus \mathfrak{p}_{\jord}^{-}$ is the natural decomposition of the $-1$ eigenspace for the Cartan involution.
We have the following relation between these two decompositions \eqref{eqn Langlands decomposition of e7} and \eqref{eqn Cartan decomposition} of $\mathfrak{e}_{7}$:
\begin{prop}\label{prop cayley transform switch two decompositions}\cite[Proposition 6.1.1]{pollack2023exceptional}
    There exists an element $\mathrm{C}_{h}\in\mathbf{H}^{1}_{\jord}(\C)$,
    called the \emph{Cayley transform},
    satisfying:
    \begin{enumerate}[label=(\arabic*)]
        \item $\mathrm{C}_{h}^{-1}\mathrm{n}_{\mathrm{L}}(\jord_{\C})\mathrm{C}_{h}=\mathfrak{p}_{\jord}^{+}$;
        \item $\mathrm{C}_{h}^{-1}\mathrm{n}_{\mathrm{L}}^{\vee}(\jord_{\C})\mathrm{C}_{h}=\mathfrak{p}_{\jord}^{-}$;
        \item $\mathrm{C}_{h}^{-1}\mathfrak{m}_{\jord}\mathrm{C}_{h}=\mathfrak{k}_{\lietype{E}{7}}$.
    \end{enumerate}
\end{prop}
By \Cref{prop cayley transform switch two decompositions}, we make the following identifications: 
\begin{itemize}
    \item Identify the factor $\mathfrak{p}_{\jord}^{+}$ 
    as $\jord_{\C}^{\vee}$, via the map  
    \[\mathfrak{p}_{\jord}^{+}\ni\mathrm{X}_{A}^{+}:=i\mathrm{C}_{h}^{-1}\mathrm{n}_{\mathrm{L}}(A)\mathrm{C}_{h} \mapsto (-,A)\in \jord_{\C},\]
    and equip it with the following $\mathbf{M}_{\jord}(\C)$-action:
    \[(m.\ell)(X)=\ell\left(\lambda(m)m^{-1}(X)\right),\text{ for any }m\in\mathbf{M}_{\jord}(\C),\ell\in\jord_{\C}^{\vee},X\in\jord_{\C}.\] 
    \item Identify $\mathfrak{p}_{\jord}^{-}$ as $\jord_{\C}$,
    via the map 
    \[\mathfrak{p}_{\jord}^{-}\ni\mathrm{X}_{A}^{-}:=i\mathrm{C}_{h}^{-1}\mathrm{n}_{\mathrm{L}}^{\vee}(A)\mathrm{C}_{h} \mapsto A\in\jord_{\C},\]
    and equip it with the following $\mathbf{M}_{\jord}(\C)$-action: 
    \[m.X=\lambda(m)^{-1}m(X)\text{ for any }m\in\mathbf{M}_{\jord}(\C),X\in\jord_{\C}.\]
\end{itemize}
The natural $\mathbf{M}_{\jord}(\C)$-invariant pairing $\{-,-\}:\jord_{\C}\times\jord_{\C}^{\vee}\rightarrow \C$
can be extended to 
{\small\begin{align}\label{eqn symmetric pairing}
    \{-,-\}:\jord_{\C}^{\otimes n}\times\left(\jord_{\C}^{\vee}\right)^{\otimes n}\rightarrow\C, \left(X_{1}\otimes\cdots\otimes X_{n},\ell_{1}\otimes\cdots\otimes\ell_{n}\right)\mapsto \frac{\sum_{\sigma\in S_{n}}\prod_{i=1}^{n}\left\{X_{i},\ell_{\sigma(i)}\right\}}{n!},
\end{align}}
which factors through $\sym^{n}\jord_{\C}\times\sym^{n}\left(\jord_{\C}^{\vee}\right)$.
\begin{ex}\label{ex special case of pairing}
    Identifying $\sym^{n}\left(\jord_{\C}^{\vee}\right)$ with 
    the space $\mathrm{P}_{n}(\jord_{\C})$ of degree $n$ homogeneous polynomials over $\jord_{\C}$,
    the $\mathbf{M}_{\jord}(\C)$-action on it is 
    $(m.P)(X)=P(\lambda(m)m^{-1}(X))$
    for any $m\in\mathbf{M}_{\jord}(\C),\,P\in\mathrm{P}_{n}(\jord_{\C})$ and $T\in\jord_{\C}$,
    and the pairing $\{T^{\otimes n},P\}$ is equal to $P(T)$.
\end{ex}

\subsection{Dual pairs}\label{section dual pair}
 Now we explain the two reductive dual pairs 
$\grpF\times\pgl_{2}$ and 
$\spin_{9}\times \sorth_{2,2}$ in $\grpE$.
\subsubsection{\texorpdfstring{$\grpF\times \pgl_{2}$}{}}\label{section dual pair involving F4}
 
We study first the centralizer of $\grpF$ in $\mathbf{M}_{\jord}$.
For any element $g$ in the centralizer $\mathbf{C}_{\mathbf{M}_{\jord}}(\grpF)$,
it stabilizes the subspace $\jord_{\Q}^{\grpF(\Q)}$,
which is a line spanned by $\mathrm{I}$,
thus $g(\mathrm{I})$ is a non-zero multiple of $\mathrm{I}$.
So we obtain a morphism $\mathbf{C}_{\mathbf{M}_{\jord}}(\grpF)\rightarrow \mathbf{G}_{\mathrm{m}}$ by restricting to the line spanned by $\mathrm{I}$.
\begin{itemize}
    \item This morphism is injective, since the center of $\grpF$ is trivial;
    \item For any scalar $\lambda\in\Q^{\times}$, the map $X\mapsto \lambda X$ is an element of $\mathbf{C}_{\mathbf{M}_{\jord}}(\grpF)(\Q)$, 
    thus morphism is also surjective.
\end{itemize}
Hence the centralizer of $\grpF$ in the Levi subgroup $\mathbf{M}_{\jord}$ of $\mathbf{H}_{\jord}^{1}$ 
is a rank $1$ torus.

The centralizer of $\grpF$ in $\mathbf{P}_{\jord,\mathrm{sc}}=\mathbf{M}_{\jord}\mathbf{N}_{\jord}$
is generated by $\mathbf{C}_{\mathbf{M}_{\jord}}(\grpF)$ and the subgroup $\{\mathrm{n}(x\mathrm{I}),\,x\in \mathbf{G}_{\mathrm{a}}\}$ of $\mathbf{N}_{\jord}$,
and it is isomorphic to the standard Borel subgroup of $\SL_{2}$ via:
\[(X\mapsto uX)\mapsto \mat{u}{}{}{u^{-1}},\,\mathrm{n}(x\mathrm{I})\mapsto \mat{1}{x}{}{1}.\]
Similarly, the centralizer of $\grpF$ in $\overline{\mathbf{P}_{\jord,\mathrm{sc}}}$ is isomorphic to the opposite Borel subgroup of $\SL_{2}$.
It can be checked that
thse two centralizers give us a subgroup $\grpF\times\SL_{2}$ inside $\mathbf{H}_{\jord}^{1}$,
which is a maximal proper subgroup of $\mathbf{H}_{\jord}^{1}$ \cite[Lemma 2.4]{karasiewicz2023dualpairmathrmautctimesf4},
so it gives a reductive dual pair in $\mathbf{H}_{\jord}^{1}$,
and induces a dual pair $\grpF\times\GL_{2}$ (\emph{resp.}\,$\grpF\times\pgl_{2}$) inside $\mathbf{H}_{\jord}$ (\emph{resp.}\,$\grpE$).

\subsubsection{\texorpdfstring{$\spin_{9}\times \sorth_{2,2}$}{}}\label{section dual Pair involving spin9}
 By \cite[Theorem 2.7.4]{Yokota2009ExceptionalLG},
the stabilizer of $\mathrm{E}_{1}=[1,0,0\,;0,0,0]$ in $\grpF$ is isomorphic to $\spin_{9}$,
the spin group of a positive definite $9$-dimensional quadratic space.
In the sequel we refer to this group as $\spin_{9}$.
The $9$-dimensional quadratic space can be found inside $\jord_{\Q}$:
\begin{lemma}\label{lemma stable spin9 subspace of Jordan algebra}
    The group $\spin_{9}$ preserves respectively the following subspaces of $\jord_{\Q}$:
    \[\jord_{1}:=\set{[0,\xi,-\xi\,;x,0,0]}{\xi\in\Q,x\in\oct_{\Q}}\]
    and 
    \[\jord_{2}:=\set{[0,0,0\,;0,y,z]}{y,z\in\oct_{\Q}}.\]
\end{lemma} 
\begin{proof}
    Since 
    \[\jord_{1}=\set{X\in \jord_{\Q}}{\mathrm{E}_{1}\circ X=0,\tr(X)=0}\]
    and 
    \[\jord_{2}=\set{X\in \jord_{\Q}}{2\mathrm{E}_{1}\circ X=X},\]
    the lemma follows from the definition that $\spin_{9}$ is the stabilizer of $\mathrm{E}_{1}$ in $\grpF$.
\end{proof}
\begin{notation}\label{notation so22}
    In this article,
    $\sorth_{2,2}$ is defined to be the special orthogonal group 
    of a split $4$-dimensional quadratic space over $\Q$,
    and we define $\spin_{2,2},\mathbf{GSpin}_{2,2}$ similarly.
    Notice that $\mathbf{GSpin}_{2,2}\simeq \{(g_{1},g_{2})\in\GL_{2}\times\GL_{2},\det(g_{1})=\det(g_{2})\}$,
    $\spin_{2,2}\simeq \SL_{2}\times\SL_{2}$,
    and $\sorth_{2,2}\simeq \mathbf{GSpin}_{2,2}/\mathbf{G}_{\mathrm{m}}^{\Delta}\simeq \spin_{2,2}/\mu_{2}^{\Delta}$.
\end{notation}
We study first the centralizer of $\spin_{9}$ in the Levi subgroup $\mathbf{M}_{\jord}\subseteq \mathbf{H}_{\jord}^{1}$:
\begin{lemma}\label{lemma centralizer in Levi}
    The centralizer $\mathbf{C}_{\mathbf{M}_{\jord}}(\spin_{9})$ is an extension of $\mathbf{G}_{\mathrm{m}}\times\mathbf{G}_{\mathrm{m}}$ by $\mu_{2}$.
\end{lemma}
\begin{proof}
    For any element $g\in\mathbf{C}_{\mathbf{M}_{\jord}}(\spin_{9})$,
it stabilizes the subspace $\jord_{\Q}^{\spin_{9}(\Q)}$,
which is spanned by $\mathrm{E}_{1}$ and $\mathrm{I}-\mathrm{E}_{1}=\mathrm{E}_{2}+\mathrm{E}_{3}$.
The rank $1$ elements in this subspace are non-zero multiples of $\mathrm{E}_{1}$,
and the rank $2$ elements are non-zero multiples of $\mathrm{E}_{2}+\mathrm{E}_{3}$.
As elements of $\mathbf{M}_{\jord}$ preserve the rank,
$g$ acts on $\mathrm{E}_{1}$ (\emph{resp.}\,$\mathrm{E}_{2}+\mathrm{E}_{3}$) by a scalar. 
So we obtain a morphism from $\mathbf{C}_{\mathbf{M}_{\jord}}(\spin_{9})$ to $\mathbf{G}_{\mathrm{m}}\times\mathbf{G}_{\mathrm{m}}$,
whose kernel is the center of $\spin_{9}$,
a cyclic group generated by the involution $[a,b,c\,;x,y,z]\mapsto [a,b,c\,;x,-y,-z]$ \cite[\S 4.3.1]{shan2024levelautomorphicrepresentationsanisotropic}.
This morphism of algebraic groups is also surjective,
since for any non-zero scalars $\lambda,\mu$,
we have the following element in $\mathbf{C}_{\mathbf{M}_{\jord}}(\spin_{9})$:
\[m_{\lambda,\mu}:[a,b,c\,;x,y,z]\mapsto [\lambda^{-1}\mu^{2}a,\lambda b,\lambda c\,;\lambda x,\mu y,\mu z].\qedhere\]
\end{proof}
Let $\mathbf{C}^{\prime}$ be the subgroup of $\mathbf{C}_{\mathbf{M}_{\jord}}(\spin_{9})$ consisting of $m_{\lambda,\mu}$,
then we have the following commutative diagram:
\begin{equation*}
    \begin{tikzcd}
        1  \arrow[r] & \mu_{2} \arrow[r] \arrow[d,equal] & \mathbf{C}_{\mathbf{M}_{\jord}}(\spin_{9})\arrow[rr] & & \mathbf{G}_{\mathrm{m}}\times\mathbf{G}_{\mathrm{m}} \arrow[r]\arrow[d,equal] & 1\\
        1  \arrow[r] & \mu_{2} \arrow[r,"\mu\mapsto m_{1,\mu}"]           & \mathbf{C}^{\prime} \arrow[rr,"m_{\lambda,\mu}\mapsto{(\lambda^{-1}\mu^{2},\lambda)}"]\arrow[u,hook]            &  & \mathbf{G}_{\mathrm{m}}\times\mathbf{G}_{\mathrm{m}} \arrow[r]          & 1
    \end{tikzcd} ,   
\end{equation*}
which shows that $\mathbf{C}_{\mathbf{M}_{\jord}}(\spin_{9})=\mathbf{C}^{\prime}$ is a split torus of rank $2$.
The centralizer of $\spin_{9}$ in $\mathbf{P}_{\jord,\mathrm{sc}}$
is generated by $\mathbf{C}_{\mathbf{M}_{\jord}}(\spin_{9})$
and $\{\mathrm{n}(x\mathrm{E}_{1}+y(\mathrm{E}_{2}+\mathrm{E}_{3})),x,y\in\Q\}\subseteq \mathbf{N}_{\jord}$,
and it is isomorphic to the standard Borel subgroup of $\spin_{2,2}=\SL_{2}\times\SL_{2}$ via:
\begin{align}\label{eqn explicit embedding into Spin4}
    m_{\lambda,\mu}\mapsto \left(\mat{\lambda}{}{}{\lambda^{-1}},\mat{\mu}{}{}{\mu^{-1}}\right),\,\mathrm{n}(x\mathrm{E}_{1}+y(\mathrm{E}_{2}+\mathrm{E}_{3}))\mapsto\left(\mat{1}{x}{}{1},\mat{1}{y}{}{1}\right).
\end{align}
Similarly, the centralizer of $\spin_{9}$ in $\overline{\mathbf{P}_{\jord,\mathrm{sc}}}$ is isomorphic to the opposite Borel subgroup of $\spin_{2,2}$.
To obtain a morphism $\spin_{9}\times\spin_{2,2}\rightarrow \mathbf{H}_{\jord}^{1}$,
one needs to check the Steinberg relations for $\SL_{2}\times\SL_{2}$:
for example, the most complicated one is that for any $x,y\in\mathbf{G}_{\mathrm{m}}$,
\[n(x\mathrm{E}_{1}+y(\mathrm{E}_{2}+\mathrm{E}_{3}))n^{\vee}(-x^{-1}\mathrm{E}_{1}-y^{-1}(\mathrm{E}_{2}+\mathrm{E}_{3}))n(x\mathrm{E}_{1}+y(\mathrm{E}_{2}+\mathrm{E}_{3}))=m_{x,y}\iota^{-1}.\]
All these relations can be verified directly using the explicit actions of $\mathbf{P}_{\jord,\mathrm{sc}}$ and $\overline{\mathbf{P}_{\jord,\mathrm{sc}}}$
on the $56$-dimensional space $\mathrm{W}_{\jord}$.

The kernel of $\spin_{9}\times\spin_{2,2}\rightarrow \mathbf{H}_{\jord}^{1}$ is $\{(\mathrm{id},\mathrm{id}),(m_{1,-1},m_{1,-1})\}$,
and we denote by $\spin_{9}\times_{\mu_{2}}\spin_{2,2}$ the quotient of $\spin_{9}\times\spin_{2,2}$ by this kernel.
The morphism $\spin_{9}\times_{\mu_{2}}\spin_{2,2}\hookrightarrow\mathbf{H}_{\jord}^{1}$ induces an embedding of $\spin_{9}\times_{\mu_{2}}\mathbf{GSpin}_{2,2}$ (\emph{resp.}\,$\spin_{9}\times_{\mu_{2}}\sorth_{2,2}$) into $\mathbf{H}_{\jord}$ (\emph{resp.}\,$\grpE$).

The centralizer $\mathbf{C}_{\grpE}(\grpF)\simeq \pgl_{2}$ is 
embedded into $\sorth_{2,2}\subseteq\mathbf{C}_{\grpE}(\spin_{9})$ 
via the map induced from the diagonal embedding $\GL_{2}\rightarrow \mathbf{GSpin}_{2,2}$.
   
\section{Local theta correspondence}\label{section local theta correspondences}
In this section we recall some results on the minimal representation of $\grpE$
and the local theta correspondences 
for the exceptional dual pairs constructed in \Cref{section dual pair}.

\subsection{Minimal representation of $\grpE$}\label{section minimal representation}
 The theory of theta correspondences studies 
the restrictions of minimal representations to reductive dual pairs,
so we first recall the definition of the minimal representation of $\grpE(F)$ for $F=\Q_{p}$ or $\R$,
and also some properties that will be used.
\begin{defi}\label{def minimal representations}
    \begin{enumerate}[label=(\roman*)]
        \item The \emph{minimal representation $\Pi_{\min,p}$ of $\grpE(\Q_{p})$}
        is the unramified representation whose Satake parameter is the $\widehat{\grpE}(\C)$-conjugacy class of $\varphi\smallmat{p^{1/2}}{}{}{p^{-1/2}}$.
        Here the morphism $\varphi:\SL_{2}(\C)\rightarrow \widehat{\grpE}(\C)$ corresponds to the subregular unipotent orbit of $\widehat{\grpE}(\C)=\mathbf{H}_{\jord}^{1}(\C)$.
        \item Let $\Pi^{+}$ be the holomorphic representation of $\mathbf{H}_{\jord}^{1}(\R)$ with the smallest Gelfand-Kirillov dimension among non-trivial representations,
        and $\Pi^{-}$ be the anti-holomorphic representation contragradient to $\Pi^{+}$.
        The \emph{minimal representation $\Pi_{\min,\infty}$ of $\grpE(\R)$} is the unique representation whose restriction to $\mathbf{H}_{\jord}^{1}(\R)$ is $\Pi^{+}\oplus\Pi^{-}$.
    \end{enumerate} 
\end{defi}
The first property that we need is the following relation between the minimal representation and a principal series:
\begin{prop}\label{prop minimal rep as subrep of principal series}
    \cite[Proposition 6.1]{Savin1994}\cite{sahi1993unitary}
    For $v=p$ or $\infty$,
    the minimal representation $\Pi_{\min,v}$ of $\grpE(\Q_{v})$ 
    is the unique irreducible submodule of the normalized degenerate principal series 
    \[\mathrm{Ind}_{\para_{\jord}(\Q_{v})}^{\grpE(\Q_{v})}\delta_{\para_{\jord}}^{-1/2}|\lambda|^{2},\]
    where $\delta_{\para_{\jord}}$ is the modulus character of $\para_{\jord}(\Q_{v})$,
    and $\lambda:\LeviM_{\jord}(\Q_{v})\rightarrow\Q_{v}^{\times}$ is the similitude character of $\mathbf{M}_{\jord}(\Q_{v})$.  
\end{prop}
The sections of $\mathrm{Ind}_{\para_{\jord}(\Q_{p})}^{\grpE(\Q_{v})}\delta_{\para_{\jord}}^{-1/2}|\lambda|^{2}$ are 
smooth functions $f:\mathbf{P}_{\jord}(\Q_{p})\rightarrow \C$
such that
\begin{align}\label{eqn sections of Siegel parabolic}
    f(pg)=|\lambda(p)|_{p}^{2}f(g),\text{ for all }p\in\mathbf{P}_{\jord}(\Q_{p}),g\in\grpE(\Q_{p}).
\end{align}
From now on, we identify $\Pi_{\min,v}$ as the unique irreducible submodule of $\mathrm{Ind}_{\para_{\jord}(\Q_{v})}^{\grpE(\Q_{v})}\delta_{\para_{\jord}}^{-1/2}|\lambda|^{2}$,
and normalize the spherical vector $\Phi_{p}$ in $\Pi_{\min,v}$ by the condition that $\Phi_{p}(1)=1$.

The second property is the 
$\mathrm{K}_{\lietype{E}{7}}$-types of the holomorphic part $\Pi^{+}$ of $\Pi_{\min}$.
The maximal compact subgroup $\mathrm{K}_{\lietype{E}{7}}$ of $\mathbf{H}_{\jord}^{1}(\R)$ is isomorphic to $\lietype{E}{6}\times\mathrm{U}(1)$,
where $\lietype{E}{6}$ is the simply-connected compact Lie group of type $\lietype{E}{6}$.
\begin{defi}\label{def representation of maximal compact subgroup of E7}
    (1) Define $\mathrm{E}(n)$ to be the irreducible representation of the compact Lie group $\lietype{E}{6}$ with highest weight $n\lambda$,
    where $\lambda$ is the highest weight of $\mathfrak{p}_{\jord}^{+}$ as a $\lietype{E}{6}$-representation.
    \\
    (2) For $n,k\in \mathbb{N}$, 
    define $\mathrm{E}(n,k)$ to be the irreducible representation of $\mathrm{K}_{\lietype{E}{7}}$ such that 
    its restriction to $\lietype{E}{6}$ is isomorphic to $\mathrm{E}(n)$
    and its restriction to $\mathrm{U}(1)$ is the character $z\mapsto z^{k}$.    
\end{defi}
The restriction of $\Pi^{+}$ to $\mathrm{K}_{\lietype{E}{7}}$ is given in \cite{WallachMinrep}:
\begin{align}\label{eqn minimal K type of minimal representation}
    \Pi^{+}|_{\mathrm{K}_{\lietype{E}{7}}}\simeq\bigoplus_{n=0}^{\infty}\mathrm{E}(n,2n+12).
\end{align}

\subsection{\texorpdfstring{$p$}{PDFstring}-adic correspondence for \texorpdfstring{$\grpF\times\pgl_{2}$}{PDFstring}}\label{section p-adic correspondence of F4 and PGL(2)}
 Over $\Q_{p}$,
the exceptional theta correspondence for $\grpF\times\pgl_{2}$ 
has been studied in \cites{Savin1994}{karasiewicz2023dualpairmathrmautctimesf4}.
Now we recall some results that we need.
\begin{defi}\label{def p-adic theta}
    Let $\pi$ be a smooth irreducible representation of $\pgl_{2}(\Q_{p})$,
    then the maximal $\pi$-isotypic quotient of $\Pi_{\min,p}$ admits an action of $\grpF(\Q_{p})$
    and factors as $\pi\boxtimes\Theta(\pi)$.
    We call $\Theta(\pi)$ the \emph{big theta lift} of $\pi$,
    and its maximal semisimple quotient $\theta(\pi)$
    the \emph{small theta lift} of $\pi$.
\end{defi}

Let $\mathbf{B}_{0}=\mathbf{T}_{0}\mathbf{N}_{0}$ be the Borel subgroup of $\pgl_{2}$ 
consisting of upper triangular matrices,
and $\overline{\mathbf{B}_{0}}$ be the opposite Borel subgroup.
Let $\chi$ be a character of $\mathbf{T}_{0}(\Q_{p})=\{\smallmat{t}{}{}{1},\,t\in\Q_{p}^{\times}\}$
satisfying $\chi=|-|^{s}\cdot\chi_{0}$,
where $s\geq 0$ and $\chi_{0}$ is a unitary character of $\mathbf{T}_{0}(\Q_{p})$. 
When $s\neq \frac{1}{2}$ or $\chi_{0}^{2}\neq 1$,
the principal series $\mathrm{Ind}_{\overline{\mathbf{B}_{0}}(\Q_{p})}^{\pgl_{2}(\Q_{p})}(\chi)$ is irreducible.
It turns out the theta lift of this principal series to $\grpF(\Q_{p})$ is also a principal series.
Before stating the result of Karasiewicz-Savin,
we introduce a maximal parabolic subgroup of $\grpF$.
\begin{defi}\label{def max parabolic of F4}
    Using Bourbaki's labeling for simple roots of $\lietype{F}{4}$,
    we define $\mathbf{Q}$ to be the maximal parabolic subgroup of $\grpF$
    obtained by removing $\alpha_{4}$ from the Dynkin diagram.
\end{defi}
The Levi subgroup of $\mathbf{Q}$ is isomorphic to $\mathbf{GSpin}_{7}$,
whose similitude map $\mathbf{GSpin}_{7}\rightarrow \mathbf{GL}_{1}$ is given by the fundamental weight $\varpi_{4}$. 
Notice that $\widehat{\mathbf{Q}}\simeq\mathbf{GSp}_{6}\simeq\left(\symp_{6}\times\mathbf{G}_{\mathrm{m}}\right)/\mu_{2}^{\Delta}$.
\begin{prop}\label{prop p-adic lift to F4}
    \cite[Proposition 6.4]{karasiewicz2023dualpairmathrmautctimesf4}
    Let $\chi=|-|^{s}\cdot\chi_{0}$ be a character of $\mathbf{T}_{0}(\Q_{p})$ such that $\chi_{0}$ is unitary, $s\geq 0$ and $s\neq \frac{1}{2},\frac{5}{2},\frac{11}{2}$,
    then the big theta lift of $\mathrm{Ind}_{\overline{\mathbf{B}_{0}}(\Q_{p})}^{\pgl_{2}(\Q_{p})}(\chi)$ to $\grpF(\Q_{p})$ is irreducible,
    and 
    \[\Theta(\mathrm{Ind}_{\overline{\mathbf{B}_{0}}(\Q_{p})}^{\pgl_{2}(\Q_{p})}(\chi))=\theta(\mathrm{Ind}_{\overline{\mathbf{B}_{0}}(\Q_{p})}^{\pgl_{2}(\Q_{p})}(\chi))\simeq \mathrm{Ind}_{\mathbf{Q}(\Q_{p})}^{\grpF(\Q_{p})}(\chi\circ \varpi_{4}).\]
\end{prop}
\begin{rmk}\label{rmk Satake parameter}
    If $\chi$ is unramified,
    then \Cref{prop p-adic lift to F4}
    tells that the Satake parameter of $\theta(\mathrm{Ind}_{\overline{\mathbf{B}_{0}}(\Q_{p})}^{\pgl_{2}(\Q_{p})}(\chi))$ is 
    the $\widehat{\grpF}(\C)$-conjugacy class of the image of $(e_{p},c_{p})$ under the map 
    \[\SL_{2}\times\SL_{2}\xrightarrow{(\mathrm{principal\,embedding},\mathrm{id})}\left(\symp_{6}\times\SL_{2}\right)/\mu_{2}^{\Delta}\hookrightarrow\widehat{\grpF},\]
    where $c_{p}=\mathrm{diag}(\chi(p),\chi(p)^{-1})$
    and $e_{p}=\mathrm{diag}(p^{1/2},p^{-1/2})$.
\end{rmk}
For twisted Steinberg representations and supercuspidal representation of $\pgl_{2}(\Q_{p})$,
we also have the following result:
\begin{prop}
    \label{prop p-adic supercuspidal and Steinberg}
    \cite[Theorem 6.3, Proposition 6.5]{karasiewicz2023dualpairmathrmautctimesf4}
    If a smooth irreducible representation $\pi$ is a supercuspidal representation or a quadratic twist of Steinberg,
    then $\Theta(\pi)=\theta(\pi)$ is a non-zero irreducible representation of $\grpF(\Q_{p})$.
\end{prop}

\subsection{Archimedean theta correspondence}\label{section theta correspondence Lie groups}
 For the dual pair $\grpF(\R)\times \pgl_{2}(\R)$ inside $\grpE(\R)$, 
we have the following result:
\begin{prop}\label{prop real theta between F4 and PGL2}\cite[Proposition 3.2]{gross_savin_1998}
    The restriction of $\Pi_{\min,\infty}$ to $\grpF(\R)\times \pgl_{2}(\R)$ is isomorphic to
    \[\bigoplus_{n\geq 0}\vrep{n\varpi_{4}}\boxtimes \mathcal{D}(2n+12),\]
    where $\vrep{n\varpi_{4}}$ is the irreducible representation of $\grpF(\R)$ with highest weight $n\varpi_{4}$,
    and $\mathcal{D}(m)$ is the unitary completion of 
    $d_{hol}(m)\oplus d_{anti\text{-}holo}(m)$,
    $d_{hol}(m)$ being the holomorphic discrete series representation of $\SL_{2}(\R)$ with minimal $\sorth_{2}(\R)$ type $m$
    and $d_{anti\text{-}holo}(m)$ being its contragradient.
\end{prop}
Before stating the result for $\spin_{9}\times\sorth_{2,2}$,
we define some notations for $\spin_{9}(\R)$.
\begin{notation}\label{notation Spin9 representation}
    Let $\lambda_{1}$ be the highest weight of the standard $9$-dimensional representation of $\spin_{9}(\R)$,
    and $\lambda_{2}$ that of the $16$-dimensional spinor representation.
    Denote by $\mathrm{U}_{m,n}$ the irreducible representation of $\spin_{9}(\R)$ with highest weight $m\lambda_{1}+n\lambda_{2}$.
\end{notation}
\begin{prop}\label{prop archimedean siegel weil formula}
    The restriction of $\Pi_{\min,\infty}$ to $\spin_{9}(\R)\times\sorth_{2,2}(\R)$ is isomorphic to
    \[\bigoplus_{m,n\geq 0}\mathrm{U}_{m,n}\boxtimes \mathcal{D}(n+4)\boxtimes \mathcal{D}(2m+n+8),\]
    where we view $\mathcal{D}(n+4)\boxtimes\mathcal{D}(2m+n+8)$ as a representation of $\sorth_{2,2}(\R)$.
\end{prop}
\begin{proof}
The proof is parallel to the argument in \cite[\S 3]{gross_savin_1998} for $\mathbf{G}_{2}\times \mathbf{PGSp}_{6}$,
using the branching laws in \cite{Lepowsky1970}.
\end{proof}

\section{Global theta correspondence}\label{section automorphic representations and global theta lift from min rep}
 In this section, we recall 
an automorphic realization of the minimal representation of $\grpE(\A)$,
and then use it to define global theta lifts.
\subsection{Automorphic forms}\label{section algebraic modular forms}
 Let $\mathbf{G}$ be a connected reductive group over $\Q$
which admits a (reductive) $\Z$-model $\G$,
in the sense of \cite{G96}.
Let $\widehat{\Z}=\prod_{p}\Z_{p}$,
$\A_{f}=\widehat{\Z}\otimes\Q$,
and $\A=\R\times\A_{f}$.
We fix a maximal compact subgroup $K_{\infty}$ of $\mathbf{G}(\R)$
and let $\mathfrak{g}=\C\otimes_{\R}\mathrm{Lie}(\mathbf{G}(\R))$.

For simplicity, we assume that the center of $\mathbf{G}$ is anisotropic,
and denote the quotient space $\mathbf{G}(\Q)\backslash\mathbf{G}(\A)$ by $[\mathbf{G}]$,
equipped with a right invariant finite Haar measure $\mu$.
Let $\mathcal{A}(\mathbf{G})$ and $\mathcal{A}_{\cusp}(\mathbf{G})$ be the space of (smooth) automorphic forms and cusp forms on $\mathbf{G}$ respectively,
as subspaces of the space of smooth functions of moderate growth on $[\mathbf{G}]$.
The topological group $\grpG(\A)$ acts on $\mathcal{A}(\mathbf{G})$ and $\mathcal{A}_{\cusp}(\mathbf{G})$ by right translation.

We also recall the $\mathrm{L}^{2}$-setting of automorphic representations.
\begin{defi}\label{def automorphic forms and representations}
    (1) An irreducible unitary representation $\pi$ of $\mathbf{G}(\A)$ is \emph{(square-integrable) discrete automorphic} in the sense of \cite[\S 4.6]{borel1979automorphic},
    if $\pi$ is isomorphic to a $\mathbf{G}(\A)$-invariant closed subspace of $\mathrm{L}^{2}([\mathbf{G}])$.
    We denote by
     $\Pi_{\disc}(\mathbf{G})$  the set of equivalence classes of discrete automorphic representations of $\mathbf{G}$,
    and by $\mathrm{L}^{2}_{\disc}([\mathbf{G}])$ the discrete part of $\mathrm{L}^{2}([\mathbf{G}])$.
    \\
    (2) Let $\mathrm{L}_{\cusp}^{2}([\mathbf{G}])$ be the subspace of $\mathrm{L}_{\disc}^{2}([\mathbf{G}])$ consisting of the classes of cuspidal functions.
    A discrete automorphic representation of $\grpG$ is \emph{cuspidal} if it is a subrepresentation of $\sqint{[\grpG]}{\cusp}$.
    Denote by $\Pi_{\cusp}(\grpG)$ the subset of $\Pi_{\disc}(\grpG)$ consisting of cuspidal representations.
    \\
    (3) An irreducible unitary representation $\pi$ of $\grpG(\A)$ has \emph{level one} if $\pi$ can be decomposed as $\pi=\pi_{\infty}\otimes\pi_{f}$,
    where $\pi_{\infty}$ is an irreducible unitary representation of $\grpG(\R)$
    and $\pi_{f}$ is a smooth irreducible representation of $\grpG(\A_{f})$ such that $\pi_{f}^{\G(\widehat{\Z})}\neq 0$.
    We denote the subset of $\Pi_{\disc}(\grpG)$ (\emph{resp.\,}$\Pi_{\cusp}(\mathbf{G})$) consisting of those with level one by $\Pi_{\disc}^{\mathrm{unr}}(\grpG)$ (\emph{resp.\,}$\Pi_{\disc}^{\mathrm{unr}}(\mathbf{G})$).
\end{defi}
\subsubsection{Automorphic forms of \texorpdfstring{$\grpF$}{PDFstring}}
\label{section automorphic forms of F4}
Now we concentrate on the level one automorphic forms of $\grpF$,
and describe them in a manner similar to the case for orthogonal groups \cite[\S 4.4]{ChenevierLannes}.
The adelic quotient $[\grpF]$ is compact,
so $\mathrm{L}^{2}([\grpF])=\mathrm{L}^{2}_{\disc}([\grpF])=\mathrm{L}_{\cusp}^{2}([\grpF])$,
and every automorphic representation of $\grpF$ is discrete and cuspidal.

A level one automorphic representation of $\grpF$ is generated by some automorphic form $\varphi\in\mathcal{A}(\grpF)^{\mathcal{F}_{4,\mathrm{I}}(\widehat{\Z})}\subseteq \mathrm{L}^{2}([\grpF])^{\mathcal{F}_{4,\mathrm{I}}(\widehat{\Z})}$.
The latter space can be viewed as 
the space of square-integrable functions on $\grpF(\Q)\backslash\grpF(\A)/\mathcal{F}_{4,\mathrm{I}}(\widehat{\Z})$,
endowed with the Radon measure that is the image of $\mu$ by the canonical map $[\grpF]\rightarrow \grpF(\Q)\backslash\grpF(\A)/\mathcal{F}_{4,\mathrm{I}}(\widehat{\Z})$.
By the Peter-Weyl theorem,
$\mathrm{L}^{2}([\grpF])^{\mathcal{F}_{4,\mathrm{I}}(\widehat{\Z})}$ can be decomposed into a direct sum of irreducible representations (note that every irreducible representation of $\grpF(\R)$ is self-dual):
\begin{lemma}\label{lemma Peter Weyl applied to F4}
    Denote by $\mathrm{Irr}(\grpF(\R))$ the set of equivalence classes of irreducible representations of $\grpF(\R)$,
then we have:
\[\mathrm{L}^{2}([\grpF])^{\mathcal{F}_{4,\mathrm{I}}(\widehat{\Z})}\simeq \overline{\bigoplus_{V\in\mathrm{Irr}(\grpF(\R))}}V\otimes \mathcal{A}_{V}(\grpF),\]
where $\mathcal{A}_{V}(\grpF)$ is defined as
\begin{align}\label{eqn definition vector valued automorphic form}
    \set{f:\grpF(\Q)\backslash\grpF(\A)/\mathcal{F}_{4,\mathrm{I}}(\widehat{\Z})\rightarrow V}{f(gh)=h^{-1}.f(g),\text{ for any }g\in\grpF(\A),h\in \grpF(\R)}.
\end{align}
\end{lemma}
Under this isomorphism, 
an automorphic form $\varphi\in\mathcal{A}(\grpF)^{\mathcal{F}_{4,\mathrm{I}}(\widehat{\Z})}$
is identified with an element of $\bigoplus\limits_{V\in\mathrm{Irr}(\grpF(\R))}V\otimes \mathcal{A}_{V}(\grpF)$.
The number of $\pi\in\Pi_{\disc}^{\mathrm{unr}}(\grpF)$ such that $\pi_{\infty}\simeq V$,
counted with multiplicities,
is exactly $\dim\mathcal{A}_{V}(\grpF)$,
which is computed explicitly in \cite{shan2024levelautomorphicrepresentationsanisotropic}.

Using \Cref{prop reformulation of double coset},
we identify $\grpF(\Q)\backslash\grpF(\A)/\mathcal{F}_{4,\mathrm{I}}(\widehat{\Z})$ with the set $\mathcal{J}$ of Albert lattices,
and equip $\mathcal{J}$ with the corresponding right $\grpF(\R)$-invariant Radon measure.
We can thus identify $\mathrm{L}^{2}([\grpF])^{\mathcal{F}_{4,\mathrm{I}}(\widehat{\Z})}$ 
with $\mathrm{L}^{2}(\mathcal{J})$,
equipped with the induced $\grpF(\R)$ action: 
\[(g.f)(J)=f(g^{-1}J),\text{ for any }g\in \grpF(\R),J\in\mathcal{J},\]
and identify $\mathcal{A}_{V}(\grpF)$
with the space 
\[\set{f:\mathcal{J}\rightarrow V}{f(gJ)=g.f(J),\text{ for any }g\in \grpF(\R),J\in\mathcal{J}}.\]
We will use either of these two formulations of $\mathcal{A}_{V}(\grpF)$,
depending on convenience.

A function $f\in\mathcal{A}_{V}(\grpF)$
is determined by its values on the set of representatives $\{1,\gamma_{\mathrm{E}}\}$ for 
$\grpF(\Q)\backslash \grpF(\A_{f})/\mathcal{F}_{4,\mathrm{I}}(\widehat{\Z})$ chosen in \Cref{notation representatives of double cosets}. 
Furthermore, we have:
\begin{lemma}\label{lemma identification of vector valued automorphic form}
    The evaluation map $f\mapsto (f(1),f(\gamma_{\mathrm{E}}))$ (or equivalently $f\mapsto (f(\jord_{\Z}),f(\jord_{\mathrm{E}}))$)
    induces an isomorphism of vector spaces:
    \begin{align*}
        \mathrm{M}_{V}(\grpF)\simeq V^{\Gamma_{\mathrm{I}}}\oplus V^{\Gamma_{\mathrm{E}}},
    \end{align*}
    where $\Gamma_{\mathrm{I}}=\mathcal{F}_{4,\mathrm{I}}(\Z)$ is the automorphism group of the Albert algebra $\jord_{\Z}$,
    and $\Gamma_{\mathrm{E}}$ is that of $\jord_{\mathrm{E}}$.
\end{lemma}

 
\subsubsection{A polynomial model of \texorpdfstring{$\vrep{n\varpi_{4}}$}{}}\label{section harmonic polynomials representations of F4}
In this paper, 
we focus on automorphic representations of $\grpF$
with archimedean component $V=\vrep{n\varpi_{4}}$.
Now we give a polynomial model of this family of irreducible representations.

When $n=1$, a natural model for the $26$-dimensional representation $\vrep{\varpi_{4}}$ 
is the trace $0$ part of $\jord_{\C}\simeq\mathfrak{p}_{\jord}^{-}$.
We choose the realization dual to this one,
\emph{i.e.} the subspace of $\mathrm{P}_{1}(\jord_{\C})\simeq \mathfrak{p}_{\jord}^{+}$ 
consisting of linear functions $\ell$ on $\jord_{\C}$ such that $\ell(\mathrm{I})=0$.

For $n\geq 1$, 
$\vrep{n\varpi_{4}}$ is a subrepresentation of $\sym^{n}\vrep{\varpi_{4}}\subseteq \sym^{n}\mathfrak{p}_{\jord}^{+}=\mathrm{P}_{n}(\jord_{\C})$,
where the action of $\grpF(\R)$ on $\mathrm{P}_{n}(\jord_{\C})$ is given as:
\[(g.P)(X)=P(g^{-1}x),\text{ for any }g\in\grpF(\R),P\in\mathrm{P}_{n}(\jord_{\C})\text{ and }X\in\jord_{\C}.\]
\begin{defi}\label{def F4 harmonic polynomials}
    Define $\mathbb{X}$ to be the following $\grpF(\C)$-orbit in $\jord_{\C}$:
    \[\mathbb{X}:=\set{A\in\jord_{\C}}{\tr(A)=0,\mathrm{rank}(A)=1}=\set{A\in\jord_{\C}}{A\neq 0,\tr(A)=0,\mathrm{rank}(A)=1}.\]
    For any $n\geq 1$,
    we define $\vrep{n}(\jord_{\C})$ to be the subspace of $\mathrm{P}_{n}(\jord_{\C})$ 
    spanned by polynomials of the form 
    $X\mapsto \left(\tr\left(X\circ A\right)\right)^{n}$, $A\in \mathbb{X}$.
\end{defi}
\begin{lemma}\label{lemma polynomial model of F4 representation}
    For any $n\geq 1$,
    $\vrep{n}(\jord_{\C})$ is an irreducible representation of $\grpF(\R)$,
    and its highest weight is $n\varpi_{4}$.
\end{lemma}
\begin{proof}
    This lemma follows from the fact that 
    $\mathbb{X}$ is the set of highest vectors in the irreducible $\grpF(\R)$-representation 
    $\{A\in\jord_{\C},\tr(A)=0\}\simeq\vrep{\varpi_{4}}$,
    and $\grpF(\R)$ acts on it transitively.
\end{proof}

\subsection{Automorphic realization of minimal representation}\label{section automorphic realization of min rep}
 Let $\Pi_{\min}=\otimes^{\prime}_{v}\Pi_{\min,v}$ be the (adelic) minimal representation of $\grpE(\A)$.
To establish the global theta correspondence for dual pairs inside $\grpE$, 
we need to choose an automorphic realization of $\Pi_{\min}$,
\emph{i.e.}\,an $\grpE(\A)$-equivariant embedding $\theta:\Pi_{\min}\hookrightarrow \sqint{[\grpE]}{}$.
In this section, 
we follow \cite[\S 6]{kim_yamauchi_2016} to give $\theta$ via an explicit modular form constructed by Kim in \cite{Kim1993}.
\subsubsection{Exceptional modular forms}\label{section exceptional modular forms}
\begin{defi}\label{def exceptional tube domain}
    The \emph{exceptional tube domain} $\mathcal{H}_{\jord}$ of complex dimension $27$ is the open subset of $\jord_{\C}=\jord_{\R}+i\jord_{\R}$
    consisting of $Z=X+iY$ with $Y$ positive definite.
\end{defi}
For any element $Z\in \jord_{\C}$,
set $\mathrm{r}_{1}(Z):=\left(Z,\det(Z),Z^{\#},1\right)\in \mathrm{W}_{\jord}\otimes\C$.
By \cite[Proposition 2.3.1]{PolFourierExpansion}, 
for any $g\in \mathbf{H}_{\jord}^{1}(\R)$ and $Z\in \mathcal{H}_{\jord}$,
there exist a unique scalar $J(g,Z)\in\C^{\times}$,
which is called the \emph{automorphy factor} for $\mathbf{H}_{\jord}^{1}(\R)$,
and a unique $Z^{\prime}\in \mathcal{H}_{\jord}$ such that
\begin{align*}
    g.\mathrm{r}_{1}(Z)=J(g,Z)\mathrm{r}_{1}(Z^{\prime}).
\end{align*}
\begin{defi}\label{def tube domain action}
    The action of $\mathbf{H}_{\jord}^{1}(\R)$-action on $\mathcal{H}_{\jord}$ is defined as follows:
    for $g\in \mathbf{H}_{\jord}^{1}(\R)$ and $Z\in\mathcal{H}_{\jord}$,
    $g.Z$ is the unique $Z^{\prime}\in\mathcal{H}_{\jord}$ satisfying $g.\mathrm{r}_{1}(Z)\in\C^{\times}\mathrm{r}_{1}(Z^{\prime})$.
\end{defi}
\begin{ex}\label{ex explicit actions on exceptional tube domain}
    We list the actions of some elements in $\mathbf{H}_{\jord}^{1}(\R)$:
    \begin{itemize}
        \item For $\mathrm{n}(A)\in \mathbf{N}_{\jord}(\R)$, $\mathrm{n}(A).Z=Z+A$ and $J(\mathrm{n}(A),Z)=1$;
        \item For $m\in\mathbf{M}_{\jord}(\R)$, $m.(X+iY)=\lambda(m)(m^{*}(X)+im^{*}(Y))$ and $J(m,Z)=\lambda(m)^{-1}$;
        \item For $\iota$ defined by \eqref{eqn action of order 4 element in E7}, $\iota.Z=-Z^{-1}$ and $J(\iota,Z)=\det(Z)$.
    \end{itemize}
\end{ex}
The center $\pm1 \simeq \langle \iota^{2}\rangle$ of $\mathbf{H}_{\jord}^{1}(\R)$ acts trivially on $\mathcal{H}_{\jord}$,
and the group of holomorphic transformations of $\mathcal{H}_{\jord}$ is 
$\mathbf{H}_{\jord}^{1}(\R)/\pm1$,
the connected component of $\grpE(\R)$.
\begin{defi}\label{def holomorphic modular form for E7} 
    A holomorphic function $F:\mathcal{H}_{\jord}\rightarrow \C$ is a \emph{modular form of level $1$ and weight $k$}
    if for any $Z\in\mathcal{H}_{\jord}$ and $\gamma\in \mathbf{H}_{\jord}^{1}(\Z)$ we have 
    \begin{align*}
        F(\gamma.Z)=J(\gamma,Z)^{k}\cdot F(Z).
    \end{align*}    
\end{defi}
Kim's modular form $\mathrm{F}_{Kim}$ is defined by the following Fourier expansion: 
\begin{align}\label{eqn Fourier expansion of Kim modular form}
    \mathrm{F}_{Kim}(Z):=1+240\sum_{\substack{\jord_{\Z}\ni T\geq 0,\\ \mathrm{rank}(T)=1}}\sigma_{3}\left(\mathrm{c}_{\jord_{\Z}}(T)\right)e^{2\pi i (T,Z)},\text{ for any }Z\in\mathcal{H}_{\jord},
\end{align}
where $\mathrm{c}_{\jord_{\Z}}(T)$ is the content of $T$,
\emph{i.e.}\,the largest integer $c$ such that $T/c\in \jord_{\Z}$,
and $\sigma_{3}(n)=\sum_{d|n}d^{3}$.
The function $\mathrm{F}_{Kim}$ defined by \eqref{eqn Fourier expansion of Kim modular form} 
is a modular form of level $1$ and weight $4$.

\subsubsection{Kim's automorphic form}\label{section automorphic forms E7}
Kim's modular form $\mathrm{F}_{Kim}$ gives rise to a level one automorphic form of $\grpE$.
Using the strong approximation property of $\mathbf{H}_{\jord}^{1}$,
and the homeomorphism 
\[\grpE(\Q)\backslash \grpE(\A)/\grpE(\widehat{\Z})\simeq \mathbf{H}_{\jord}^{1}(\Q)\backslash \mathbf{H}_{\jord}^{1}(\A)/\mathbf{H}_{\jord}^{1}(\widehat{\Z}),\]
we can write any element $g\in\grpE(\A)$ as
$g=g_{\Q}g_{\infty}g_{\widehat{\Z}}$,
where $g_{\Q}\in\grpE(\Q)$,
$g_{\widehat{\Z}}\in\grpE(\widehat{\Z})$
and $g_{\infty}\in \grpE(\R)$ is the image of an element in $\mathbf{H}_{\jord}^{1}(\R)$ under the projection $\mathbf{H}_{\jord}(\R)\rightarrow\grpE(\R)$. 
In other words,
$g_{\infty}$ is an element of $\mathbf{H}_{\jord}^{1}(\R)/\pm 1$,
the group of holomorphic automorphisms of $\mathcal{H}_{\jord}$.
Now for $g=g_{\Q}g_{\infty}g_{\widehat{\Z}}\in\grpE(\A)$, 
we define
\[\Theta_{Kim}(g):=J(g_{\infty},i\mathrm{I})^{-4}\cdot\mathrm{F}_{Kim}(g_{\infty}.i\mathrm{I}),\]
which is an automorphic form of $\grpE$.
Using the explicit action on $\mathcal{H}_{\jord}$ given in \Cref{ex explicit actions on exceptional tube domain},
one gets the following:
\begin{lemma}\label{lemma invariance of Kim modular form}
    The automorphic form $\Theta_{Kim}\in\mathcal{A}(\grpE)$ is invariant under $\grpF(\R)\times\grpE(\widehat{\Z})$.
\end{lemma}
Now we use $\Theta_{Kim}$ to embed $\Pi_{\min}$ into $\mathrm{L}^{2}([\grpE])$:
\begin{defi}\label{def automorphic realization theta via Kim modular form}
    Let $\Phi_{p}\in\Pi_{\min,p}$ be the normalized spherical vector,
    $\Phi_{\infty}\in\Pi^{+}\subseteq\Pi_{\min,\infty}$ the unique (up to scalar) holomorphic vector with the minimal $\mathrm{K}_{\lietype{E}{7}}$-type,
    and $\Phi_{0}:=\Phi_{\infty}\otimes \Phi_{f}=\otimes_{v}\Phi_{v}\in\Pi_{\min}$.
    The automorphic realization $\theta:\Pi_{\min}\hookrightarrow \sqint{[\grpE]}{}$ is defined to be the unique $\grpE(\A)$-equivariant map sending $\Phi_{0}$ to $\Theta_{Kim}$.
\end{defi}

\subsubsection{Constructing automorphic forms with non-minimal \texorpdfstring{$K_{\lietype{E}{7}}$}{PDFstring}-types}\label{section schimid operator and other K types}
 
The holomorphic vector $\Phi_{\infty}$ lies in the minimal $\mathrm{K}_{\lietype{E}{7}}$-type of $\Pi^{+}\subseteq\Pi_{\min,\infty}$,
and we follow the method in \cite{PolFourierExpansion} to produce automorphic forms with higher $\mathrm{K}_{\lietype{E}{7}}$-types.

For the two summands $\mathfrak{p}_{\jord}^{\pm}$ in the Cartan decomposition \eqref{eqn Cartan decomposition} of $\mathfrak{e}_{7}$,
choose a basis $\{\mathrm{X}_{\alpha}\}_{\alpha}$ of $\mathfrak{p}_{\jord}^{+}$ 
and its dual basis $\{\mathrm{X}_{\alpha}^{\vee}\}_{\alpha}$ of $\mathfrak{p}_{\jord}^{-}$
with respect to $\mathfrak{p}_{\jord}^{+}\times\mathfrak{p}_{\jord}^{-}\simeq \jord_{\C}^{\vee}\times\jord_{\C}\xrightarrow{\left\{-,-\right\}}\C$.
\begin{defi}\label{def derivative of automorphic form}
    We define a linear differential operator $\mathrm{D}:\mathcal{A}(\grpE)\rightarrow \mathcal{A}(\grpE)\otimes \mathfrak{p}_{\jord}^{-}$ by 
    \[\mathrm{D}\varphi(g):=\sum_{\alpha}(\mathrm{X}_{\alpha}\varphi)(g) \otimes \mathrm{X}_{\alpha}^{\vee},\text{ for every }\varphi\in\mathcal{A}(\grpE),\]
    which is independent of the choice of $\left\{\mathrm{X}_{\alpha}\right\}_{\alpha}$.
    For any integer $n\geq 0$, set
    $\mathrm{D}^{n}$ to be the $n$-times composition of $\mathrm{D}$. 
\end{defi}
Applying the differential operator $\mathrm{D}^{n}$ defined in \Cref{def derivative of automorphic form} to $\Theta_{Kim}$,
we obtain 
\[\Theta_{n}:=\mathrm{D}^{n}\Theta_{Kim}\in \mathcal{A}(\grpE)\otimes(\mathfrak{p}_{\jord}^{-})^{\otimes n},\]
whose coordinates belong to the $\mathrm{K}_{\lietype{E}{7}}$-type $\mathrm{E}(n,2n+12)$ in \eqref{eqn minimal K type of minimal representation}.
\begin{notation}\label{notation expansion of derivative Kim form}
    \begin{enumerate}[label=(\arabic*)]
        \item For any Albert lattice $J\in\mathcal{J}$,
        denote by $J^{+}$ the set of rank $1$ and positive semi-definite elements in $J$,
        and set $\mathrm{a}_{J}(T):=\sigma_{3}(\mathrm{c}_{J}(T))$ for any $T\in J$, 
        where $\mathrm{c}_{J}(T)$ is the content of $T$ in $J$.
        \item For any element $T\in\jord_{\R}$,
        denote by $h_{T}$ the function:  
        \[\mathbf{H}_{\jord}^{1}(\R)\rightarrow\C,\,g_{\infty}\mapsto J(g_{\infty},i\mathrm{I})^{-4}\cdot e^{2\pi i (T,g_{\infty}.i\mathrm{I})}.\]
    \end{enumerate}
\end{notation}
With these notations, 
for any $n\geq 1$,
we rewrite $\Theta_{n}$ as:
\begin{align}\label{eqn Fourier expansion of higher kernel function}
    \Theta_{n}(g)=\mathrm{D}^{n}h_{0}(g_{\infty})+240\sum_{T\in\jord_{\Z}^{+}}\mathrm{a}_{\jord_{\Z}}(T)\cdot \mathrm{D}^{n}h_{T}(g_{\infty}),
\end{align}
where $g=g_{\Q}g_{\infty}g_{\widehat{\Z}}$ as in \Cref{section automorphic forms E7}.

We end this section by the following property of $\Theta_{n}$:
\begin{lemma}\label{lemma F4 action on kernel function}
    For any $g_{\infty}\in\mathbf{H}_{\jord}^{1}(\R)$ and $h_{\infty}\in\grpF(\R)$,
    we have $\Theta_{n}(g_{\infty}h_{\infty})=h_{\infty}^{-1}.\Theta_{n}(g_{\infty})$,
    where the action of $h_{\infty}^{-1}$ is applied on $(\mathfrak{p}_{\jord}^{-})^{\otimes n}$.
\end{lemma}
\begin{proof}
    By the definition of $\Theta_{n}=\mathrm{D}^n\Theta_{Kim}$,
    we have:
    {\footnotesize\begin{align*}
        \Theta_{n}(g_{\infty}h_{\infty})=&\sum_{\alpha_{1},\dots,\alpha_{n}}(\mathrm{X}_{\alpha_{n}}\cdots \mathrm{X}_{\alpha_{1}}\Theta_{Kim})(g_{\infty}h_{\infty})\otimes \mathrm{X}_{\alpha_{1}}^{\vee}\otimes\cdots\otimes \mathrm{X}_{\alpha_{n}}^{\vee}\\
        =&\sum_{\alpha_{1},\dots,\alpha_{n}}\frac{d}{dt_{n}}\bigg{|}_{t_{n}=0}\cdots\frac{d}{dt_{1}}\bigg{|}_{t_{1}=0}\Theta_{Kim}(g_{\infty}h_{\infty}e^{t_{n}\mathrm{X}_{\alpha_{n}}}\cdots e^{t_{1}\mathrm{X}_{\alpha_{1}}})\otimes \mathrm{X}_{\alpha_{1}}^{\vee}\otimes\cdots\otimes \mathrm{X}_{\alpha_{n}}^{\vee}\\
        =&\sum_{\alpha_{1},\dots,\alpha_{n}}\frac{d}{dt_{n}}\bigg{|}_{t_{n}=0}\cdots\frac{d}{dt_{1}}\bigg{|}_{t_{1}=0}\Theta_{Kim}(g_{\infty}e^{t_{n}\mathrm{Ad}(h_{\infty})\mathrm{X}_{\alpha_{n}}}\cdots e^{t_{1}\mathrm{Ad}(h_{\infty})\mathrm{X}_{\alpha_{1}}}h_{\infty})\otimes \mathrm{X}_{\alpha_{1}}^{\vee}\otimes\cdots\otimes \mathrm{X}_{\alpha_{n}}^{\vee}\\
        =&\sum_{\alpha_{1},\dots,\alpha_{n}}\frac{d}{dt_{n}}\bigg{|}_{t_{n}=0}\cdots\frac{d}{dt_{1}}\bigg{|}_{t_{1}=0}\Theta_{Kim}(g_{\infty}e^{t_{n}h_{\infty}.\mathrm{X}_{\alpha_{n}}}\cdots e^{t_{1}h_{\infty}.\mathrm{X}_{\alpha_{1}}})\otimes \mathrm{X}_{\alpha_{1}}^{\vee}\otimes\cdots\otimes \mathrm{X}_{\alpha_{n}}^{\vee},
    \end{align*}}
    where $h_{\infty}.\mathrm{X}_{\alpha}=\mathrm{Ad}(h_{\infty})\mathrm{X}_{\alpha}$
    and the last equality follows from \Cref{lemma invariance of Kim modular form}. 
    Since $\grpF(\R)$ is a subgroup of the maximal compact subgroup $\mathrm{K}_{\lietype{E}{7}}$ of $\grpE(\R)$,
    $\{h_{\infty}.\mathrm{X}_{\alpha}\}_{\alpha}$
    also gives a basis of $\mathfrak{p}_{\jord}^{+}$,
    and its dual basis of $\mathfrak{p}_{\jord}^{-}$ is $\{h_{\infty}.\mathrm{X}_{\alpha}^{\vee}\}_{\alpha}$.  
    As the differential operator $\mathrm{D}$ is independent of the choice of $\{\mathrm{X}_{\alpha}\}_{\alpha}$,
    we have:
    \[\Theta_{n}(g_{\infty}h_{\infty})=\sum_{\alpha_{1},\ldots,\alpha_{n}}\left(\mathrm{X}_{\alpha_{n}}\cdots\mathrm{X}_{\alpha_{1}}\Theta_{Kim}\right)(g_{\infty})\otimes h_{\infty}^{-1}.\mathrm{X}_{\alpha_{1}}^{\vee}\otimes\cdots\otimes h_{\infty}^{-1}.\mathrm{X}_{\alpha_{n}}^{\vee}=h_{\infty}^{-1}.\Theta_{n}(g_{\infty}).\qedhere\]
\end{proof}

\subsection{Global theta lifts}\label{section definition of Global theta lift}
Let $\mathbf{G}\times\mathbf{H}$ be one of the two reductive dual pairs given in \Cref{section dual pair},
\emph{i.e.}\,$\mathbf{G}\times\mathbf{H}=\grpF\times\pgl_{2}$ or $\spin_{9}\times\sorth_{2,2}$.
\begin{defi}\label{def global theta lift}
    For $\varphi\in\mathcal{A}(\mathbf{H})$ 
    and $\phi\in\Pi_{\min}$,
    the \emph{global theta lift of $\varphi$ with respect to $\phi$} is 
    the automorphic form of $\mathbf{G}$
    defined by the following absolutely convergent integral:
    \begin{align*}
        \Theta_{\phi}(\varphi)(g):=\int_{[\mathbf{H}]}\theta(\phi)(gh)\overline{\varphi(h)}dh,\text{ for any }g\in \mathbf{G}(\A).
    \end{align*}
    For a cuspidal automorphic representation $\pi\in\Pi_{\cusp}(\mathbf{H})$,
    its \emph{global theta lift} $\Theta(\pi)$ is the $\mathbf{G}(\A)$-subspace of $\sqint{[\mathbf{G}]}{}$
    generated by 
    $\set{\Theta_{\phi}(\varphi)}{\varphi\in\pi,\,\phi\in\Pi_{\min}}$.
\end{defi}
\begin{rmk}\label{rmk no convergence problem}
    In this paper, we are always in the situation that 
    either $[\mathbf{H}]$ is compact or $\varphi\in\mathcal{A}(\mathbf{H})$ is cuspidal.
    For the second case, the absolute convergence comes from the rapid decay of $\varphi$.
\end{rmk}
We also define the global theta lift of a ``vector-valued automorphic form'' $\alpha\in\mathcal{A}_{\vrep{n\varpi}}(\grpF)$ defined as \eqref{eqn definition vector valued automorphic form},
which is compatible with \Cref{def global theta lift}:
\begin{defi}\label{def global theta lift for alg modular form}
    For a function $\alpha:\grpF(\Q)\backslash\grpF(\A)/\mathcal{F}_{4,\mathrm{I}}(\widehat{\Z})\rightarrow \vrep{n\varpi_{4}}$ in $\mathcal{A}_{\vrep{n\varpi_{4}}}(\grpF)$,
    its \emph{global theta lift} $\Theta(\alpha)$ is defined as:
    \begin{align}\label{eqn global theta lift of algebraic modular form}
        \Theta(\alpha)(g)=\int_{[\grpF]}\{\Theta_{n}(gh),\alpha(h)\}dh,\,\text{for any }g\in \pgl_{2}(\A),
    \end{align}
    where $\{-,-\}:\jord_{\C}^{\otimes n}\times(\jord_{\C}^{\vee})^{\otimes n}\rightarrow \C$ is the pairing defined in \eqref{eqn symmetric pairing},
    and we view $\alpha(h)\in\vrep{n\varpi_{4}}$ as a homogeneous polynomial over $\jord_{\C}$.
\end{defi}
\section{Exceptional theta series}\label{section exceptional theta series on F4}
In this section, 
we compute the Fourier expansion of the theta lift $\Theta(\alpha)$ 
of $\alpha\in \mathcal{A}_{\vrep{n\varpi_{4}}}(\grpF)$,
and prove \Cref{introthm theta series is modular form} in the introduction.
From now on,
we will identify $\alpha$ with its values $\alpha_{\mathrm{I}}\in\vrep{n}(\jord_{\C})^{\Gamma_{\mathrm{I}}},\alpha_{\mathrm{E}}\in\vrep{n}(\jord_{\C})^{\Gamma_{\mathrm{E}}}$ at $1,\gamma_{\mathrm{E}}$ 
as in \Cref{lemma identification of vector valued automorphic form}.

\subsection{Fourier expansions of global theta lifts}\label{section Fourier expansion theta lift}
Normalize the Haar measure $dh$ of $\grpF(\A)$ in \eqref{eqn global theta lift of algebraic modular form}
so that
$\grpF(\R)\mathcal{F}_{4,\mathrm{I}}(\widehat{\Z})$ has measure $1$.
Write $g\in\pgl_{2}(\A)$ as $g=g_{\Q}g_{\infty}g_{\widehat{\Z}}$,
where $g_{\Q}\in\pgl_{2}(\Q),g_{\widehat{\Z}}\in \pgl_{2}(\widehat{\Z})$ 
and $g_{\infty}$ is the image of an element in $\SL_{2}(\R)$,
then using \Cref{lemma delta element between Albert algebras},
\Cref{lemma F4 action on kernel function} 
and the $\grpF(\R)$-invariance of $\{-,-\}$,
we obtain:

{\footnotesize\begin{align}
    \begin{aligned}\label{eqn simplify def of theta lift}
        \Theta(\alpha)(g)=&\frac{1}{|\Gamma_{\mathrm{I}}|}\int_{\grpF(\R)\mathcal{F}_{4,\mathrm{I}}(\widehat{\Z})}\{\Theta_{n}(gh_{\infty}h_{\widehat{\Z}}),\alpha(h_{\infty}h_{\widehat{\Z}})\}dh+\frac{1}{|\Gamma_{\mathrm{E}}|}\int_{\grpF(\R)\mathcal{F}_{4,\mathrm{I}}(\widehat{\Z})}\{\Theta_{n}(gh_{\infty}\gamma_{\mathrm{E}}h_{\widehat{\Z}}),\alpha(h_{\infty}\gamma_{\mathrm{E}}h_{\widehat{\Z}})dh\}\\
    =&\frac{1}{|\Gamma_{\mathrm{I}}|}\int_{\grpF(\R)\mathcal{F}_{4,\mathrm{I}}(\widehat{\Z})}\{h_{\infty}^{-1}.\Theta_{n}(g_{\infty}),h_{\infty}^{-1}.\alpha_{\mathrm{I}}\}dh+\frac{1}{|\Gamma_{\mathrm{E}}|}\int_{\grpF(\R)\mathcal{F}_{4,\mathrm{I}}(\widehat{\Z})}\{h_{\infty}^{-1}.\Theta_{n}(\delta_{\infty}^{-1}g_{\infty}),h_{\infty}^{-1}.\alpha_{\mathrm{E}}\}\\
    =&\frac{1}{|\Gamma_{\mathrm{I}}|}\{\Theta_{n}(g_{\infty}),\alpha_{\mathrm{I}}\}+\frac{1}{|\Gamma_{\mathrm{E}}|}\{\Theta_{n}(\delta_{\infty}^{-1}g_{\infty}),\alpha_{\mathrm{E}}\}.
    \end{aligned}
\end{align}}

If the global theta lift $\Theta(\alpha)\in\mathcal{A}(\mathbf{PGL}_{2})$ is non-zero,
then the following result shows that it arises from a weight $2n+12$
classical holomorphic modular form on $\SL_{2}(\Z)$:
\begin{prop}\label{prop theta lift comes from modular form}
    Let $\mathcal{H}\subseteq \C$ be the Poincar\'e half plane,
    and $j:\SL_{2}(\R)\times\mathcal{H}\rightarrow\C^{\times}$
    the automorphy factor given by $j\left(\smallmat{a}{b}{c}{d},z\right)=cz+d$.
    For any $\alpha\in\mathcal{A}_{\vrep{n\varpi_{4}}}(\grpF)$,
    the function 
    \[f_{\Theta(\alpha)}(z):=j(g,i)^{2n+12}\Theta(\alpha)(g),\,z=g.i\in\mathcal{H},\,g\in\SL_{2}(\R),\]
    is well-defined and is a level one holomorphic modular form of weight $2n+12$.
    Furthermore, it is a cusp form when $n>0$.
\end{prop}
We postpone the proof of \Cref{prop theta lift comes from modular form} to \Cref{section proof of main computation},
and prove the following main theorem on the Fourier expansion of $f_{\Theta(\alpha)}$: 
\begin{thm}\label{thm Fourier of theta lift}
    (\Cref{introthm theta series is modular form} in \Cref{section introduction})
    Let $\alpha\in \mathcal{A}_{\vrep{n\varpi_{4}}}(\grpF),\,n>0$ and 
    $f_{\Theta(\alpha)}$ the cusp form associated to its global theta lift $\Theta(\alpha)$.
    Up to a non-zero constant, 
    $f_{\Theta(\alpha)}$ has the following Fourier expansion:
    \begin{align*}
        f_{\Theta(\alpha)}(z)=\frac{1}{|\Gamma_{\mathrm{I}}|}\sum_{T\in\jord_{\Z}^{+}}\mathrm{a}_{\jord_{\Z}}(T)\alpha_{\mathrm{I}}(T)q^{\tr(T)}
        +\frac{1}{|\Gamma_{\mathrm{E}}|}\sum_{T\in\jord_{\mathrm{E}}^{+}}\mathrm{a}_{\jord_{\mathrm{E}}}(T)\alpha_{\mathrm{E}}(T)q^{\tr(T)},\,q=e^{2\pi i z}.
    \end{align*}
\end{thm}
\begin{rmk}\label{rmk review trivial weight}
    The case when $n=0$ is studied by Elkies and Gross in \cite{Leech}. 
    In this case 
    $\alpha\in\mathcal{A}_{\triv}(\grpF)$ can be identified as a pair of complex numbers.
    For $\alpha$ corresponding to $(|\Gamma_{\mathrm{I}}|,0)$,
    $f_{\Theta(\alpha)}=E_{12}+\frac{432000}{691}\Delta$;
    for $\alpha$ corresponding to $(0,|\Gamma_{\mathrm{E}}|)$,
    $f_{\Theta(\alpha)}=E_{12}-\frac{65520}{691}\Delta$,
    where $E_{12}(z)=1+\frac{2}{\zeta(-11)}\sum_{n\geq 1}\sigma_{11}(n)q^{n}$ is the normalized weight $12$ Eisenstein series,  
    and $\Delta(z)=q\prod_{n\geq 1}(1-q^{n})^{24}$ is the discriminant modular form.
\end{rmk}
Before proving \Cref{thm Fourier of theta lift},
we state a result that will be used in the proof,
whose proof is also postponed to \Cref{section proof of main computation}. 
\begin{thm}\label{thm main theorem in Fourier computation}
    Let $P\in \vrep{n}(\jord_{\C})\simeq \vrep{n\varpi_{4}}$ for any $n>0$, 
    $T$ an element of $\jord_{\R}$,
    and $h_{T}(g)=J(g_{\infty},i\mathrm{I})^{-4}\cdot e^{2\pi i (T,g_{\infty}.i\mathrm{I})}$ the function given in \Cref{notation expansion of derivative Kim form},
    then we have:
    \[\left\{(\mathrm{D}^{n}h_{T})(g),P\right\}=(-4\pi)^{n}\cdot j(g,i)^{-2n-12}P\left(T\right)e^{2\pi i (T,\,g.i\mathrm{I})},\,\text{ for any }g\in\SL_{2}(\R).\]
\end{thm}
\begin{proof}[Proof of \Cref{thm Fourier of theta lift}]
    By \eqref{eqn simplify def of theta lift},
    we have 
    \begin{align}\label{eqn simplifying the modular form}
        f_{\Theta(\alpha)}(z)=j(g,i)^{2n+12}\left(\frac{1}{|\Gamma_{\mathrm{I}}|}\{\Theta_{n}(g),\alpha_{\mathrm{I}}\}+\frac{1}{|\Gamma_{\mathrm{E}}|}\{\Theta_{n}(\delta_{\infty}^{-1}g),\alpha_{\mathrm{E}}\}\right),\,z=g.i\in\mathcal{H}.   
    \end{align}
    Using the Fourier expansion \eqref{eqn Fourier expansion of higher kernel function} of $\Theta_{n}$
    and \Cref{thm main theorem in Fourier computation},
    the first term in \eqref{eqn simplifying the modular form} equals
    \begin{align*}
        \frac{1}{|\Gamma_{\mathrm{I}}|}j(g,i)^{2n+12}\{\Theta_{n}(g),\alpha_{\mathrm{I}}\}&=\frac{240}{|\Gamma_{\mathrm{I}}|}j(g,i)^{2n+12}\sum_{T\in\jord_{\Z}^{+}}\mathrm{a}_{\jord_{\Z}}(T)\{\mathrm{D}^{n}h_{T}(g),\alpha_{\mathrm{I}}\}\\
        &= \frac{240 (-4\pi)^{n}}{|\Gamma_{\mathrm{I}}|}\sum_{T\in\jord_{\Z}^{+}}\mathrm{a}_{\jord_{\Z}}(T)\alpha_{\mathrm{I}}(T)q^{(T,\mathrm{I})_{\mathrm{I}}},
    \end{align*}
    where the constant term $\mathrm{D}^{n}h_{0}(g)$ of $\Theta_{n}$ disappears because by \Cref{thm main theorem in Fourier computation} we have 
    \[\{(\mathrm{D}^{n}h_{0})(g),P\}=(-4\pi)^{n}\cdot  j(g,i)^{-2n-12}P(0)=0,\text{ for any }P\in\vrep{n}(\jord_{\C})\text{ with }n>0.\]
    The second term in \eqref{eqn simplifying the modular form} equals
    \begin{align*}
        \frac{1}{|\Gamma_{\mathrm{E}}|}j(g,i)^{2n+12}\{\Theta_{n}(\delta_{\infty}^{-1}g),\alpha_{\mathrm{E}}\}&=\frac{240}{|\Gamma_{\mathrm{E}}|}j(g,i)^{2n+12}\sum_{T\in\jord_{\Z}^{+}}\mathrm{a}_{\jord_{\Z}}(T)\{\mathrm{D}^{n}h_{T}(\delta_{\infty}^{-1}g),\alpha_{\mathrm{E}}\}\\
        &=\frac{240}{|\Gamma_{\mathrm{E}}|}j(g,i)^{2n+12}\sum_{T\in\jord_{\Z}^{+}}\mathrm{a}_{\jord_{\Z}}(T)\left\{\mathrm{D}^{n}h_{\delta_{\infty}^{*}T}(g),\alpha_{\mathrm{E}}\right\}\\
        &=\frac{240(-4\pi)^{n}}{|\Gamma_{\mathrm{E}}|}\sum_{T\in\jord_{\Z}^{+}}\mathrm{a}_{\jord_{\Z}}(T)\alpha_{\mathrm{E}}(\delta_{\infty}^{*}T)e^{2\pi i\left(\delta_{\infty}^{*}T,g.i\mathrm{I}\right)}.
    \end{align*}
    As the lattices $\jord_{\Z}$ and $\jord_{\mathrm{E}}$ are self-dual with respect to the pairing $(\,,\,)$,
    we have $\delta_{\infty}^{*}(\jord_{\Z})=\delta_{\infty}^{*}(\jord_{\Z}^{\vee})=\jord_{\mathrm{E}}^{\vee}=\jord_{\mathrm{E}}$.
    Since $\mathbf{M}_{\jord}^{1}(\R)$ preserves the rank and stabilizes the set of positive semi-definite elements \cite[Proposition 2.4]{Leech},
    $\delta_{\infty}^{*}$ maps $\jord_{\Z}^{+}$ onto $\jord_{\mathrm{E}}^{+}$,
    thus 
    \[\sum_{T\in\jord_{\Z}^{+}}\mathrm{a}_{\jord_{\Z}}(T)\alpha_{\mathrm{E}}(\delta_{\infty}^{*}T)e^{2\pi i\left(\delta_{\infty}^{*}T,g.i\mathrm{I}\right)}=\sum_{T\in\jord_{\mathrm{E}}^{+}}\mathrm{a}_{\jord_{\mathrm{E}}}(T)\alpha_{\mathrm{E}}(T)e^{2\pi i (T,g.i\mathrm{I})}=\sum_{T\in\jord_{\mathrm{E}}^{+}}\mathrm{a}_{\jord_{\mathrm{E}}}(T)\alpha_{\mathrm{E}}(T)q^{\tr(T)}.\qedhere\]
\end{proof}
A direct corollary of \Cref{thm Fourier of theta lift} is the following:
\begin{cor}\label{Cor Fourier theta lift no average}
    For any Albert lattice $J\in\mathcal{J}$
    and any polynomial $P\in \vrep{n}(\jord_{\C})$,
    the (weighted) theta series 
    \begin{align}\label{eqn weighted theta series for any Albert algebra}
        \vartheta_{J,P}(z):=\sum_{T\in J^{+}}\mathrm{a}_{J}(T)P(T)q^{\tr(T)},\,z\in\mathcal{H},q=e^{2\pi i},         
    \end{align}
    is a modular form on $\SL_{2}(\Z)$ of weight $2n+12$,
    and it is cuspidal if $P$ is not constant.
\end{cor}
\begin{proof}  
    Since the theta series \eqref{eqn weighted theta series for any Albert algebra}
    is invariant under the $\grpF(\R)$-action on the pair $(J,P)$
    in the sense that $\vartheta_{gJ,gP}=\vartheta_{J,P}$,
    it suffices to prove the modularity for $J\in\left\{\jord_{\Z},\jord_{\mathrm{E}}\right\}$.
    Here we give the proof for $J=\jord_{\Z}$, and that for $\jord_{\mathrm{E}}$ is almost the same.

    Let $\alpha:\mathcal{J}\rightarrow \vrep{n}(\jord_{\C})$ be the element in $\mathcal{A}_{\vrep{n}(\jord_{\C})}(\grpF)$
    that is supported on the $\grpF(\R)$-orbit of $\jord_{\Z}$ and takes the value $\sum_{\gamma\in \Gamma_{\mathrm{I}}}\gamma.P$ at $\jord_{\Z}\in\mathcal{J}$.
    By \Cref{thm Fourier of theta lift} and \Cref{rmk review trivial weight},
    $f_{\Theta(\alpha)}$ is a modular form on $\SL_{2}(\Z)$ of weight $2n+12$.    
    On the other hand, 
    $\jord_{\Z}^{+}$ is stable under the action of $\Gamma_{\mathrm{I}}$,
    thus one has:
    \begin{align*}
        f_{\Theta(\alpha)}(z)&=\frac{1}{|\Gamma_{\mathrm{I}}|}\sum_{T\in\jord_{\Z}^{+}}\mathrm{a}_{\jord_{\Z}}(T)\left(\sum_{\gamma\in\Gamma_{\mathrm{I}}}P(\gamma^{-1}T)\right)q^{\tr(T)}\\
        &=\frac{1}{|\Gamma_{\mathrm{I}}|}\sum_{\gamma\in\Gamma_{\mathrm{I}}}\left(\sum_{T\in\jord_{\Z}}\mathrm{a}_{\jord_{\Z}}(\gamma T)P(T)q^{\tr(\gamma T)}\right)\\
        &=\vartheta_{\jord_{\Z},P}(z) \qedhere
    \end{align*}
\end{proof}
If we view $\alpha\in\mathcal{A}_{\vrep{n\varpi_{4}}}(\grpF)$
as a function $\alpha:\mathcal{J}\rightarrow \vrep{n}(\jord_{\C})$,
the modular form $f_{\Theta(\alpha)}$ can be written in the following forms:
\[f_{\Theta(\alpha)}=\frac{1}{|\Gamma_{\mathrm{I}}|}\vartheta_{\jord_{\Z},\alpha(\jord_{\Z})}+\frac{1}{|\Gamma_{\mathrm{E}}|}\vartheta_{\jord_{\mathrm{E}},\alpha(\jord_{\mathrm{E}})}.\]

\subsection{Theta series attached to \texorpdfstring{$\spin_{9}(\R)$}{}-invariant polynomials}\label{section theta series spin invariant polynomial}
As an application of \Cref{thm Fourier of theta lift},
we are going to show that for every weight $k$ with $\mathrm{S}_{k}(\SL_{2}(\Z))\neq 0$,
there exists a polynomial $P\in\vrep{\frac{k-12}{2}}(\jord_{\C})$
such that the weighted theta series $\vartheta_{\jord_{\Z},P}$ defined as \eqref{eqn weighted theta series for any Albert algebra} is non-zero.
This result will be used later in \Cref{section nonvanishing archimedean zeta integral alternative proof}.

The $\lietype{F}{4}\downarrow\lietype{B}{4}$ branching law \cite[\S2,Theorem 7]{Lepowsky1970}
says that $\dim \vrep{n\varpi_{4}}^{\spin_{9}(\R)}=1$ for any $n>0$,
where $\spin_{9}$ is defined as the stabilizer of $\mathrm{E}_{1}=[1,0,0\,;0,0,0]$ in $\grpF$,
thus the $\spin_{9}(\R)$-invariant polynomial in $\vrep{n}(\jord_{\C})$ is unique up to a non-zero scalar.
\begin{thm}\label{thm for each weight a non-zero lift}
    For $n\geq 2$ and any non-zero polynomial $P\in\vrep{n}(\jord_{\C})^{\spin_{9}(\R)}$,
    the weighted theta series $\vartheta_{\jord_{\Z},P}$ is non-zero.   
\end{thm}
\begin{proof}
    We first construct an explicit polynomial $P_{n}\in\vrep{n}(\jord_{\C})^{\spin_{9}(\R)}$.
    In the real definite octonion algebra $\oct_{\R}$, 
    we pick three purely imaginary elements $x_{0},y_{0},z_{0}$
    such that $\R\oplus\R x_{0}\oplus \R y_{0}\oplus \R z_{0}$ is isomorphic to Hamilton's quaternion algebra,
    \emph{i.e.} 
    \[x_{0}^{2}=y_{0}^{2}=z_{0}^{2}=-1\text{ and }x_{0}y_{0}=-y_{0}x_{0}=z_{0}.\]
    Take $x_{1}=x_{0},\,y_{1}=\sqrt{-2}y_{0}$ and $z_{1}=\sqrt{-2}z_{0}$,
    and choose $B=[2,-1,-1\,;x_{1},y_{1},z_{1}]\in \jord_{\C}$.
    It can be easily verified that $B\in\mathbb{X}$,
    thus the polynomial $Q_{n}(X):=\left(\tr(X\circ B)\right)^{n}=\bracket{X,B}^{n}$ lies in $\vrep{n}(\jord_{\C})$,
    and take $P_{n}(X):=\int_{\spin_{9}(\R)}k.Q_{n}(X)dk=\int_{\spin_{9}(\R)}(X,kB)^{n}dk$ to be the average of $Q_{n}$ over $\spin_{9}(\R)$.
    Now it suffices to show that the associated theta series $\vartheta_{\jord_{\Z},P_{n}}\neq 0$.

    Consider the first Fourier coefficient $a_{1}$ of $\vartheta_{\jord_{\Z},P_{n}}$.
    The elements in $\jord_{\Z}^{+}$ having contributions to the coefficient of $q$ are $\mathrm{E}_{1},\,\mathrm{E}_{2}$ and $\mathrm{E}_{3}$,
    thus:
    \begin{align}\label{eqn first Fourier coefficient of weighted theta}
        a_{1}=\sum_{i=1}^{3}P_{n}(\mathrm{E}_{i})=\int_{\spin_{9}(\R)}\left(\sum_{i=1}^{3}(\mathrm{E}_{i},kB)^{n}\right)dk.
    \end{align}
    By \Cref{lemma stable spin9 subspace of Jordan algebra},
    $\spin_{9}(\R)$ preserves the subspaces
    $\jord_{1}=\set{[0,\xi,-\xi\,;x,0,0]}{\xi\in\R,x\in\oct_{\R}}$
    and $\jord_{2}=\set{[0,0,0\,;0,y,z]}{y,z\in\oct_{\R}}$
    respectively.
    So for any $k\in\spin_{9}(\R)$ we set:
    \begin{align*}
        &k[0,0,0\,;x_{1},0,0]=[0,\xi(k),-\xi(k)\,;x(k),0,0]\in\jord_{1},\\
        &k[0,0,0\,;0,y_{1},z_{1}]=[0,0,0\,;0,y(k),z(k)]\in\jord_{2}\otimes\C.
    \end{align*}
    We have the equality $2\xi(k)^{2}+\lrangle{x(k)}{x(k)}=\lrangle{x_{1}}{x_{1}}=2$,
    as $k$ preserves the inner product on $\jord_{\R}$, which implies that $|\xi(k)|\leq 1$.
    The three diagonal entries of $kB$ are
    $2,-1+\xi(k)$ and $-1-\xi(k)$,
    thus $\sum\limits_{i=1}^{3}(\mathrm{E}_{i},kB)^{n}=2^{n}+(-1+\xi(k))^{n}+(-1-\xi(k))^{n}\in\R_{\geq 0}$.
    When we take $k=1$, $\sum\limits_{i=1}^{3}\bracket{\mathrm{E}_{i},B}^{n}=2^{n}+(-1)^{n}+(-1)^{n}$
    is positive for any $n\geq 2$.
    Hence the integral in \eqref{eqn first Fourier coefficient of weighted theta}
    is strictly positive,
    and as a consequence the weighted theta series $\vartheta_{\jord_{\Z},P_{n}}$ is non-zero.
\end{proof}

\subsection{Proof of \texorpdfstring{\Cref{thm main theorem in Fourier computation}}{}}\label{section proof of main computation}
 In this section, we will prove \Cref{prop theta lift comes from modular form} and \Cref{thm main theorem in Fourier computation},
following a similar strategy to that of Pollack in \cite[\S 6]{pollack2023exceptional}.

We first define a basis $\{\mathrm{X}_{\alpha}\}_{\alpha}$ of $\mathfrak{p}_{\jord}^{+}$ as follows:
for any $A\in\jord_{\C}$, write
$\mathrm{X}_{A}:=\mathrm{X}_{A}^{+}=i\mathrm{C}_{h}^{-1}\mathrm{n}_{\mathrm{L}}(A)\mathrm{C}_{h}
$
as in \Cref{section Lie algebra E7},
which is an element of $\mathfrak{p}_{\jord}^{+}$ by \Cref{prop cayley transform switch two decompositions}.
Choose a $\C$-basis $\{e_{1},\ldots,e_{27}\}$ of $\jord_{\C}$, 
then we have a basis $\{\mathrm{X}_{e_{i}}\}_{1\leq i\leq 27}$ of $\mathfrak{p}^{+}_{\jord}$,
and we denote its dual basis by $\{\mathrm{X}_{e_{i}}^{\vee}\}_{1\leq i\leq 27}$.
In \cite[\S 6.2]{pollack2023exceptional},
Pollack calculates the action of $\mathrm{X}_{A_{n}}\cdots\mathrm{X}_{A_{1}}$ on $h_{T}|_{\mathbf{M}_{\jord}(\R)}$.
Before recalling his result, we explain some notations that will appear in the statement.

Let $\mathrm{T}(\jord_{\C})=\bigoplus\limits_{k=0}^{\infty}\jord_{\C}^{\otimes k}$
be the tensor algebra of $\jord_{\C}$.
Define a family of $\grpF(\R)$-equivariant maps $\mathscr{P}_{k}:\jord_{\C}^{\otimes k}\rightarrow \mathrm{T}(\jord_{\C})$ inductively:
\begin{itemize}
    \item let $\mathscr{P}_{0}=1$ be the constant map;
    \item for $k\geq 0$, define 
    \footnote{In \cite[\S 6.2]{pollack2023exceptional}, the Jordan product $A\circ B$ is denoted by $\frac{1}{2}\{A,B\}$,
    where $\{A,B\}=AB+BA$ is defined in \cite[\S 3.3.1]{PolFourierExpansion}.}
    {\small\begin{align*}
        \mathscr{P}_{k+1}(A_{1}\otimes\cdots\otimes A_{k}\otimes A_{k+1})=&\mathscr{P}_{k}(A_{1}\otimes\cdots\otimes A_{k})\otimes A_{k+1}+4\tr(A_{k+1})\mathscr{P}_{k}(A_{1}\otimes\cdots\otimes A_{k})\\
        &+A_{k+1}\circ\mathscr{P}_{k}(A_{1}\otimes\cdots\otimes A_{k})+\mathscr{P}_{k}(A_{k+1}\circ(A_{1}\otimes\cdots \otimes A_{k})),   
    \end{align*}}
    where $A\circ (A_{1}\otimes\cdots\otimes A_{r}):=\sum_{j=1}^{r}A_{1}\otimes\cdots\otimes (A\circ A_{j})\otimes\cdots\otimes A_{r}$.
\end{itemize}
For any $T\in \jord_{\R}$ and $m\in \LeviM_{\jord}(\R)$, 
we define a linear form $w_{T,m}$
on $\mathrm{T}(\jord_{\C})$ by:
\begin{align*}
    w_{T,m}(A_{1}\otimes\cdots\otimes A_{r})=(-4\pi)^{r}\prod_{j=1}^{r}\left(T,m(A_{j})\right),\text{ for any }r\geq 0.
\end{align*}
\begin{prop}\label{prop action iterated X}
    \cite[Proposition 6.2.2]{pollack2023exceptional}
    Let the notations be as above,
    then for any $m\in\LeviM_{\jord}(\R)$ and $A_{1},\ldots,A_{n}\in\jord_{\C}$,
    we have
    \begin{align*}
        \mathrm{X}_{A_{n}}\cdots \mathrm{X}_{A_{1}}h_{T}(m)=w_{T,\lambda(m)m^{*}}(\mathscr{P}_{n}(A_{1}\otimes\cdots\otimes A_{n}))h_{T}(m).
    \end{align*}
\end{prop}
\begin{rmk}\label{rmk small mistake}
    There is an abuse of notation in the statement of \cite[Proposition 6.2.2]{pollack2023exceptional}, 
    where the formula is (imported into our notations):
    \begin{align}\label{eqn Pollack orginial formula}
        \mathrm{X}_{A_{n}}\cdots\mathrm{X}_{A_{1}}h_{T}(m)=w_{T,m}\left(\mathscr{P}_{n}(A_{1}\otimes\cdots\otimes A_{n})\right)h_{T}(m).   
    \end{align}
    In \cite{pollack2023exceptional}, Pollack identifies the Levi subgroup $\mathbf{M}_{\jord}$ of $\mathbf{P}_{\jord}$ with the group consisting of $M(\delta,m)$,
    where $\delta\in \mathbf{GL}_{1}$,
    $m$ is a linear automorphism of $\jord_{\Q}$ that preserve the cubic norm up to scaling,
    and $M(\delta,m)$ acts on $\mathrm{W}_{\jord}$ as 
    \[M(\delta,m)(X,\xi,X^{\prime},\xi^{\prime})=(m^{*}X,\lambda(m)\xi,mX^{\prime},\lambda(m)^{-1}\xi^{\prime}).\]
    The $h_{T}(m)$ in \Cref{eqn Pollack orginial formula} is a shorthand for $h_{T}(M(\delta,m))$. 
\end{rmk}
Observe that $\mathscr{P}_{n}(A_{1}\otimes\cdots\otimes A_{n})$ is the sum of 
$A_{1}\otimes\cdots\otimes A_{n}$ with tensors of smaller degrees.
The following lemma enables us to consider only the leading term of $\mathscr{P}_{n}$.
\begin{lemma}\label{lemma only leading term matters}
    Let $P$ be an element in $\vrep{n}(\jord_{\C})\simeq\vrep{n\varpi_{4}}$, then:
    {\footnotesize\begin{align}\label{eqn only leading term matters}
        \sum_{i_{1},\ldots,i_{n}}\mathscr{P}_{n}(e_{i_{1}}\otimes\cdots\otimes e_{i_{n}})\{\mathrm{X}_{e_{i_{1}}}^{\vee}\otimes\cdots\otimes \mathrm{X}_{e_{i_{n}}}^{\vee},P\}=
        \sum_{i_{1},\ldots,i_{n}}e_{i_{1}}\otimes\cdots\otimes e_{i_{n}}\{\mathrm{X}_{e_{i_{1}}}^{\vee}\otimes\cdots\otimes \mathrm{X}_{e_{i_{n}}}^{\vee},P\}.
    \end{align}}
\end{lemma}
\begin{proof}
    Since the pairing $\{-,-\}$ is $\grpF(\R)$-invariant and $\mathscr{P}_{n}$ is $\grpF(\R)$-equivariant,
    for any $g\in \grpF(\R)$, we have:
    \begin{align*}
        &\sum_{i_{1},\ldots,i_{n}}\mathscr{P}_{n}(e_{i_{1}}\otimes\cdots\otimes e_{i_{n}})\{\mathrm{X}_{e_{i_{1}}}^{\vee}\otimes\cdots\otimes \mathrm{X}_{e_{i_{n}}}^{\vee},g.P\}\\
        =&\sum_{i_{1},\ldots,i_{n}}\mathscr{P}_{n}(e_{i_{1}}\otimes\cdots\otimes e_{i_{n}})\{\mathrm{X}_{g^{-1}.e_{i_{1}}}^{\vee}\otimes\cdots\otimes \mathrm{X}_{g^{-1}.e_{i_{n}}}^{\vee},P\}\\
        =&\sum_{i_{1},\ldots,i_{n}}\mathscr{P}_{n}(g.e_{i_{1}}\otimes\cdots\otimes g.e_{i_{n}})\{\mathrm{X}_{e_{i_{1}}}^{\vee}\otimes\cdots\otimes \mathrm{X}_{e_{i_{n}}}^{\vee},P\}\\
        =&\sum_{i_{1},\ldots,i_{n}}g.\mathscr{P}_{n}(e_{i_{1}}\otimes\cdots\otimes e_{i_{n}})\{\mathrm{X}_{e_{i_{1}}}^{\vee}\otimes\cdots\otimes \mathrm{X}_{e_{i_{n}}}^{\vee},P\}.
    \end{align*}
    Comparing this with the right-hand side of \eqref{eqn only leading term matters},
    it suffices to prove \eqref{eqn only leading term matters} for one non-zero vector in $\vrep{n\varpi_{4}}$,
    so we take $P$ to be $\left(\tr(X\circ A)\right)^{n}\in \vrep{n}(\jord_{\C})$ for an arbitrary $A\in \mathbb{X}$,
    as explained in \Cref{section harmonic polynomials representations of F4}.

    Both sides of \eqref{eqn only leading term matters} are independent of the choice of the basis $\{e_{i}\}_{1\leq i\leq 27}$ of $\jord_{\C}$,
    thus we choose a specific basis $\{e_{i}\}_{1\leq i\leq 27}$ such that $e_{1}=A$.
    With this choice,
    it suffices to prove $\mathscr{P}_{n}(e_{1}^{\otimes n})=e_{1}^{\otimes n}$,
    which follows from the inductive definition of $\mathscr{P}_{n}$ and the fact that $\tr(e_{1})=0,\,e_{1}\circ e_{1}=0$.
\end{proof}
\begin{prop}\label{prop pairing with harmonic polynomial}
    For $m\in\LeviM_{\jord}(\R)$ and $P\in\vrep{n}(\jord_{\C})\simeq\vrep{n\varpi_{4}}$,
    we have 
    \begin{align*}
        \{\mathrm{D}^{n}h_{T}(m),P\}=(-4\pi)^{n}P\left(\lambda(m)m^{-1}T\right)h_{T}(m).
    \end{align*}
\end{prop}
\begin{proof}
    Combining \Cref{prop action iterated X} and \Cref{lemma only leading term matters} together,
    we have:
    \begin{align*}
        \{\mathrm{D}^{n}h_{T}(m),P\}&=\sum_{i_{1},\ldots,i_{n}}\mathrm{X}_{e_{i_{n}}}\cdots \mathrm{X}_{e_{i_{1}}}h_{T}(m)\left\{\mathrm{X}_{e_{i_{1}}}^{\vee}\otimes\cdots\otimes \mathrm{X}_{e_{i_{n}}}^{\vee},P\right\}\\
        &=\sum_{i_{1},\ldots,i_{n}}w_{T,\lambda(m)m^{*}}(\mathscr{P}_{n}(e_{i_{1}}\otimes\cdots\otimes e_{i_{n}}))h_{T}(m)\left\{\mathrm{X}_{e_{i_{1}}}^{\vee}\otimes\cdots\otimes \mathrm{X}_{e_{i_{n}}}^{\vee},P\right\}\\
        &=h_{T}(m)\sum_{i_{1},\ldots,i_{n}}w_{T,\lambda(m)m^{*}}(e_{i_{1}}\otimes\cdots\otimes e_{i_{n}})\left\{\mathrm{X}_{e_{i_{1}}}^{\vee}\otimes\cdots\otimes\mathrm{X}_{e_{i_{n}}}^{\vee},P\right\}\\
        &=(-4\pi)^{n}h_{T}(m)\sum_{i_{1},\ldots,i_{n}}\left(\prod_{j=1}^{n}(T,\lambda(m)m^{*}(e_{i_{j}}))\right)\left\{\mathrm{X}_{e_{i_{1}}}^{\vee}\otimes\cdots\otimes\mathrm{X}_{e_{i_{n}}}^{\vee},P\right\}\\
        &=(-4\pi)^{n}h_{T}(m)\sum_{i_{1},\ldots,i_{n}}\left(\prod_{j=1}^{n}\left(\lambda(m)m^{-1}T,e_{i_{j}}\right)\right)\left\{\mathrm{X}_{e_{i_{1}}}^{\vee}\otimes\cdots\otimes\mathrm{X}_{e_{i_{n}}}^{\vee},P\right\}\\
        &=(-4\pi)^{n}h_{T}(m)\left\{\left(\lambda(m)m^{-1}T\right)^{\otimes n},P\right\}\\
        &=(-4\pi)^{n}P\left(\lambda(m)m^{-1}T\right)h_{T}(m).\qedhere
    \end{align*}
\end{proof}
To prove \Cref{thm main theorem in Fourier computation},
we use the Iwasawa decomposition 
to write $g\in\SL_{2}(\R)$ as: 
\[g=tnk,\text{ where }t=\smallmat{u}{}{}{u^{-1}},\,n=\smallmat{1}{x}{}{1},\,k=\smallmat{\cos\theta}{\sin\theta}{-\sin\theta}{\cos\theta}.\]
By a direct calculation, we have the following:
\begin{lemma}\label{lemma compatibility with Lie group elements}
    For $A_{1},\ldots,A_{n}\in \jord_{\C}$, we have the following identities:
    \begin{enumerate}[label=(\arabic*)]
        \item $\mathrm{X}_{A_{n}}\cdots\mathrm{X}_{A_{1}}h_{T}(m\mathrm{n}(A))=e^{2\pi i (T,\lambda(m)m^{*}A)}\mathrm{X}_{A_{n}}\cdots \mathrm{X}_{A_{1}}h_{T}(m),\,\forall A\in \jord_{\C}, m\in \LeviM_{\jord}(\R)$;
        \item $\mathrm{X}_{A_{n}}\cdots \mathrm{X}_{A_{1}}h_{T}(gk)=J(k,i\mathrm{I})^{-4}\left(k.\mathrm{X}_{A_{n}}\right)\cdots \left(k.\mathrm{X}_{A_{1}}\right)h_{T}(g),\,\forall k\in \mathrm{K}_{\lietype{E}{7}}, g\in \mathbf{H}_{\jord}^{1}(\R)$.
    \end{enumerate}
\end{lemma}
\begin{proof}[Proof of \Cref{thm main theorem in Fourier computation}]
    Let the notations be as above.
    By \Cref{lemma compatibility with Lie group elements},
    we have:
    \begin{align*}
        \mathrm{D}^{n}h_{T}(g)&=\mathrm{D}^{n}h_{T}(tnk)\\
        &=J(k,i\mathrm{I})^{-4}\sum_{i_{1},\ldots,i_{n}}(k.\mathrm{X}_{e_{i_{1}}}\cdots k.\mathrm{X}_{e_{i_{n}}})h_{T}(tn)\otimes \mathrm{X}_{e_{i_{1}}}^{\vee}\otimes\cdots\otimes\mathrm{X}_{e_{i_{n}}}^{\vee}\\
        &=J(k,i\mathrm{I})^{-4}e^{2\pi i(T,u^{2}x\mathrm{I})}\sum_{i_{1},\ldots,i_{n}}(k.\mathrm{X}_{e_{i_{1}}}\cdots k.\mathrm{X}_{e_{i_{n}}})h_{T}(t)\otimes \mathrm{X}_{e_{i_{1}}}^{\vee}\otimes\cdots\otimes\mathrm{X}_{e_{i_{n}}}^{\vee}.\\
        &=j(k,i)^{-2n-12}e^{2\pi i (T,u^{2}x\mathrm{I})}\cdot\mathrm{D}^{n}h_{T}(t),
    \end{align*}
    where the last equality follows from $k.\mathrm{X}_{A}=(\cos\theta+i\sin\theta)^{2}\mathrm{X}_{A}=j(k,i)^{-2}\mathrm{X}_{A}$.
    Now we take the pairing of $\mathrm{D}^{n}h_{T}(g)$ with $P$,
    and use \Cref{prop pairing with harmonic polynomial} to obtain the desired identity:
    \begin{align*}
        \{\mathrm{D}^{n}h_{T}(g),P\}&=j(k,i)^{-2n-12}e^{2\pi i (T,u^{2}x\mathrm{I})}(-4\pi)^{n}P\left(u^{2}T\right)J(t,i\mathrm{I})^{-4}e^{2\pi i (T,t.i\mathrm{I})}\\
        &=(-4\pi)^{n}j(k,i)^{-2n-12}j(t,i)^{-12}u^{2n}P(T)e^{2\pi i (T,t.(i\mathrm{I}+x\mathrm{I}))}\\
        &=(-4\pi)^{n}j(g,i)^{-2n-12}P\left(T\right)e^{2\pi i (T,g.i\mathrm{I})}.
        \qedhere
    \end{align*} 
\end{proof}
\begin{proof}[Proof of \Cref{prop theta lift comes from modular form}]
    To show that $f_{\Theta(\alpha)}(z):=j(g,i)^{2n+12}\Theta(\alpha)(g)$ is well-defined, 
    it suffices to verify that for $k$ in the maximal compact subgroup of $\SL_{2}(\R)$, we have:
    \[\Theta(\alpha)(gk)=j(k,i)^{-2n-12}\Theta(\alpha)(g),\text{ for any }g\in\SL_{2}(\R).\]
    This follows from \Cref{lemma compatibility with Lie group elements}
    and the identity $k.\mathrm{X}_{A}=j(k,i)^{-2}\cdot\mathrm{X}_{A}$.
    By the definition of $\Theta(\alpha)$
    and \Cref{prop real theta between F4 and PGL2},
    $f_{\Theta(\alpha)}$ is a level one holomorphic modular form with weight $2n+12$,
    and when $n>0$ it is a cusp form.
\end{proof}

\section{Global theta lifts from \texorpdfstring{$\pgl_{2}$}{PDFstring} to \texorpdfstring{$\grpF$}{PDFstring}}\label{section nonvanishing global theta lift}

We look at the other direction of the global theta correspondence,
\emph{i.e.}\,from $\pgl_{2}$ to $\grpF$.
Let $f$ be a Hecke eigenform in $\mathrm{S}_{2n+12}(\Gamma_{0}(N)),\,n>0$,
and $\pi\simeq \otimes^{\prime}\pi_{v}$ the cuspidal representation of $\pgl_{2}$ associate with $f$.
Take $S$ to be the set of places $\{\infty\}\cup\{p,\,p|N\}$,
so that $\pi_{v}$ is unramified for $v\notin S$.

Our goal is to prove $\Theta(\pi)\neq 0$.
In other words,
we need to find a vector $\phi\in\Pi_{\min}$ and a vector $\varphi\in \pi$
such that $\Theta_{\phi}(\varphi)\neq 0$.
The strategy is to calculate the \emph{$\spin_{9}$-period} of the global theta lift $\Theta_{\phi}(\varphi)$:
\[\mathcal{P}_{\spin_{9}}\left(\Theta_{\phi}(\varphi)\right):=\int_{[\spin_{9}]}\Theta_{\phi}(\varphi)(g)dg.\]
As stated in \Cref{rmk intro connection with SV conj}, one motivation for considering this period integral is the conjecture of Sakellaridis-Venkatesh.

Plugging the definition of the global theta lift $\Theta_{\phi}(\varphi)$ in this period integral
and changing the order of integration,
we obtain:
\begin{align}\label{eqn change the order of integration}
    \begin{aligned}
        \mathcal{P}_{\spin_{9}}(\Theta_{\phi}(\varphi))&=\int_{[\spin_{9}]}\int_{[\pgl_{2}]}\theta(\phi)\left(gh\right)\overline{\varphi(h)}dhdg\\
        &=\int_{[\pgl_{2}]}\overline{\varphi(h)}\left(\int_{[\spin_{9}]}\theta(\phi)\left(gh\right)dg\right)dh.
    \end{aligned}
\end{align}

\subsection{Exceptional Siegel-Weil formula}\label{section exceptional siegel weil formula}
The integral $\int_{[\spin_{9}]}\theta(\phi)(gh)dg$ appearing in \eqref{eqn change the order of integration}, 
as a function of $h\in\sorth_{2,2}(\A)$, 
is the global theta lift of the constant function on $[\spin_{9}]$ to $\sorth_{2,2}$.
In this section,
we will prove an \emph{exceptional Siegel-Weil formula} for $\spin_{9}\times \sorth_{2,2}$,
which represents this theta lift as an Eisenstein series on $\sorth_{2,2}$.
\begin{defi}\label{def principal series of two copies SL2}
    Let $\mathbf{B}=\mathbf{T}\mathbf{N}$ be the Borel subgroup of 
    \[\sorth_{2,2}=\mathbf{GSpin}_{2,2}/\mathbf{G}_{\mathrm{m}}=\set{(g_{1},g_{2})\in\GL_{2}\times\GL_{2}}{\det g_{1}=\det g_{2}}/\mathbf{G}_{\mathrm{m}}^{\Delta}\]
    consisting of the equivalence classes of $(g_{1},g_{2})$,
    where $g_{1}$and $g_{2}$ are upper triangular matrices.
    For $s_{1},s_{2}\in\C$,
    we define a character $\chi_{s_{1},s_{2}}$ on $\mathbf{T}(\A)$ by:
    \[\chi_{s_{1},s_{2}}\left(\left(\begin{smallmatrix}
    a_{1}&\\
    &b_{1}
\end{smallmatrix}\right),\left(\begin{smallmatrix}
    a_{2}&\\
    &b_{2}
\end{smallmatrix}\right)\right):=|a_{1}/b_{1}|^{\frac{s_{1}}{2}}\cdot |a_{2}/b_{2}|^{\frac{s_{2}}{2}},\]
and define $\mathrm{I}(s_{1},s_{2})$ to be the (normalized) degenerate principal series $\mathrm{Ind}_{\mathbf{B}(\A)}^{\sorth_{2,2}(\A)}\chi_{s_{1},s_{2}}$.
\end{defi}
By \Cref{prop minimal rep as subrep of principal series},
we identify the (adelic) minimal representation $\Pi_{\min}$ of $\grpE(\A)$ as a subrepresentation of 
$\mathrm{Ind}_{\para_{\jord}(\A)}^{\grpE(\A)}\delta_{\para_{\jord}}^{-1/2}|\lambda|^{2}$.
\begin{lemma}\label{lemma restriction of degenerate principal series}
    The restriction of sections gives a morphism $\mathrm{Ind}_{\para_{\jord}(\A)}^{\grpE(\A)}\delta_{\para_{\jord}}^{-1/2}|\lambda|^{2}\rightarrow \mathrm{I}(3,7)$.
\end{lemma}
\begin{proof}
    A section $f\in\mathrm{Ind}_{\para_{\jord}(\A)}^{\grpE(\A)}\delta_{\para_{\jord}}^{-1/2}|\lambda|^{2}$
    satisfies the functional equation \eqref{eqn sections of Siegel parabolic}.
    Combining the explicit morphisms \eqref{eqn identification of Levi subgroups} and \eqref{eqn explicit embedding into Spin4},
    the image of $\left(\smallmat{a_{1}}{}{}{b_{1}},\smallmat{a_{2}}{}{}{b_{2}}\right)\in\mathbf{T}(\A)$ in $\mathbf{M}_{\jord}\subseteq \grpE$
    has similitude $(a_{1}/b_{1})\cdot (a_{2}/b_{2})^{2}$,
    thus the restriction of $f$ to $\sorth_{2,2}(\A)$ satisfies:
    \[f(tng)=\chi_{4,8}(t)f(g),\text{ for any }t\in\mathbf{T}(\A),n\in\mathbf{N}(\A),g\in\sorth_{2,2}(\A).\]
    This shows that $f|_{\sorth_{2,2}(\A)}$ is a section of $\mathrm{Ind}_{\mathbf{B}(\A)}^{\sorth_{2,2}(\A)}\delta_{\mathbf{B}}^{-1/2}\chi_{4,8}=\mathrm{I}(3,7)$.
\end{proof}
\Cref{lemma restriction of degenerate principal series} gives us a $\sorth_{2,2}(\A)$-equivariant map: 
\[\mathrm{Res}:\Pi_{\min}\hookrightarrow\mathrm{Ind}_{\para_{\jord}(\A)}^{\grpE(\A)}\delta_{\para_{\jord}}^{-1/2}|\lambda|^{2}\rightarrow\mathrm{I}(3,7).\]
Given a smooth vector $\phi\in\Pi_{\min}$, 
we have the following two automorphic forms on $\sorth_{2,2}$:
\begin{itemize}
    \item The theta integral:
        \[\Theta_{\phi}(1)(g)=\int_{[\spin_{9}]}\theta(\phi)(gh)dh,\text{ for any }g\in \sorth_{2,2}(\A),\]
    \item The Eisenstein series associated to $\widetilde{\phi}:=\mathrm{Res}(\phi)\in\mathrm{I}(3,7)$:
        \[E(\widetilde{\phi})(g):=\sum_{\gamma\in\mathbf{B}(\Q)\backslash\sorth_{2,2}(\Q)}\widetilde{\phi}(\gamma g),\text{ for any }g\in\sorth_{2,2}(\A).\]
\end{itemize}
\begin{thm}\label{thm exceptional siegel weil formula for spin9 so4}
    Let $\Phi_{f}:=\otimes_{p}\Phi_{p}$ be the normalized spherical vector in $\Pi_{\min,f}$ chosen in \Cref{section automorphic realization of min rep},
    then for any smooth holomorphic vector $\phi_{\infty}\in\Pi_{\min,\infty}$, 
    up to some scalar we have:
    \[E(\mathrm{Res}(\phi_{\infty}\otimes \Phi_{f}))=\Theta_{\phi_{\infty}\otimes\Phi_{f}}(1).\]
\end{thm}
Before proving this formula for any smooth vector $\phi_{\infty}\in\Pi_{\min,\infty}$,  
we verify it for the specific vector $\Phi_{\infty}$ chosen in \Cref{section automorphic realization of min rep}.
\begin{prop}\label{prop exceptional siegel weil using Kim modular form}
    For the vector $\Phi_{0}=\Phi_{\infty}\otimes\Phi_{f}\in\Pi_{\min}$, 
    up to some scalar we have:
    \[E(\mathrm{Res}(\Phi_{0}))=\Theta_{\Phi_{0}}(1).\]
\end{prop}
\begin{proof}
By the choice of $\Phi_{0}$,
$\mathrm{Res}(\Phi_{0})_{p}$ 
is the normalized spherical vector of $\mathrm{I}(3,7)_{p}$ for each prime $p$,
and $\mathrm{Res}(\Phi_{0})_{\infty}$ is the unique holomorphic vector in $\mathrm{I}(3,7)_{\infty}$ with minimal $\mathrm{K}_{\lietype{E}{7}}\cap \spin_{2,2}(\R)$-type.
As a result, 
the Eisenstein series $E(\mathrm{Res}(\Phi_{0}))$ is a non-zero multiple of the automorphic form associated to $E_{4}\boxtimes E_{8}$,
where $E_{k}$ is the normalized holomorphic Eisenstein series in $\mathrm{M}_{k}(\SL_{2}(\Z))$. 

On the other side,
the global theta lift is a non-zero multiple of 
\[(g_{1},g_{2})\in\sorth_{2,2}(\A)\mapsto j(g_{1,\infty})^{-4}j(g_{2,\infty})^{-8}\mathrm{F}_{Kim}\left(\mathrm{diag}(g_{1,\infty}.i,g_{2,\infty}.i,g_{2,\infty}.i)\right),\]
where $(g_{1,\infty},g_{2,\infty})\in\spin_{2,2}(\R)$ is the archimedean component of $(g_{1},g_{2})$ (up to some left translation by $\sorth_{2,2}(\Q)$).
It suffices to show that $\mathrm{F}_{Kim}(\mathrm{diag}(z_{1},z_{2},z_{2}))$,
as a function on $\mathcal{H}\times\mathcal{H}$,
is a non-zero multiple of $E_{4}(z_{1})E_{8}(z_{2})$.
Since the diagonal entries of elements in $\mathcal{H}_{\jord}$ are symmetric,
the restriction of $\mathrm{F}_{Kim}$ to $\mathrm{diag}(z_{1},z_{2},z_{3})$
gives a weight $4$ modular form in each variable $z_{i}\in\mathcal{H}$,\,$i=1,2,3$.
As a consequence, $\mathrm{F}_{Kim}(\mathrm{diag}(z_{1},z_{2},z_{2}))$ is a modular form of weight $4$ in $z_{1}$ and of weight $8$ in $z_{2}$.
The space of modular forms $\mathrm{M}_{k}(\SL_{2}(\Z)),\,k=4$ or $8$, 
is $1$-dimensional and spanned by $E_{k}$,
thus $\mathrm{F}_{Kim}(\mathrm{diag}(z_{1},z_{2},z_{2}))$
equals $E_{4}(z_{1})E_{8}(z_{2})$ up to scaling.
%
\end{proof}
\begin{proof}[Proof of \Cref{thm exceptional siegel weil formula for spin9 so4}]
    By \Cref{prop archimedean siegel weil formula},
    we write $\Pi^{+}$
    as $(\Pi^{+})^{\mathbf{Spin}_{9}(\R)}\oplus \Pi^{+}_{0}$,
    where \[(\Pi^{+})^{\mathbf{Spin}_{9}(\R)}=\Pi^{+}\cap \left(\triv\boxtimes \mathcal{D}(4)\boxtimes \mathcal{D}(8)\right)\] 
    and \[\Pi^{+}_{0}=\Pi^{+}\cap\left(\bigoplus_{\substack{m,n\geq 0\\ (m,n)\neq (0,0)}}\mathrm{U}_{m,n}\boxtimes\mathcal{D}(n+4)\boxtimes \mathcal{D}(2m+n+8)\right).\]
    
    Suppose that a smooth vector $\phi_{\infty}$ lies in $\Pi_{0}^{+}$.
    For $(m,n)\neq (0,0)$,
    there is no non-zero $\mathbf{Spin}_{9}(\R)$-equivariant map from $U_{m,n}$ to $\triv$,
    and no non-zero $\sorth_{2,2}(\R)$-equivariant map from $\mathcal{D}(n+4)\boxtimes \mathcal{D}(2m+n+8)$ to $\mathrm{I}(3,7)$.
    Hence both $\Theta_{\phi_{\infty}\otimes \Phi_{f}}(1)$ and $\mathrm{Res}(\phi_{\infty}\otimes \Phi_{f})$ vanish.

    Now we assume that the smooth vector $\phi_{\infty}\in(\Pi^{+})^{\spin_{9}(\R)}$,
    and it lies in the $\spin_{2,2}(\R)$-orbit of $\Phi_{\infty}$,
    then the theorem follows from \Cref{prop exceptional siegel weil using Kim modular form}
    and the fact that the maps 
    $E(\mathrm{Res}(-))$ and $\Theta_{-}(1)$
    are both $\sorth_{2,2}(\A)$-equivariant.
\end{proof}
Denote by $\Pi_{\min}^{\infty}$ the smooth vectors of $\Pi_{\min}$,
and let $\mathrm{SW}:\Pi_{\min}^{\infty}\rightarrow \mathcal{A}(\sorth_{2,2})$ be the map given by 
taking the $\spin_{9}$-period of the theta function $\theta(\phi)$.
\Cref{prop exceptional siegel weil using Kim modular form} tells us that $\mathrm{SW}$ maps the spherical vector $\Phi_{0}$ to the Eisenstein series $E(\mathrm{Res}(\Phi_{0}))$.
Since $\mathrm{SW}$ is $\sorth_{2,2}(\A)$-equivariant,
the image of $\mathrm{SW}$ contains the (smooth) automorphic representation generated by the Eisenstein series $E(\mathrm{Res}(\Phi_{0}))$,
denoted by $\pi_{\mathrm{Eis}}$.

\subsection{Rankin-Selberg integral}
\label{section unfolding of rankin selberg integral}
For any smooth automorphic form $\varphi_{E}$ in $\pi_{\mathrm{Eis}}\subseteq \mathcal{A}(\sorth_{2,2})$,
the Eisenstein series automorphic representation associated with $\mathrm{I}(3,7)$,
by the discussion in the end of \Cref{section exceptional siegel weil formula},
there is a smooth vector $\phi_{E}\in \Pi_{\min}$ 
such that 
\[\varphi_{E}(g)=\mathrm{SW}(\phi_{E})(g)=\Theta_{\phi_{E}}(1)(g)=\int_{[\spin_{9}]}\theta(\phi_{E})(gh)dh,\text{ for any }g\in\sorth_{2,2}(\A).\]
We put this vector $\phi_{E}$ in the period integral \eqref{eqn change the order of integration} and obtain:
\begin{align*}
    \mathcal{P}_{\spin_{9}}(\Theta_{\phi_{E}}(\varphi))=\int_{[\pgl_{2}]}\overline{\varphi(h)}\varphi_{E}(h^{\Delta})dh,
\end{align*}
where $h^{\Delta}$ denotes the image of $h\in\pgl_{2}(\A)$ under $\pgl_{2}(\A)\rightarrow\sorth_{2,2}(\A)$.

The representation $\pi_{\mathrm{Eis}}$ can be identified with $\pi_{\mathrm{Eis},1}\boxtimes\pi_{\mathrm{Eis},2}$,
where $\pi_{\mathrm{Eis},1}$ (\emph{resp.\,}$\pi_{\mathrm{Eis},2}$)
is the automorphic representation on $\GL_{2}$ with trivial central character generated by the holomorphic Eisenstein series $E_{4}$ (\emph{resp.\,}$E_{8}$).
We take $\varphi_{E}$ to be $\varphi_{E,1}\boxtimes\varphi_{E,2}$,
and the period integral becomes:
\begin{align}\label{eqn Rankin-Selberg integral GL2}
    \mathcal{P}_{\spin_{9}}(\Theta_{\phi_{E}}(\varphi))=\int_{[\pgl_{2}]}\overline{\varphi(h)}\varphi_{E,1}(h)\varphi_{E,2}(h)dh,
\end{align}
which is the Rankin-Selberg integral for $\pi\times \pi_{\mathrm{Eis},1}$ on $\GL_{2}\times\GL_{2}$.

Fix a non-trivial continuous unitary character $\psi=\otimes_{v}\psi_{v}$ of $\mathbb{Q}\backslash \A$ such that 
the conductor of $\psi_{p}$ is $\Z_{p}$ for any prime $p\notin S$.
Let $\mathbf{B}_{0}=\mathbf{T}_{0}\mathbf{N}_{0}$ be the standard Borel subgroup of $\pgl_{2}$.
Define $W(g):=\int_{[\mathbf{N_{0}}]}\varphi(ng)\psi^{-1}(n)dn,\,g\in \pgl_{2}(\A)$ 
to be the Whittaker function of $\varphi$,
and $W^{\prime}$ to be the Whittaker function of $\varphi_{E,1}$.
For $\varphi_{E,2}$ in the second Eisenstein series $\pi_{\mathrm{Eis,2}}$,
we write it in the form $\varphi_{E,2}(g)=\sum\limits_{\gamma\in \mathbf{B}_{0}(\Q)\backslash\pgl_{2}(\Q)}f_{E}(\gamma g),\,g\in\pgl_{2}(\A)$
for a section $f_{E}=\otimes_{v}f_{E,v}$ of the corresponding principal series representation.

Before calculating \eqref{eqn Rankin-Selberg integral GL2}, we make the following choice of vectors:
\begin{itemize}
    \item Take $\varphi\in \pi$ (\emph{resp.\,}$\varphi_{E,1}\in \pi_{\mathrm{Eis},1}$) to be a decomposable vector $\varphi\simeq\otimes_{v}\varphi_{v}$ (\emph{resp.\,}$\varphi_{E,1}\simeq\otimes_{v}\varphi_{E,1,v}$)
    so that the global  Whittaker function $W(g)$ (\emph{resp.\,}$W^{\prime}(g)$) factorizes into 
    $\prod_{v}W_{{v}}(g_{v})$ (\emph{resp.\,}$\prod_{v}W^{\prime}_{v}(g_{v})$).
    For $p\notin S$, we choose the vector $\varphi_{p}$ (\emph{resp.\,}$\varphi_{E,1,p}$) such that $W_{p}$ (\emph{resp.\,}$W^{\prime}_{p}$) is the normalized spherical Whittaker function
    in the local Whittaker model for the unramified principal series $\pi_{p}$ (\emph{resp.\,}$\pi_{\mathrm{Eis},1,p}$).
    \item For $p\notin S$, take the section $f_{E,p}$ to be the spherical section such that $f_{E}(1)=1/(1-p^{-8})$.
\end{itemize}

By the classical Rankin-Selberg theory on $\GL_{2}\times \GL_{2}$ \cite[Proposition 3.8.2]{Bump_Automorphic_forms_representations},
the integral \eqref{eqn Rankin-Selberg integral GL2}
is Eulerian, and we have:
\begin{align}\label{eqn decomposition of Rankin-Selberg}
    \int_{[\pgl_{2}]}\overline{\varphi(h)}\varphi_{E,1}(h)\varphi_{E,2}(h)dh=\left(\prod_{v\in S}Z_{v}(4,\overline{W_{v}},W^{\prime}_{v},f_{E,v})\right)\cdot \mathrm{L}^{S}(\pi\times\pi_{\mathrm{Eis,1}},4).
\end{align}
Here we explain the notations appearing in \eqref{eqn decomposition of Rankin-Selberg}:
\begin{itemize}
    \item For $v\in S$, $Z_{v}(s,\overline{W_{v}},W^{\prime}_{v},f_{E,v})$ is the local zeta integral:
        \[\int_{K_{v}}\int_{\mathbf{T}_{0}(\Q_{v})}\overline{W_{v}(tk)}W_{v}^{\prime}(tk)f_{E,v}(tk)\delta_{\mathbf{B}_{0}}(t)^{-1}dtdg,\]
        where $K_{p}=\pgl_{2}(\Z_{p})$ for any finite prime $p$ and $K_{\infty}$ is the image of $\mathrm{O}(2)\subset \GL_{2}(\R)$ in $\pgl_{2}(\R)$.
    \item The partial Rankin-Selberg L-function $\mathrm{L}^{S}(\pi\times\pi_{\mathrm{Eis},1},s)$ in our case is indeed the product of two standard L-functions: 
        \[\mathrm{L}^{S}(\pi\times\pi_{\mathrm{Eis},1},s)=\mathrm{L}^{S}(\pi,s+\frac{3}{2})\cdot\mathrm{L}^{S}(\pi,s-\frac{3}{2}).\]
\end{itemize}
\begin{rmk}\label{rmk remark on Rankin-Selberg in Bump}
    In \cite{Bump_Automorphic_forms_representations}, the Rankin-Selberg integral is calculated for two cuspidal representations on $\GL_{2}$,
    but the computation is also valid for a cuspidal representation with an Eisenstein series.
\end{rmk}
For any place $v\in S$,
Jacquet \cite[Lemma 14.7.5, Proposition 17.4]{Jacquet_Automorphi_forms_GL2_part_2} shows that 
the data $W_{v},W_{v}^{\prime}$ and $f_{E,v}$ may be chosen so that $Z_{v}(4,\overline{W_{v}},W_{v}^{\prime},f_{E,v})=1$.
As a consequence,
we have:
\begin{thm}\label{thm prod local zeta integral L function}
    (\Cref{introthm SV L-factor for Spin9 F4} in \Cref{section introduction})
    For any integer $n\geq 0$, 
    let $\pi$ be the cuspidal representation of $\pgl_{2}$ associated with a Hecke eigenform $f\in \mathrm{S}_{2n+12}(\Gamma_{0}(N))$.
    There is a smooth vector $\phi_{E}$ in $\Pi_{\min}$
    and an automorphic form $\varphi\in \pi\subseteq \mathcal{A}(\pgl_{2})$
    such that the $\spin_{9}$-period integral of the global theta lift $\Theta_{\phi_{E}}(\varphi)$ is the product of two values of the (partial) standard L-function of $\pi$:
    \[\mathcal{P}_{\spin_{9}}(\Theta_{\phi_{E}}(\varphi))=\mathrm{L}^{S}(\pi,\frac{5}{2})\cdot\mathrm{L}^{S}(\pi,\frac{11}{2}),\]
    where $S$ is the set of all the finite primes $p\mid N$.
\end{thm}
Since the standard L-function $\mathrm{L}^{S}(\pi,s)$ has no zero at $s=5/2$ or $11/2$,
the period integral $\mathcal{P}_{\spin_{9}}(\Theta_{\phi_{E}}(\varphi))$ is non-zero for the vectors $\phi_{E}$ and $\varphi$ chosen in \Cref{thm prod local zeta integral L function}.
\begin{cor}\label{cor nonvanishing global theta lift from pgl2 to f4}
    (\Cref{introthm main theorem} in \Cref{section introduction})
    Let $\pi\in\Pi_{\cusp}(\pgl_{2})$ be the automorphic representation
    associated to a Hecke eigenform in $\mathrm{S}_{2n+12}(\Gamma_{0}(N))$, $n\geq 0$,
    then its theta lift $\Theta(\pi)$ to $\grpF$ is non-zero. 
    Furthermore,
    we have the local-global compatibility of theta correspondence,
    \emph{i.e.}
    \[\Theta(\pi)\simeq\otimes_{v}^{\prime}\theta(\pi_{v}).\]
\end{cor}
\begin{proof}
    The non-vanishing comes from \Cref{thm prod local zeta integral L function},
    and the local-global compatibility follows from \Cref{prop p-adic lift to F4}, \Cref{prop p-adic supercuspidal and Steinberg} and \Cref{prop real theta between F4 and PGL2}.
\end{proof}
\subsection{An alternative proof for the non-vanishing of the archimedean zeta integral}
\label{section nonvanishing archimedean zeta integral alternative proof}
In this section,
we give an alternative proof for the non-vanishing of the archimedean zeta integral 
$Z_{\infty}(s=4,\overline{W_{\infty}},W_{\infty}^{\prime},f_{E,\infty})$ for some suitable choice of local data,
via a global argument.

Notice that the local component $\pi_{\infty}$ only depends on the weight of $f$,
thus we assume $f$ to be of level one,
and make the following choice of vectors:
\begin{itemize}
    \item Take $\varphi\simeq \otimes_{v}\varphi_{v}$ such that $\varphi_{p}$ is a spherical vector for every finite prime $p$ and $\varphi_{\infty}$ is a lowest weight vector in the holomorphic discrete series $d_{hol}(2n+12)$. 
    \item Take $\phi_{E}\in \Pi_{\min}$ to be $\phi_{\infty}\otimes \Phi_{f}$, where $\Phi_{f}$ is the spherical vector chosen in \Cref{section automorphic realization of min rep}.
\end{itemize}
With this choice,
up to some non-zero scalar,
from \Cref{section unfolding of rankin selberg integral} we know: 
\[\mathcal{P}_{\spin_{9}}(\Theta_{\phi_{E}}(\varphi))=\mathrm{L}(\pi,\frac{5}{2})\mathrm{L}(\pi,\frac{11}{2})Z_{\infty}(4,\overline{W_{\infty}},W_{\infty}^{\prime},f_{E,\infty}),\]
where $\mathrm{L}(\pi,s)$ is the standard L-function obtained by taking the product of local L-factors at all finite primes $p$.
Hence in order to find $W_{\infty},W_{\infty}^{\prime}$ and $f_{E,\infty}$ such that 
$Z_{\infty}(4,\overline{W_{\infty}},W_{\infty}^{\prime},f_{E,\infty})$ is non-zero,
it suffices to find a choice of $\phi_{\infty}$ and $\varphi$ such that $\mathcal{P}_{\spin_{9}}(\Theta_{\phi_{E}}(\varphi))\neq 0$.

\begin{prop}\label{prop nonvanishing archimedean zeta integral}
    For any natural number $n\neq 1$,
    there exist a cusp form $\varphi_{n}\in\mathcal{A}(\pgl_{2})$
    associated to some Hecke eigenform in $\mathrm{S}_{2n+12}(\SL_{2}(\Z))$, 
    and a smooth vector $\phi_{n}\in\Pi^{+}\subseteq\Pi_{\min,\infty}$,    
    such that 
    $\mathcal{P}_{\spin_{9}}(\Theta_{\phi_{n}\otimes\Phi_{f}}(\varphi_{n}))\neq 0$.
\end{prop}
\begin{proof}
    For each $n>1$,
    \Cref{thm for each weight a non-zero lift} 
    shows that there exists a non-zero $\spin_{9}(\R)$-invariant polynomial $P_{n}$ in $\vrep{n}(\jord_{\C})$
    such that the weighted theta series $\vartheta_{\jord_{\Z},P_{n}}$ defined as \eqref{eqn weighted theta series for any Albert algebra} is non-zero.
    For $n=0$, we take $P_{0}=1$ to be the constant polynomial,
    and the theta series $\vartheta_{\jord_{\Z},P_{0}}\in \mathrm{M}_{12}(\SL_{2}(\Z))$ is a non-zero multiple of $E_{12}+\frac{432000}{691}\Delta$ (see \Cref{rmk review trivial weight}).  

    Let $\alpha_{n}\in \mathcal{A}_{\vrep{n\varpi_{4}}}(\grpF)$ 
    to be the vector-valued automorphic form
    such that $\alpha_{n}(1)=\sum_{\gamma\in\Gamma_{\mathrm{I}}}\gamma.P_{n}\in\vrep{n}(\jord_{\C})^{\Gamma_{\mathrm{I}}}$
    and $\alpha_{n}(\gamma_{\mathrm{E}})=0$,
    then the global theta lift $\Theta(\alpha_{n})$ is a non-zero holomorphic automorphic form of $\pgl_{2}$ (cuspidal when $n>1$).
    There exists an automorphic form $\varphi_{n}\in\mathcal{A}_{\cusp}(\pgl_{2})$ 
    associated to some Hecke eigenform in $\mathrm{S}_{2n+12}(\SL_{2}(\Z))$,
    such that
    the Petersson inner product 
    \begin{align}\label{eqn Petersson inner product of theta lift with Hecke eigenform}
        \int_{[\pgl_{2}]}\Theta(\alpha_{n})(g)\overline{\varphi_{n}(g)}dg
    \end{align}   
    is non-zero.
    Putting the definition of $\Theta(\alpha_{n})$ into \eqref{eqn Petersson inner product of theta lift with Hecke eigenform},
    we have:
    \begin{align}\label{eqn rewrite the Petersson inner product as period of some theta lift}
        0\neq\frac{1}{|\Gamma_{\mathrm{I}}|}\int_{[\pgl_{2}]}\left\{\Theta_{n}(g),\sum_{\gamma\in\Gamma_{\mathrm{I}}}\gamma.P_{n}\right\}\overline{\varphi_{n}(g)}dg
        =\int_{[\pgl_{2}]}\left\{\Theta_{n}(g),P_{n}\right\}\overline{\varphi_{n}(g)}dg.
    \end{align}
    Take the following smooth vector in $\Pi^{+}\subseteq \Pi_{\min,\infty}$:
    \[\phi_{n}:=\{\mathrm{D}^{n}\Phi_{\infty},P_{n}\}=\sum_{\alpha_{1},\ldots,\alpha_{n}}\{\mathrm{X}_{\alpha_{1}}^{\vee}\otimes\cdots\mathrm{X}_{\alpha_{n}}^{\vee},P_{n}\}\cdot\left(\mathrm{X}_{\alpha_{n}}\cdots\mathrm{X}_{\alpha_{1}}.\Phi_{\infty}\right),\]
    where $\Phi_{\infty}$ is the specific vector chosen in \Cref{section automorphic realization of min rep}
    and $\mathrm{D}$ is the operator $\Pi^{+}\rightarrow\Pi^{+}\otimes\mathfrak{p}_{\jord}^{-}$
    sending $\phi$ to $\sum_{\alpha}\mathrm{X}_{\alpha}\phi\otimes\mathrm{X}_{\alpha}^{\vee}$,
    with an arbitrary choice of basis $\{\mathrm{X}_{\alpha}\}$ of $\mathfrak{p}_{\jord}^{+}$ 
    and its dual basis $\{\mathrm{X}_{\alpha}^{\vee}\}$. 
    By \Cref{def automorphic realization theta via Kim modular form},
    the automorphic realization $\theta:\Pi_{\min}\hookrightarrow\mathrm{L}^{2}([\grpE])$
    maps $\phi_{n}\otimes\Phi_{f}$ to 
    \[\theta(\phi_{n}\otimes\Phi_{f})=\left\{\mathrm{D}^{n}\theta(\Phi_{\infty}\otimes\Phi_{f}),P_{n}\right\}=\left\{\mathrm{D}^{n}\Theta_{Kim},P_{n}\right\}=\{\Theta_{n},P_{n}\}.\]
    Use $\theta(\phi_{n}\otimes\Phi_{f})$ as the kernel function to define a global theta lift of $\varphi_{n}$,
    then we calculate the $\spin_{9}$-period integral of this global theta lift:    
    \begin{align*}
        \mathcal{P}_{\spin_{9}}(\Theta_{\phi_{n}\otimes\Phi_{f}}(\varphi_{n}))&=\int_{[\pgl_{2}]\times[\spin_{9}]}\{\Theta_{n}(gh),P_{n}\}\overline{\varphi_{n}(g)}dgdh.
    \end{align*}
    Since we have the strong approximation property $\spin_{9}(\A)=\spin_{9}(\Q)\spin_{9}(\R)\spin_{9}(\widehat{\Z})$,
    the $\spin_{9}$-period integral is a non-zero multiple of 
    \begin{align*}
        \int_{[\pgl_{2}]}\int_{\spin_{9}(\R)}\{\Theta_{n}(gh_{\infty}),P_{n}\}\overline{\varphi_{n}(g)}dgdh_{\infty}&=\int_{[\pgl_{2}]}\int_{\spin_{9}(\R)}\{h_{\infty}^{-1}.\Theta_{n}(g),P_{n}\}\overline{\varphi_{n}(g)}dgdh_{\infty}\\
        &=\int_{\spin_{9}(\R)}dh_{\infty}\cdot\int_{[\pgl_{2}]}\{\Theta_{n}(g),P_{n}\}\overline{\varphi_{n}(g)}dg,
    \end{align*}
    where we use \Cref{lemma F4 action on kernel function} and the $\spin_{9}(\R)$-invariance of $P_{n}$.
    Combining this with \eqref{eqn rewrite the Petersson inner product as period of some theta lift},
    we obtain the non-vanishing of $\mathcal{P}_{\spin_{9}}(\Theta_{\phi_{n}\otimes\Phi_{f}}(\varphi_{n}))$.
\end{proof}
\begin{rmk}
    \label{rmk why not 14}
    In this proof, we exclude the $n=1$ case, because in level one there is no non-zero cusp form of weight $2n+12=14$.
    If a theory of exceptional theta series for higher levels is developed,
    then one will be able to apply the same strategy to the $n=1$ case.
\end{rmk}
This proposition also leads to the following analogue to Waldspurger's result for classical theta series, mentioned in \Cref{section intro exceptional theta series}:
\begin{cor}\label{cor theta series span the whole space}
    (\Cref{introthm theta series generate whole space of cusp form} in \Cref{section introduction})
    For $n\geq 2$,
    the following map is surjective:
    \begin{align*}
        \mathcal{A}_{\vrep{n\varpi_{4}}}(\grpF)&\rightarrow \mathrm{S}_{2n+12}(\SL_{2}(\Z))\\
        \left(\alpha:\mathcal{J}\rightarrow\vrep{n\varpi_{4}}\right)&\mapsto f_{\Theta(\alpha)}=\frac{1}{|\Gamma_{\mathrm{I}}|}\vartheta_{\jord_{\Z},\alpha(\jord_{\Z})}+\frac{1}{|\Gamma_{\mathrm{E}}|}\vartheta_{\jord_{\mathrm{E}},\alpha(\jord_{\mathrm{E}})}
    \end{align*}
\end{cor}
\begin{proof}
    Suppose that the map $\alpha\mapsto f_{\Theta(\alpha)}$ is not surjective,
    then there exists a non-zero Hecke eigenform $f\in\mathrm{S}_{2n+12}(\SL_{2}(\Z))$,
    such that its associated automorphic form $\varphi\in\mathcal{A}(\pgl_{2})$
    is orthogonal to $\Theta(\alpha)$ for all $\alpha\in\mathcal{A}_{\vrep{n\varpi_{4}}}(\grpF)$,
    with respect to the Petersson inner product.
    In particular, $\varphi$ is orthogonal to $\Theta(\alpha_{n})$,
    where $\alpha_{n}$ is the
    algebraic modular form chosen in the proof of \Cref{prop nonvanishing archimedean zeta integral}.
    Take $\phi_{n}\in\Pi_{\min}$ to be the one in \Cref{prop nonvanishing archimedean zeta integral},
    we have:
    \[0=\int_{[\pgl_{2}]}\Theta(\alpha_{n})(g)\overline{\varphi(g)}dg=\mathcal{P}_{\spin_{9}}(\Theta_{\phi_{n}\otimes\Phi_{f}}(\varphi)),\]
    which equals a non-zero multiple of $\mathrm{L}(\pi,\frac{5}{2})\mathrm{L}(\pi,\frac{11}{2})$,
    $\pi$ being the automorphic representation generated by $\varphi$.
    This contradicts with the non-vanishing of L-values.
\end{proof}

\printbibliography

\end{document}